\documentclass[11pt,a4paper]{article}

\usepackage{amsmath,amsthm, amssymb, latexsym}
\usepackage{amsfonts}

\usepackage[T1]{fontenc}
\usepackage{lmodern}
\usepackage{graphicx}
\usepackage{amsfonts}
\usepackage{multicol}
\usepackage{multirow}
\usepackage{caption}
\usepackage{color}
\usepackage{graphicx}
\usepackage{amsmath}
\usepackage{mathrsfs}
\usepackage{multicol}
\usepackage{float}
\usepackage{bbold}
\usepackage{color}
\usepackage{subfigure}
\usepackage{url}
\usepackage{tabularx}
\usepackage{xcolor}
\usepackage{grffile}
\usepackage{cancel}
\usepackage[english]{babel}
\usepackage{fullpage}
\usepackage{enumitem}
\usepackage[hidelinks]{hyperref}
\usepackage{array}
\usepackage{tabularx, caption}
\usepackage{cellspace}
 \normalbaselines
\newtheorem{Theorem}{Theorem}[section]
\newtheorem{Definition}[Theorem]{Definition}
\newtheorem{Proposition}[Theorem]{Proposition}
\newtheorem{Lemma}[Theorem]{Lemma}
\newtheorem{Corollary}[Theorem]{Corollary}

\newtheorem{Remark}[Theorem]{Remark}
\newtheorem{Example}[Theorem]{Example}
\newtheorem{Assumption}[Theorem]{Assumption}
\newtheorem{manualassumptioninner}{Assumption}
\newenvironment{manualassumption}[1]{%
  \renewcommand\themanualassumptioninner{#1}%
  \manualassumptioninner
}{\endmanualassumptioninner}

\newcommand{\N}{\mathbb{N}}
\newcommand{\Z}{\mathbb{Z}}
\newcommand{\R}{\mathbb{R}}
\newcommand{\C}{\mathbb{C}}

\newcommand{\norm}[1]{\left\lVert #1 \right\rVert}	
\newcommand{\normabs}[1]{\left| #1 \right|}	
\newcommand{\abs}[1]{\left| #1 \right|}	
\newcommand{\ra}{\rightarrow}
\newcommand{\Lnorm}[1]{{\left\lVert #1 \right\rVert}_{L^2}}

\newcommand\restr[2]{{
  \left.\kern-\nulldelimiterspace 
  #1 
  \vphantom{\big|} 
  \right|_{#2} 
  }}
\allowdisplaybreaks[4]
\newcommand{\mylabel}[2]{#2\def\@currentlabel{#2}\label{#1}}
\DeclareMathOperator*{\argmin}{arg\,min}
\begin{document}

\title{Statistical inference for continuous-time locally stationary processes using stationary approximations
}
\author{Bennet Ströh\footnote{Ulm University, Institute of Mathematical Finance, Helmholtzstra\ss e 18, 89069 Ulm, Germany. Email: bennet.stroeh@uni-ulm.de. }}

\maketitle

\textwidth=160mm \textheight=225mm \parindent=8mm \frenchspacing
\vspace{3mm}

\begin{abstract}
We establish asymptotic properties of $M$-estimators, defined in terms of a contrast function and observations from a continuous-time locally stationary process. Using the stationary approximation of the sequence, $\theta$-weak dependence, and hereditary properties, we give sufficient conditions on the contrast function that ensure consistency and asymptotic normality of the $M$-estimator.\\
As an example, we obtain consistency and asymptotic normality of a localized least squares estimator for observations from a sequence of time-varying L\'evy-driven Ornstein-Uhlenbeck processes. Furthermore, for a sequence of time-varying L\'evy-driven state space models, we show consistency of a localized Whittle estimator and an $M$-estimator that is based on a quasi maximum likelihood contrast.
Simulation studies show the applicability of the estimation procedures.
\end{abstract} 


\section{Introduction}
Various powerful inference methodologies for continuous-time stochastic processes are based on stationarity. One reason for this is that, in many cases, stationarity is essential to derive asymptotic results. For instance, based on stationarity arguments, different estimation procedures have been successfully applied to flexible and widely used continuous-time models in \cite{BDY2011, C2018, FHM2020, HKLZ2007} and \cite{SS2012b}.\\
However, numerous established models, including processes used in the references above, are inappropriate for modeling data that shows non-stationary behavior. To overcome this issue, \cite{SS2021} recently introduced a general theory on stationary approximations for non-stationary continuous-time processes that allows modeling non-stationary data. Heuristically, this approach follows the intuitive idea of local stationarity as discussed in \cite{D2012, DRW2019, DSR2006, V2012}, and assumes that a sequence of non-stationary processes can be locally approximated by a stationary process. Noticeable examples of time-series models discussed in \cite{SS2021} come from the class of time-varying L\'evy-driven Ornstein-Uhlenbeck processes and time-varying L\'evy-driven state space models.
Since such processes are non-stationary, classical methods used for statistical inference in a stationary setting cannot be applied, and novel estimation procedures are needed.\\
In the present work, we address this issue and provide inference methodologies for sequences of non-parametric non-stationary continuous-time processes that possess a locally stationary approximation. Noteworthy, our results are established in a model-free setting using limit theorems from \cite{SS2021}. We apply these results to study concrete estimators for several well-known non-stationary time series models and analyze their asymptotic properties.\\
To the best of our knowledge, the only comparable results can be found in \cite{KL2012}, where the authors investigate time-varying Gaussian-driven diffusion models and provide asymptotic results of a proposed estimator. Different from \cite{KL2012}, our theory also encompasses non-Gaussian and non-linear time series models. In the discrete-time setting, results similar to our theory have been obtained in \cite{BDW2020} and \cite{DRW2019}, where the authors derive a remarkably versatile theory including various analytical and statistical results for locally stationary processes. \\ 
More precisely, we introduce a class of kernel-based $M$-estimators whose objective function is a contrast that depends on observations sampled from a sequence of non-stationary processes. To establish consistency and asymptotic normality in a general setting, we impose conditions on the stationary approximation of the sequence and the contrast function. Specifically, the stationary approximation is assumed to be $\theta$-weakly dependent as introduced in \cite{DD2003} and the contrast function is assumed to satisfy identifiability and regularity conditions. In particular, these conditions ensure the existence of a $\theta$-weakly dependent stationary approximation of the contrast. The relative simplicity of the conditions allows us to readily derive asymptotic results for different contrast functions of finite and infinite memory.\\
For instance, we consider a sequence of time-varying L\'evy-driven Ornstein-Uhlenbeck processes and obtain, based on a least squares contrast, a consistent and asymptotically normally distributed $M$-estimator of the underlying coefficient function. The estimator's good performance is demonstrated through a simulation study in a finite sample for different coefficient functions.\\ 
Moreover, we consider a sequence of time-varying L\'evy-driven state space models, whose locally stationary approximation is a time-invariant L\'evy-driven state space model. The latter processes build a flexible class of continuous-time models that encompasses the well-known class of CARMA processes (see \cite{B2014,MS2007} for an introduction) and allow modeling high-frequency and irregularly spaced data occurring, for example, in finance and turbulence. Recently, a quasi-maximum likelihood and a Whittle estimator for L\'evy-driven state space models sampled at low frequencies have been discussed and compared in \cite{FHM2020}, and \cite{SS2012b}. We use results from these works and establish consistency results for two novel estimators, a localized quasi-maximum likelihood and a localized Whittle estimator. While the localized quasi-maximum likelihood estimator is a time domain $M$-estimator that is based on a log-likelihood contrast, the localized Whittle estimator is a frequency domain estimator constructed from a consistent estimator of the sample autocovariance. We compare both estimators in a simulation study, where their finite sample performances and convergence behaviors are studied.\\
The paper is structured as follows. In Section \ref{sec2}, all technical results needed throughout this work are presented. We review locally stationary approximations, introduce the sampling schemes in use, discuss $\theta$-weak dependence and outline hereditary properties of this measure of dependence and the stationary approximations.
In Section \ref{sec3}, we discuss the aforementioned class of $M$-estimators and establish consistency and asymptotic normality. 
In Section \ref{sec4}, we first review elementary properties of L\'evy processes, stochastic integration with respect to them, and time-varying Ornstein-Uhlenbeck processes. 
We then apply our results to a least squares contrast and obtain asymptotic results of the corresponding $M$-estimator, where the observations are sampled from a sequence of time-varying Ornstein-Uhlenbeck processes.\\
In Section \ref{sec5}, we first review time-varying L\'evy-driven state space models. Then, for observations sampled from a sequence of such processes, we propose a localized quasi-maximum likelihood estimator in Section \ref{sec5-2} and a localized Whittle estimator in Section \ref{sec5-4}. We show consistency of both estimators and present a truncated version of the localized quasi-maximum likelihood estimator in Section \ref{sec5-3}.\\The outcomes of the simulation study are discussed in Section \ref{sec6} and the proofs of most results are given in Section \ref{sec7}.

\subsection{Notation}
\label{sec1-1}
In this paper, we denote the set of positive integers by $\N$, non-negative integers by $\N_0$, positive real numbers by $\R^+$, non-negative real numbers by $\R_+^0$, the set of $m\times n$ matrices over a ring $R$ by $M_{m\times n}(R)$, and $\mathbf{1}_n$ stands for the $n\times n $ identity matrix. The real part of a complex number $z\in\C$ is written as $\mathfrak{Re}(z)$. For square matrices $A,B\in M_{n\times n}(R)$, $[A,B]=AB-BA$ denotes the commutator of $A$ and $B$. We shortly write the transpose of a matrix $A \in M_{m\times n}(\R)$ as $A'$, and norms of matrices and vectors are denoted by $\norm{\cdot}$. If the norm is not further specified, we take the Euclidean norm or its induced operator norm, respectively. For a bounded function $h$, $\norm{h}_\infty$ denotes the uniform norm of $h$. In the following Lipschitz continuous is understood to mean globally Lipschitz. For $u,n\in\N$, let $\mathcal{G}_u^*$ be the class of bounded functions from $(\R^n)^u$ to $\R$ and $\mathcal{G}_u$ be the class of bounded, Lipschitz continuous functions from $(\R^n)^u$ to $\R$ with respect to the distance $\sum_{i=1}^{u}\norm{x_i-y_i}$, where $x,y\in(\R^n)^u$. For $G\in\mathcal{G}_u$ we define 
\begin{align*}
Lip(G)=\sup_{x\neq y}\tfrac{|G(x)-G(y)|}{\norm{x_1-y_1}+\ldots+\norm{x_u-y_u}}.
\end{align*}
The Borel $\sigma$-algebras are denoted by $\mathcal{B}(\cdot)$ and $\lambda$ stands for the Lebesgue measure, at least in the context of measures. For a normed vector space $W$, we denote by $\ell^\infty(W)$ the space of all bounded sequences in $W$. 
In the following, we will assume all stochastic processes and random variables to be defined on a common complete probability space $(\Omega,\mathcal{F},P)$ equipped with an appropriate filtration if necessary. 
Finally, we simply write $L^p$ to denote the space $L^p(\Omega,\mathcal{F},P)$ and $L^p(\R)$ to denote the space $L^p(\R,\mathcal{B}(\R),\lambda)$ with corresponding norms $\norm{\cdot}_{L^p}$. 

\section{Locally stationary approximations and $\theta$-weak dependence}
\label{sec2}

\subsection{Locally stationary approximations}
\label{sec2-1}
Throughout this paper, we consider sequences of processes that can be locally approximated in $L^p$ by a stationary process. This concept is a non-parametric approach to express the intuitive idea of local stationarity, as discussed by Dahlhaus and others (see e.g. \cite{D2012,V2012}). In this paper, we consider locally stationary approximations defined as follows.

\begin{Definition}[{\cite[Definition 2.1]{SS2021}}]\label{definition:statapproxconttime}
Let $Y_N=\{Y_N(t),t\in\R\}_{N\in\N}$ be a sequence of real-valued stochastic processes and $\tilde{Y}=\{\tilde{Y}_u(t),t\in\R\}_{u\in\R^+}$ a family of real-valued stationary processes. We assume that the process $\tilde{Y}_u$ is ergodic for all $u\in\R^+$ and $\sup_{u\in \R^+} \norm{\tilde{Y}_u(0)}_{L^p}<\infty$ for some $p\geq1$. If there exists a constant $C>0$, such that uniformly in $t\in \R$ and $u,v\in\R^+$
\begin{gather}
\lVert\tilde{Y}_u(t)-\tilde{Y}_v(t)\rVert_{L^p}\leq C \abs{u-v} \text{ and }\qquad \lVert Y_N(t)-\tilde{Y}_{t}(Nt)\rVert _{L^p} \leq C \frac{1}{N},\label{assumption:LS}\tag{LS}
\end{gather}
then we call $\tilde{Y}_u$ a \emph{locally stationary approximation} of the sequence $Y_N$ \emph{for} $p$. 
\end{Definition}

If $\tilde{Y}_u$ is a locally stationary approximation of $Y_N$ for $p$, then it is also a locally stationary approximation for $p'$, where $1\leq p'\leq p$. \\
Whenever we investigate estimators based on observations from a sequence of processes, we assume the observations to be sampled according to one of the following schemes.

\begin{Assumption}\label{assumption:observations}
For fixed $N\in\N$ and $u\in\R^+$ we assume $Y_N$ to be equidistantly observed at times $\tau_i^N=u+i\delta_N$ with grid size $\delta_N=|\tau_i^N-\tau_{i-1}^N|$ such that $\delta_N\downarrow 0$ for $N\rightarrow\infty$. For a sequence $b_N\downarrow0$ we consider the observation window $[u-b_N,u+b_N]$ and set $m_N=\lfloor b_N/ \delta_N \rfloor$. Thus, the number of observations is given by $2m_N+1=|\{i\in\Z: \tau_i^N\in [u-b_N,u+b_N]\}|$. We require $b_N/\delta_N\rightarrow\infty$ as $N\rightarrow\infty$ and either
\begin{enumerate}[label=\textbf{(O\arabic*)}]
\item $N\delta_N=\delta>0$ for all $N\in\N$ or \label{observations:O1}
\item $N\delta_N\rightarrow\infty$ as $N\rightarrow\infty$.\label{observations:O2}
\end{enumerate}
\end{Assumption}

Note that these conditions on $N$, $b_N$ and $\delta_N$ immediately imply that $Nb_N\rightarrow\infty$ as $N\rightarrow\infty$. For a comprehensive discussion on the above approximations and observations, including examples of sequences that satisfy Definition \ref{definition:statapproxconttime}, we refer to \cite{SS2021}.

\subsection{$\theta$-weak dependence and hereditary properties}
\label{sec2-2}
In this section we summarize results that are needed throughout the paper. We start with a brief review of the concept of $\theta$-weak dependence.

\begin{Definition}[{\cite{DD2003}}]\label{thetaweaklydependent}
Let $X=\{X(t)\}_{t\in\R}$ be an $\R^n$-valued stochastic process. Then, $X$ is called $\theta$-weakly dependent if
\begin{gather*}
\theta(h)=\sup_{v\in\N}\theta_{v}(h) \underset{h\ra\infty}{\longrightarrow} 0,
\end{gather*} 
where
\begin{align*}
\theta_{v}(h)\!=\!\sup\bigg\{\frac{|Cov(F(X(i_1),\ldots,X(i_v)),G(X(j)))|}{\norm{F}_{\infty}Lip(G)}, F\in\mathcal{G}_u^*,G\in\mathcal{G}_1, i_1\leq\ldots\leq i_v\leq i_v+h\leq j \bigg\}.
\end{align*}
We call $(\theta(h))_{h\in\R_0^+}$ the $\theta$-coefficients. 
\end{Definition}

Next, we summarize hereditary properties of locally stationary approximations and $\theta$-weak dependence under transformations (see \cite[Section 2.3 and 2.4]{SS2021} for a comprehensive discussion).\\
Let $Y_N$ be a sequence of stochastic processes with locally stationary approximation $\tilde{Y}_u$ for some $p\geq1$. For $k\in\N_0$ we define the infinite and finite memory vectors
\begin{alignat*}{3}
Z_N(t)&=\Big(Y_N(t),Y_N\Big(t-\tfrac{1}{N}\Big),\ldots\Big)&\text{ and } &\tilde{Z}_u(t)=(\tilde{Y}_u(t),\tilde{Y}_u(t-1),\ldots),\text{ as well as}\\
Z_N^{(k)}(t)&=\Big(Y_N(t),Y_N\Big(t-\tfrac{1}{N}\Big),\ldots,Y_N\Big(t-\tfrac{k}{N}\Big)\Big)&\text{ and } &\tilde{Z}_u^{(k)}(t)=(\tilde{Y}_u(t),\tilde{Y}_u(t-1),\ldots,\tilde{Y}_u(t-k)).
\end{alignat*}

For functions from the following two classes we obtain hereditary properties.

\begin{Definition}[{\cite[Definition 2.4]{DRW2019}}]
A measurable function $g:\R^{k+1}\rightarrow\R$ is said to be in the class $\mathcal{L}_{k+1}(M,C)$ for $M\geq0$ and $C\in[0,\infty]$, if 
\begin{align*}
\sup_{x\neq y}\frac{|g(x)-g(y)|}{\norm{x-y}_1(1+\norm{x}_1^M+\norm{y}_1^M)}\leq C.
\end{align*}
\end{Definition}

\begin{Definition}[{\cite[Definition 2.10]{SS2021}}]\label{definition:functionclassinfinite}
A measurable function $h:\R^\infty\rightarrow\R^n$ is said to belong to the class $\mathcal{L}_\infty^{p,q}(\alpha)$ for $p,q \geq1$, if there exists a sequence $\alpha=(\alpha_k)_{k\in\N_0}\subset\R_0$ satisfying $\sum_{k=0}^\infty\alpha_k<\infty$ and a function $f:\R_0^+\rightarrow\R_0^+$ such that for all sequences $X=(X_k)_{k\in\N_0}\in \ell^\infty(L^q)$ and $Y=(Y_k)_{k\in\N_0}\in\ell^\infty(L^q)$ it holds
\begin{align*}
\norm{h(X)-h(Y)}_{L^p}\leq f\Big(\sup_{k\in\N_0}\{\norm{X_k}_{L^q}\vee \norm{Y_k}_{L^q}\}\Big) \sum_{k=0}^\infty \alpha_k\norm{X_k-Y_k}_{L^q}.
\end{align*}
\end{Definition}

The next proposition is a combination of Proposition 2.7 and 2.11 from \cite{SS2021}.

\begin{Proposition}\label{proposition:inheritanceproperties}
Let $Y_N$ be a sequence of stochastic processes with locally stationary approximation $\tilde{Y}_u$ for some $q\geq1$. Then, for $g\in\mathcal{L}_{k+1}(M,C)$ and a real-valued function $h\in \mathcal{L}_{\infty}^{p,q}(\alpha)$, where $M\geq0$, $C\in [0,\infty)$, $p\geq1$ and $\sum_{k=0}^\infty k\alpha_k<\infty$, it holds:
\begin{enumerate}[label={(\alph*)}]
\item $g(\tilde{Z}_u^{(k)}(t))$ is a locally stationary approximation of the sequence $g(Z_N^{(k)}(t))$ for $\tilde{p}=\frac{q}{M+1}$.
\item If $\tilde{Y}_u$ is $\theta$-weakly dependent with $\theta$-coefficients $\theta_{\tilde{Y}_u}(h)$, then $\tilde{Z}_u^{(k)}$ is $\theta$-weakly dependent with $\theta$-coefficients $\theta_{\tilde{Z}_u^{(k)}}(h)\leq (k+1) \theta_{\tilde{Y}_u}(h-(k+1))$ for $h\geq (k+1)$.
\item If $\tilde{Y}_u$ is $\theta$-weakly dependent with $\theta$-coefficients $\theta_{\tilde{Y}_u}(h)$, satisfies $E[|\tilde{Y}_u(t)|^{(1+M+\gamma)}]<\infty$ for some $\gamma>0$ and additionally $|g(x)|\leq \tilde{C} \norm{x}_1^{M+1}$ for a constant $\tilde{C}>0$, then $g(\tilde{Z}_u^{(k)}(t))$ is $\theta$-weakly dependent with $\theta$-coefficients $\theta_{g(\tilde{Z}_u^{(k)})}(h)=\mathcal{O}\left(\theta_{\tilde{Y}_u}(h)^{\frac{\gamma}{M+\gamma}}\right)$.
\item $h(\tilde{Z}_u(t))$ is a locally stationary approximation of $h(Z_N(t))$ for $p$.
\end{enumerate}
\end{Proposition}

If $\tilde{Y}_u(t)$ is a L\'evy-driven moving average processes (see Section \ref{sec4-1}) we give sufficient conditions for $h(\tilde{Z}_u(t))$ to be $\theta$-weak dependence in Proposition \ref{proposition:infinitememorymovingaverage}.

\section{$M$-estimation of contrast functions based on locally stationary approximations}
\label{sec3}
Let $Y_N$ be a sequence of stochastic processes with locally stationary approximation $\tilde{Y}_u$ as described in Definition \ref{definition:statapproxconttime}. 
In this section, we study localized $M$-estimators of contrast functions based on observations of $Y_N$. In a discrete-time setting, such an estimation procedure has recently been investigated in \cite{BDW2020} and \cite{DRW2019}. The contrast functions we investigate are assumed to be of the form 
\begin{gather}\label{equation:contrastfunctioninfinitememory}
\Phi\Big(\Big(\tilde{Y}_u(\Delta (1-k))\Big)_{k\in\N_0},\vartheta\Big),
\end{gather}
where $\tilde{Y}=(\tilde{Y}_u(\Delta (1-k)))_{k\in\N_0}$ is a sequence in $L^p$ for $p\geq1, \Delta>0$ and $\vartheta\in\Theta$, where $\Theta\subset\R^d$ is a parameter space. We assume that the true parameter $\vartheta^*$ is identifiable from the contrast, i.e.

\begin{manualassumption}{(M1)}\label{assumption:M1}
Assume that $\Phi(\tilde{Y},\vartheta)\in L^1$ for all $\vartheta\in\Theta$ and that $\vartheta^*$ is the unique minimum in $\Theta$ of the function $\vartheta\mapsto E[\Phi(\tilde{Y},\vartheta)]=M(\vartheta)$.
\end{manualassumption}

\noindent In a stationary setting, the natural choice of an $M$-estimator for $\vartheta^*$ is
\begin{gather}\label{equation:naturalMestimator}
\argmin_{\vartheta\in\Theta}\frac{1}{n}\sum_{i=1}^{n}\Phi\left((\tilde{Y}_u(i+\Delta (1-k)))_{k\in\N_0},\vartheta\right).
\end{gather}
For processes that possess a locally stationary approximation, a localized law of large numbers has recently been proven in \cite[Theorem 3.5 and 3.6]{SS2021}. There, the localization is achieved by using a localizing kernel of the following type.

\begin{Definition}\label{definition:localizingkernel}
Let $K:\R\rightarrow \R$ be a bounded function. If $K$ is of bounded variation, has compact support $[-1,1]$ and satisfies $\int_\R K(x)dx=1$, then we call $K$ a localizing kernel.
\end{Definition}

\noindent From now on, if not otherwise stated, $K$ always denotes a localizing kernel.\\
Following this approach we replace the observations of $\tilde{Y}_u$ in (\ref{equation:naturalMestimator}) by observations of the sequence $Y_N$ as defined in Assumption \ref{assumption:observations}, leading to the localized estimator
\begin{align}\label{eq:M_N}
\begin{aligned}
\hat{\vartheta}_N&=\argmin_{\vartheta\in\Theta}M_N(\vartheta), \text{where}\\
M_N(\vartheta)&=\frac{\delta_N}{b_N}\sum_{i=-m_N}^{m_N}K\left(\frac{\tau_i^N-u}{b_N} \right)\Phi\bigg(\bigg(Y_N\bigg(\tau_i^N+\frac{\Delta (1-k)}{N}\bigg)\bigg)_{k\in\N_0},\vartheta\bigg).
\end{aligned}
\end{align}

\noindent In the next two sections, we derive sufficient conditions on the contrast $\Phi$ that ensure consistency and asymptotic normality of $\hat{\vartheta}_N$. As first step, we give conditions ensuring that $\Phi$ is integrable.   

\begin{Lemma}\label{lemma:integrabilityphi}
Let $\Phi:\R^\infty\times\R^d\rightarrow\R$ be a measurable function. If $\Phi(\cdot,\vartheta)\in \mathcal{L}_\infty^{p,q}(\alpha)$ for all $\vartheta\in\Theta$ and $\sup_{\vartheta\in\Theta}\norm{\Phi(0,\vartheta)}<\infty$, then $\Phi(X,\vartheta)\in L^1$ for all $X=(X_k)_{k\in\N_0}\in \ell^\infty(L^q)$ and $\vartheta\in\Theta$. Moreover, if $X=(X_t)_{t\in\R}$ is a stationary integrable ergodic process, then $\left(\Phi\left((X_{t-k})_{k\in\N_0},\vartheta\right)\right)_{t\in\R}$ is again a stationary integrable ergodic process for all $\vartheta\in\Theta$.
\end{Lemma}
\begin{proof}
For $t\in\R$ and $m\in\N$ we have that $\phi_{t,m}=\Phi(X_{t},\ldots,X_{t-m},0,\ldots,\vartheta)\in L^1$, since $\sup_{\vartheta\in\Theta}\norm{\Phi(0,\vartheta)}<\infty$ and $\Phi(\cdot,\vartheta)\in \mathcal{L}_\infty^{p,q}(\alpha)$. Then, similar to \cite[Lemma 3.1]{BDW2020}, one can show that $\phi_{t,m}$ is a Cauchy sequence in $L^1$. Noting that $\Phi$ is measurable, we conclude analogous to \cite[Proposition 4.3]{K1985}.
\end{proof}

\noindent Note that Lemma \ref{lemma:integrabilityphi} is often used implicitly in the following.

\subsection{Consistency}
\label{sec3-1}
We now show pointwise convergence of $M_N(\vartheta)$, i.e. $M_N(\vartheta)\overset{P}{\rightarrow} M(\vartheta)$ for all $\vartheta\in\Theta$ as $N\rightarrow\infty$ and stochastic equicontinuity of the sequence $\{M_N(\vartheta)\}_{N\in\N}$. Together, these properties imply $\hat{\vartheta}_N\overset{P}{\longrightarrow}\vartheta^*$ as $N\rightarrow\infty$ along usual lines.
To show pointwise convergence, we use the localized law of large numbers from \cite{SS2021}. To this end, it is necessary to impose regularity conditions on $\Phi$. We demand that $\Phi$ belongs to $\mathcal{L}_{\infty}^{p,q}$ for each $\vartheta \in \Theta$. Moreover, if \hyperref[observations:O2]{(O2)} holds, it is clear that $\Phi(\tilde{Y},\vartheta)$ has to be $\theta$-weakly dependent for all $\vartheta\in\Theta$ (see \cite[Theorem 3.6]{SS2021}). Besides this, $\Phi$ has to additionally belong to $\mathcal{L}_d$ with respect to the parameter space to ensure stochastic equicontinuity.
%
%
%
\begin{Theorem}\label{theorem:consistency}
Let $Y_N$ be a sequence of stochastic processes with locally stationary approximation $\tilde{Y}_u$ for some $q\geq1$ such that either \hyperref[observations:O1]{(O1)} or \hyperref[observations:O2]{(O2)} holds. Besides, for some $p\geq1$ and a compact set $\Theta\subset\R^d$, we assume that 
\begin{enumerate}[label={(\alph*)}]
\item $\Phi(\cdot,\vartheta)\in \mathcal{L}_\infty^{p,q}(\alpha)$ for all $\vartheta\in\Theta$, such that $\sum_{k=0}^\infty k \alpha_k<\infty$,
\item $\Phi(x,\cdot)\in \mathcal{L}_d\left(0,D_1(1+\sum_{k=0}^\infty \beta_k|x_k|^q) \right)$ for all real sequences $x=(x_k)_{k\in\N_0}$ and some $D_1\geq0$, where $(\beta_k)_{k\in\N_0}\subset\R^+$ is a sequence such that $\sum_{k=0}^\infty k\beta_k<\infty$,
\item if \hyperref[observations:O2]{(O2)} holds, both $\Phi((\tilde{Y}_u(t+\Delta (1-k)))_{k\in\N_0},\vartheta)$ as well as $g((\tilde{Y}_u(t+\Delta (1-k)))_{k\in\N_0})$ are $\theta$-weakly dependent for all $\vartheta\in\Theta$, where $g(x)=\sum_{k=0}^\infty \beta_k|x_k|^q$,
\item $\sup_{\vartheta\in\Theta}\norm{\Phi(0,\vartheta)}<\infty$ and
\item the identifiability condition \hyperref[assumption:M1]{(M1)} holds.
\end{enumerate}
Then, $\hat{\vartheta}_N$ is consistent, i.e. $\hat{\vartheta}_N\overset{P}{\longrightarrow}\vartheta^*$ as $N\rightarrow\infty$.
\end{Theorem}
\begin{proof}
See Section \ref{sec7-1}.
\end{proof}

\begin{Remark}\label{remark:consistencyfinitememory}
In the case where \hyperref[observations:O2]{(O2)} holds and the contrast function $\Phi$ is of finite memory, i.e. there exists $n\in\N_0$ such that $\Phi((\tilde{Y}_u(\Delta (1-k)))_{k\in\N_0},\vartheta)=\Phi(\tilde{Y}_u(\Delta),\ldots,\tilde{Y}_u(\Delta(1-n)),\vartheta)$, Proposition \ref{proposition:inheritanceproperties} shows that the conditions (c) and (d) of Theorem \ref{theorem:consistency} are implied by the condition that
\begin{enumerate}[label={(\alph*)}]
\item[(c$^*$)] $\Phi(x,\vartheta)\leq C\norm{x}_1^{M+1}$ and $\Phi(\cdot,\vartheta)\in\mathcal{L}_{n+1}(M,C)$ for some $C,M\geq0$ and all $x\in\R^{n+1}$, $\vartheta\in\Theta$. Moreover, $\tilde{Y}_u$ is $\theta$-weakly dependent and $E[|\tilde{Y}_u|^{(q\vee (M+1))+\gamma}]<\infty$ for some $\gamma>0$.
\end{enumerate}
For contrast functions that are of infinite memory, we give sufficient conditions for (c) in Corollary \ref{corollary:thetaweakdependenceinfinitememory}, where we investigate processes, whose locally stationary approximation $\tilde{Y}_u$ is a L\'evy-driven moving average. 
\end{Remark}

\subsection{Asymptotic normality}
\label{sec3-2}
To establish asymptotic normality of $\hat{\vartheta}_N$ we follow the classical approach (see e.g. \cite[Section 5.3]{V1998}) to show asymptotic normality of an $M$-estimator. We impose conditions on the first and second order partial derivatives of the contrast $\Phi$ and investigate the Taylor expansion of $\nabla_\vartheta M_N$ at $\vartheta^*$. The individual components of the expansion are then shown to either converge to $0$ or to be asymptotically normal. The localization is achieved by using the rectangular kernel
\begin{gather}\label{equation:rectangularkernel}
K_{rect}(x)=\frac{1}{2}\mathbb{1}_{\{x\in[-1,1]\}}.
\end{gather}
It is easy to see that $K_{rect}$ is a localizing kernel. Depending on whether \hyperref[observations:O1]{(O1)} or \hyperref[observations:O2]{(O2)} holds, we obtain different asymptotic variances.\\
To establish asymptotic normality of the components of the Taylor expansion we use results from \cite{SS2021}. There, the authors derived central limit type results under the following condition on the $\theta$-coefficient $\theta(h)$ of the locally stationary approximation
\begin{align*}
\hypertarget{DD}{DD(\varepsilon):} \qquad\sum_{h=1}^\infty\theta(h)h^{\frac{1}{\varepsilon}}<\infty \quad \text{for some }\varepsilon>0.
\end{align*}
Sufficient conditions for \hyperlink{DD}{DD($\varepsilon$)} to hold are for instance $\theta(h)\in\mathcal{O}(h^{-\alpha})$ for some $\alpha>(1+\frac{1}{\varepsilon})$ or $\theta(h)\in\mathcal{O}\big(\big(h\ln(h)\big)^{-1-\frac{1}{\varepsilon}}\big)$.


\begin{Theorem}\label{theorem:asymptoticnormality}
Let $q,\tilde{q},\bar{q}\geq1$ and $Y_N$ be a sequence of stochastic processes with locally stationary approximation $\tilde{Y}_u$ for some $s\geq\max\{q,\tilde{q},\bar{q}\}$ such that either \hyperref[observations:O1]{(O1)} or \hyperref[observations:O2]{(O2)} holds. The contrast function $\Phi$ is assumed to be of the form (\ref{equation:contrastfunctioninfinitememory}) such that the Hessian matrix $\nabla_\vartheta^2 \Phi$ of $\Phi$ with respect to $\vartheta$ exists. 
Moreover, assume that
\begin{enumerate}[label={(\alph*)}]
\item the parameter space $\Theta\subset\R^d$ is compact, \hyperref[assumption:M1]{(M1)} holds and the unique minimum $\vartheta^*$ is located in the interior of $\Theta$.
\item $\sqrt{m_N}b_N\rightarrow0$ as $N\rightarrow\infty$ and the localizing kernel is given by (\ref{equation:rectangularkernel}).
\item $\Phi(x,\cdot)\in \mathcal{L}_d\left(0,D_0(1+\sum_{k=0}^\infty \beta_k|x_k|^q) \right)$ for all real sequences $x=(x_k)_{k\in\N_0}$ and some $D_0\geq0$, where $(\beta_k)_{k\in\N_0}\subset\R^+$ is a sequence such that $\sum_{k=0}^\infty k\beta_k<\infty$ and $\Phi(\cdot,\vartheta)\in\mathcal{L}_{\infty}^{p,q}(\alpha)$ for all $\vartheta\in\Theta$, where $p\geq1$ and $\sum_{k=0}^\infty k \alpha_k<\infty$.
\item $\frac{\partial}{\partial\vartheta_i}\Phi(\cdot,\vartheta^*)\in\mathcal{L}_{\infty}^{\tilde{p},\tilde{q}}(\tilde{\alpha})$ for all $i=1,\ldots,d$, where $\tilde{p}\geq2$ and $\sum_{k=0}^\infty k \tilde{\alpha}_k<\infty$. 
\item the stationary process $\nabla_\vartheta \Phi(t):=\nabla_\vartheta \Phi\left( \left(\tilde{Y}_u(t+\Delta(1-k))\right)_{k\in\N_0},\vartheta^*\right)\in L^{2+\gamma_1}$ for some $\gamma_1>0$. Moreover, $\nabla_\vartheta \Phi(t)$ is $\theta$-weakly dependent with $\theta$-coefficients $\theta(h)$ satisfying \hyperlink{DD}{DD($\gamma_1$)}.
\item $\frac{\partial^2}{\partial\vartheta_i\partial\vartheta_j}\Phi(x,\cdot)\in \mathcal{L}_d\left(0,D_1(1+\sum_{k=0}^\infty \bar{\beta}_k|x_k|^{\bar{q}}) \right)$ for all real sequences $x=(x_k)_{k\in\N_0}$, $i,j=1,\ldots,d$ and some $D_1\geq0$, where $(\bar{\beta}_k)_{k\in\N_0}\subset\R^+$ is a sequence such that $\sum_{k=0}^\infty k\bar{\beta}_k<\infty$ and $\frac{\partial^2}{\partial\vartheta_i\partial\vartheta_j}\Phi(\cdot,\vartheta)\in\mathcal{L}_{\infty}^{\bar{p},\bar{q}}(\bar{\alpha})$ for all $\vartheta\in\Theta$, $i,j=1,\ldots,d$, where $\bar{p}\geq1$ and $\sum_{k=0}^\infty k \bar{\alpha}_k<\infty$.
\item if \hyperref[observations:O2]{(O2)} holds, the following conditions are satisfied:
\begin{enumerate}[label={(\alph*)}]
\item[(g1)] the processes $\Phi\left( \left(\tilde{Y}_u(t+\Delta(1-k))\right)_{k\in\N_0},\vartheta\right)$ and $g\left(\left(\tilde{Y}_u(t+\Delta(1-k))\right)_{k\in\N_0}\right)$ are $\theta$-weakly dependent for all $\vartheta\in\Theta$, where $g((x_k)_{k\in\N_0})=\sum_{k=0}^\infty\beta_k\abs{x_k}^{\bar{q}}$.
\item[(g2)] the  processes $\frac{\partial^2}{\partial\vartheta_i\partial\vartheta_j}\Phi\left( \left(\tilde{Y}_u(t+\Delta(1-k))\right)_{k\in\N_0},\vartheta\right)$ and $\bar{g}\left(\left(\tilde{Y}_u(t+\Delta(1-k))\right)_{k\in\N_0}\right)$ are $\theta$-weakly dependent for all $\vartheta\in\Theta$ and $i,j=1,\ldots,d$, where $\bar{g}((x_k)_{k\in\N_0})=\sum_{k=0}^\infty\Bar{\beta}_k\abs{x_k}^{\bar{q}}$.
\end{enumerate}
\item $\sup_{\vartheta\in\Theta}\norm{\Phi(0,\vartheta)}<\infty$ and $\sup_{\vartheta\in\Theta}\norm{\frac{\partial^2}{\partial\vartheta_i\partial\vartheta_j}\Phi(0,\vartheta)}<\infty$ for all $i,j=1,\ldots, d$.
\item the matrices 
\begin{align*}
I(u)&=\begin{cases}\frac{1}{2}I(u,0)+\sum_{k=1}^\infty I(u,k), &\text{if \hyperref[observations:O1]{(O1)} holds,}\\
\frac{1}{2}I(u,0), &\text{if \hyperref[observations:O2]{(O2)} holds}\end{cases}\text{ and}\\
V(u)&=E\left[ \nabla_\vartheta^2 \Phi\left(\left(\tilde{Y}_u(\Delta(1-k))\right)_{k\in\N_0},\vartheta^* \right)\right]
\end{align*}
are positive definite, where
\begin{gather*}
I(u,k)=E\left[ \nabla_\vartheta \Phi\left(\left(\tilde{Y}_u(\Delta(1-k))\right)_{k\in\N_0},\vartheta^* \right)\nabla_\vartheta 
\Phi\left(\left(\tilde{Y}_u(k\delta+\Delta(1-k))\right)_{k\in\N_0},\vartheta^* \right)'\right].
\end{gather*}
\end{enumerate}
Then, it holds
\begin{gather}\label{eq:asymptoticnormalitythetahat}
\sqrt{\frac{b_N}{\delta_N}} \left(\hat{\vartheta}_N-\vartheta^* \right)\overset{d}{\underset{N\rightarrow\infty}{\longrightarrow}} \mathcal{N}\left(0, V(u)^{-1}I(u)V(u)^{-1}\right).
\end{gather}
\end{Theorem}
\begin{proof}
See Section \ref{sec7-2}.
\end{proof}

\begin{Remark}\label{remark:asymptoticnormalityfinitememory}
If the contrast function $\Phi$ is of finite memory (see Remark \ref{remark:consistencyfinitememory}), Proposition \ref{proposition:inheritanceproperties} and the obvious analog of \cite[Proposition 3.4]{CS2018} for our $\theta$-weak dependence coefficient, show that condition (e) is implied by
\begin{enumerate}[label={(\alph*)}]
\item[(e$^*$)] $\nabla_\vartheta \Phi(0,\vartheta^*)=0$, $\frac{\partial}{\partial\vartheta_i} \Phi(\cdot,\vartheta^*)\in\mathcal{L}_{n+1}(M_1,C_1)$ for some $C_1,M_1\geq0$ and $\tilde{Y}_u$ is $\theta$-weakly dependent with $\theta$-coefficients $\theta_{\tilde{Y}_u}(h)\in\mathcal{O}(h^{-\alpha})$ for some $\alpha>(1+\frac{M_1+1}{\gamma_1})(\frac{1+2M_1+\gamma_1}{1+M_1+\gamma_1})$, where $\gamma_1>0$ such that $\tilde{Y}_u(0)\in L^{2(M_1+1)+\gamma_1}$.
\end{enumerate}
If, in addition, \hyperref[observations:O2]{(O2)} holds, condition (g) is implied by
\ref{theorem:consistency} by
\begin{enumerate}[label={(\alph*)}]
\item[(g1$^*$)] $\Phi(x,\vartheta)\leq C_2\norm{x}^{M_2+1}$ and $\Phi(\cdot,\vartheta)\in\mathcal{L}_{n+1}(M_2,C_2)$ for some $M_2,C_2\geq0$ and all $x\in\R^{n+1}$, $\vartheta\in\Theta$. Moreover, $\tilde{Y}_u$ is $\theta$-weakly dependent and $\tilde{Y}_u\in L^{(q\vee (M_2+1))+\gamma_2}$ for some $\gamma_2>0$.
\item[(g2$^*$)] $\frac{\partial^2}{\partial\vartheta_i\partial\vartheta_j}\Phi(x,\vartheta)\leq C_3\norm{x}^{M_3+1}$ and $\frac{\partial^2}{\partial\vartheta_i\partial\vartheta_j}\Phi(\cdot,\vartheta)\in\mathcal{L}_{n+1}(M_3,C_3)$ for some $M_3,C_3\geq0$ and all $x\in\R^{n+1}$, $i,j=1,\ldots,d$ and $\vartheta\in\Theta$. Moreover, $\tilde{Y}_u$ is $\theta$-weakly dependent and $\tilde{Y}_u\in L^{(\bar{q}\vee (M_3+1))+\gamma_3}$ for some $\gamma_3>0$.
\end{enumerate}
\end{Remark}

\section{Least squares estimation for time-varying L\'evy-driven Orn\-stein-Uhlenbeck processes}
\label{sec4}
In this section, we establish consistency and asymptotic normality of an $M$-estimator using results from Section \ref{sec3} for a least squares contrast. The observations are assumed to be sampled according to Assumption \ref{assumption:observations} from a sequence of time-varying L\'evy-driven Ornstein-Uhlenbeck processes, which possesses a locally stationary approximation.\\ 
Before we give the definition of (time-varying) L\'evy-driven Ornstein-Uhlenbeck processes, we review L\'evy processes and discuss basic results including stochastic integration with respect to them. For further insight we refer to \cite{A2009} and \cite{S2014}.

\subsection{L\'evy processes and stochastic integration}
\label{sec4-1}

\begin{Definition}
A real-valued stochastic process $L=\{L(t),t\in \R_0^+\}$ is called L\'evy process if
\begin{enumerate}[label={(\alph*)}]
\item $L(0)=0$ almost surely,
\item the random variables $(L(t_0),L(t_1)-L(t_0),\dots,L(t_n)-L(t_{n-1}))$ are independent for any $n\in\N$ and $t_0<t_1<t_2<\dots<t_n$, 
\item for all $s,t \geq 0$, the distribution of $L(s+t)-L(s)$ does not depend on $s$ and
\item $L$ is stochastically continuous.
\end{enumerate}
Without loss of generality we additionally consider $L$ to be c\`adl\`ag, i.e. right continuous with finite left limits.
\end{Definition}

Let $L=\{L(t),t\in \R_0^+\}$ be a real-valued L\'evy process. Then, $L(1)$ is an infinitely divisible real-valued random variable with characteristic triplet $(\gamma,\Sigma,\nu)$, where $\gamma \in \R$, $\Sigma>0$ and $\nu$ is a L\'evy measure on $\R$, i.e. $\nu(0)=0$ and $\int_{\R}\left(1\wedge\normabs{x}^2\right)\nu(dx)<\infty$. The characteristic function of $L(t)$ is given by $\varphi_{L(t)}(z)=E[e^{izL(t)}]=e^{t \Psi_{L}(z)}$ with characteristic exponent $\Psi_L(z)=( i\gamma z -\frac{\Sigma z^2}{2}+\int_{\R } (e^{i z x}-1-i z x \mathbb{1}_{Z}(x))\nu(dx))$, $z \in \R$ and $Z=\{x \in \R, \normabs{x}\leq 1\}$. If $\nu$ has finite second moment, i.e.
\begin{gather}\label{eq:condlevyproc}
\int_{\normabs{x}>1}\normabs{x}^2\nu(dx)<\infty \left(\iff\int_{\R}\normabs{x}^2\nu(dx)<\infty\right),
\end{gather}
then $L(t)\in L^2$ for all $t\geq0$ and we have $E[L(t)]=t \left(\gamma+\int_{\normabs{x}>1}x\nu(dx)\right)<\infty$ and $Var(L(t))=t \left(\Sigma +\int_{\R}x^2\nu(dx)\right)<\infty$. 
In the remainder we work with two-sided L\'evy processes, i.e. $L(t)=L_1(t)\mathbb{1}_{\{t\geq0\}} - L_2(-t)\mathbb{1}_{\{t<0\}}$, where $L_1$ and $L_2$ are independent copies of a one-sided L\'evy process. Consider
\begin{gather}\label{eq:stochint}
X(t)=\int_\R f(t,s) L(ds),
\end{gather}
where $t\in\R$ and $f:\R\times\R\mapsto\R^n$ is $\mathcal{B}(\R\times\R)-\mathcal{B}(\R^n)$ measurable.
Necessary and sufficient conditions for the stochastic integral (\ref{eq:stochint}) to exist are given in \cite[Theorem 3.3]{S2014}, namely if 
\begin{itemize}
\item $\Sigma\int_\R \norm{f(t,s)f(t,s)'}ds<\infty$,
\item $\int_\R \int_{\R}\left((\norm{f(t,s)}x)^2\wedge 1 \right)\nu(dx)ds<\infty$ and
\item $\int_\R \norm{f(t,s)\left( \gamma+\int_{\R}x\left(\mathbb{1}_{[0,1]} \left(\norm{f(t,s)x}\right)- \mathbb{1}_{[0,1]}\left( \normabs{x}\right)\nu(dx)\right) \right)}ds<\infty$
\end{itemize}
are satisfied, then (\ref{eq:stochint}) is well-defined. If $L$ satisfies (\ref{eq:condlevyproc}) and $f(t,\cdot)\in L^1(\R)\cap L^2(\R)$, then the above conditions are satisfied and the integral $X(t)=\int_\R f(t,s)L(ds)$ exists in $L^2$. If $X=\{X(t), t\in\R\}$ with $X(t)$ as in (\ref{eq:stochint}) is well-defined, $X(t)$ is infinitely divisible with characteristic triplet $(\gamma_{int},\Sigma_{int},\nu_{int})$, where
\begin{itemize}
\item $\gamma_{int}=\int_{\R}(f(t,s)\gamma ds+ \int_\R\int_{\R}f(t,s)x(\mathbb{1}_{[0,1]}(\norm{f(t,s)x})-\mathbb{1}_{[0,1]}(\normabs{x}))\nu(dx))ds$,
\item $\Sigma_{int}=\Sigma\!\int_{\R}f(t,s)f(t,s)'ds$ and
\item $\nu_{int}(B)= \int_{\R}\int_{\R} \mathbb{1}_B(f(t,s)x)\nu(dx) ds$,\qquad $B\in\mathcal{B}(\R^n)$.
\end{itemize}

The following proposition shows that infinite memory transformations of L\'evy-driven moving average processes, i.e. processes of the form (\ref{eq:stochint}) for which $g(u,t,s)=g(u,t-s)$, are $\theta$-weakly dependent. 

\begin{Proposition}\label{proposition:infinitememorymovingaverage}
Let $L$ be a two-sided L\'evy process satisfying (\ref{eq:condlevyproc}), $\mu+\int_{\abs{x}>1}x\nu(dx)=0$ and $g:\R^+\times\R\mapsto\R$ a function such that $g(u,\cdot)\in L^1(\R)\cap L^q(\R)$ for all $u\in\R^+$ and some $q\in\{2,4\}$. For fixed $u\in\R^+$ we define the process $X_u=\{X_u(t),t\in\R\}$ as
\begin{align}\label{equation:infinitememorylevydrivenmovingaverage}
X_u(t)=\int_{-\infty}^t g(u,t-s)L(ds).
\end{align}
Consider an $\R^n$-valued function $\varphi\in\mathcal{L}_\infty^{p,q}(\alpha)$, where $p\geq1$ and for some $\Delta>0$ the infinite memory vector $Z_u(t)=\left(X_u(t+\Delta),X_u(t),X_u(t-\Delta),\ldots\right)$. Then, the process $\varphi(Z_u(t))$ is $\theta$-weakly dependent with $\theta$-coefficients satisfying 
\begin{align}\label{equation:thetacoefficientinequality}
\begin{aligned}
q&=2:\quad\theta_{\varphi(Z_u)}(h)\leq C\sum_{k=\left\lfloor\frac{h}{2\Delta} \right\rfloor}^\infty \alpha_k+C\left(\Sigma_L\int_{-\infty}^{-\frac{h}{2}}g(u,-s)^2ds\right)^{\frac{1}{2}} =\hat{\theta}_{\varphi(Z_u)}(h)\\
q&=4:\quad \theta_{\varphi(Z_u)}(h)\leq C\sum_{k=\left\lfloor\frac{h}{2\Delta} \right\rfloor}^\infty \alpha_k+C\Bigg(\int_{-\infty}^{-\frac{h}{2}}g(u,-s)^4ds\left(\int_{\R}x^4\nu(dx)\right)\\
&\qquad\qquad\qquad\qquad\qquad\qquad\quad\ \ \ + 3\Sigma_L^2\bigg(\int_{-\infty}^{-\frac{h}{2}}g(u,-s)^2ds\bigg)^2\Bigg)^{\frac{1}{4}} =\hat{\theta}_{\varphi(Z_u)}(h)
\end{aligned}
\end{align}
for a constant $C\geq0$ and all $h\geq1$.
\end{Proposition}

\begin{proof}
See Section \ref{sec7-3}.
\end{proof}

\begin{Corollary}\label{corollary:thetaweakdependenceinfinitememory}
Consider $Y_N$, $\tilde{Y}_u$ and $\Phi$ as in Theorem \ref{theorem:consistency} for $q=2$ where $\Phi$ is of infinite memory. Then, condition (c) from Theorem \ref{theorem:consistency} is implied by 
\begin{enumerate}
\item[(c$^{**}$)] $\tilde{Y}_u$ satisfies the conditions of Proposition \ref{proposition:infinitememorymovingaverage}.
\end{enumerate}
\end{Corollary}

\begin{Corollary}
Consider $Y_N$, $\tilde{Y}_u$ and $\Phi$ as in Theorem \ref{theorem:asymptoticnormality}, where $q,\bar{q}=2$, $\tilde{q}=4$ and $\Phi$ is of infinite memory. Then, the conditions (e) and (g) are respectively implied by
\begin{enumerate}
\item[(e$^{**}$)] $\nabla_\vartheta \Phi(t)\in L^{2+\gamma_1}$ for some $\gamma_1>0$ and $\tilde{Y}_u$ satisfies the conditions of Proposition \ref{proposition:infinitememorymovingaverage}. Moreover, $\hat{\theta}_{\nabla_\vartheta \Phi(t)}(h)$ from (\ref{equation:thetacoefficientinequality}) satisfies \hyperlink{DD}{DD($\gamma_1$)} and
\item[(g$^{**}$)] $\tilde{Y}_u$ satisfies the conditions of Proposition \ref{proposition:infinitememorymovingaverage}.
\end{enumerate}
\end{Corollary}

%

\subsection{Time-varying L\'evy-driven Ornstein-Uhlenbeck processes}
\label{sec4-2}
We consider a sequence of time-varying L\'evy-driven Ornstein-Uhlenbeck processes 
\begin{align}\label{eq:tvcar}
\begin{aligned}
Y_N(t)&=\int_{-\infty}^{\infty}g_N(Nt,Nt-u)L(du), \text{ with kernel function}\\
g_N(Nt,Nt-u)&=\mathbb{1}_{\{Nt-u\geq0\}} e^{-\int_{u}^{Nt}a\left(\frac{s}{N}\right)ds}= \mathbb{1}_{\{Nt-u\geq0\}} e^{-\int_{-(Nt-u)}^0a\left(\frac{s+Nt}{N}\right)ds},
\end{aligned}
\end{align}
where $a:\R\rightarrow\R_0^+$ is continuous such that $u\mapsto e^{-\int_{-u}^{0}a\left(\frac{s+Nt}{N} \right)ds}\in L^1(\R^+)$ for all $t\in\R$ and $N\in\N$, which ensures the existence of (\ref{eq:tvcar}), since additionally (\ref{eq:condlevyproc}) holds. In the next proposition we review sufficient conditions under which the sequence (\ref{eq:tvcar}) possesses a locally stationary approximation $\tilde{Y}_u$ for $p=2$ and $p=4$ given by 
\begin{align}\label{eq:locapproxcar}
\begin{aligned}
\tilde{Y}_{u}(t)&=\int_\R g(u,t-s)L(ds), \text{ with kernel function}\\
g(u,t-s)&=\mathbb{1}_{\{t-s\geq0\}} e^{-a(u)(t-s)}.
\end{aligned}
\end{align}

\begin{Remark}
The process $\tilde{Y}_{u}$ from (\ref{eq:locapproxcar}) is the unique stationary solution to the stochastic differential equation
\begin{align*}
d\tilde{Y}_{u}(t)=-a(u)\tilde{Y}_{u}(t)dt+L(dt).
\end{align*}
\end{Remark}

\begin{Proposition}[{\cite[Proposition 5.3]{SS2021}}]\label{proposition:car1locstat}
Let $Y_N$ be a sequence of time-varying L\'evy-driven Ornstein-Uhlenbeck processes as given in (\ref{eq:tvcar}) such that (\ref{eq:condlevyproc}) holds. Then, $\tilde{Y}_{u}$ as given in (\ref{eq:locapproxcar}) is a locally stationary approximation of $Y_N$ for $p=2$, if
\begin{enumerate}[label={(\alph*)}]
\item the coefficient function $a$ is Lipschitz with constant L,
\item $\inf_{s\in\R}a(s)>0$.
\end{enumerate}
If additionally $\int_{\abs{x}>1}x^4\nu(dx)<\infty$, then $\tilde{Y}_{u}$ is also a locally stationary approximation of $Y_N$ for $p=4$.
\end{Proposition}

\subsection{Least squares estimation}
\label{sec4-3}
Let us assume that we observe a sequence of time-varying L\'evy-driven Ornstein-Uhlenbeck processes $Y_N$ as defined in (\ref{eq:tvcar}), where the characteristic triplet of the driving L\'evy process $L$ is known and the observations are sampled according to Assumption \ref{assumption:observations} such that either \hyperref[observations:O1]{(O1)} or \hyperref[observations:O2]{(O2)} hold. \\
Our goal is to estimate the coefficient function at a fixed point $u>0$, i.e. to estimate $a(u)$. To this aim we assume that $a(u)\in\Theta\subset\R^+$, where $\Theta$ is a compact parameter space.\\
For $\vartheta\in\Theta$ and $\Delta>0$ we define the following least squares contrast
\begin{gather}\label{equation:tvCARleastsquaresestimator}
\Phi^{LS}\Big((\tilde{Y}_u(\Delta (1-k)))_{k\in\N_0},\vartheta\Big)=\Phi^{LS}\Big(\tilde{Y}_u(\Delta),\tilde{Y}_u(0),\vartheta\Big)=\Big(\tilde{Y}_u(\Delta)-e^{-\Delta\vartheta}\tilde{Y}_u(0)\Big)^2.
\end{gather}
We show consistency and asymptotic normality of the estimator $\hat{\vartheta}_N$ from (\ref{eq:M_N}), defined with respect to the least squares contrast (\ref{equation:tvCARleastsquaresestimator}), using Theorem \ref{theorem:consistency} and \ref{theorem:asymptoticnormality}.

\begin{Theorem}\label{theorem:asymptoticpropertiesleastsquares}
Let $Y_N$ be a sequence of time-varying L\'evy-driven Ornstein-Uhlenbeck processes as given in (\ref{eq:tvcar}) such that (\ref{eq:condlevyproc}) and either \hyperref[observations:O1]{(O1)} or \hyperref[observations:O2]{(O2)} hold. Assume that 
\begin{enumerate}[label={(\alph*)}]
\item $\gamma+\int_{\abs{x}>1}\nu(dx)=0$,
\item the parameter space $\Theta\subset\R^+$ is compact,
\item the coefficient function $a$ is Lipschitz and satisfies $\inf_{s\in\R}a(s)>0$ and
\item if \hyperref[observations:O2]{(O2)} holds, $\int_{\abs{x}>1}\abs{x}^{2+\gamma_1}\nu(dx)<\infty$ for some $\gamma_1>0$.
\end{enumerate}
Then, $\hat{\vartheta}_N$ is consistent, i.e. $\hat{\vartheta}_N\overset{P}{\longrightarrow}a(u)$ as $N\rightarrow\infty$.
Moreover, if 
\begin{enumerate}[label={(\alph*)}]
\item[(d)] $\int_{\abs{x}>1}x^{4+\gamma_2}\nu(dx)<\infty$ for some $\gamma_2>0$,
\item[(e)] $a(u)$ is in the interior of $\Theta$ and $\sqrt{m_N}b_N\rightarrow0$ as $N\rightarrow\infty$ and
\item[(f)] $\hat{\vartheta}_N$ is defined with respect to the rectangular kernel (\ref{equation:rectangularkernel}),
\end{enumerate}
then $\sqrt{\frac{b_N}{\delta_N}} \left(\hat{\vartheta}_N-a(u) \right)\overset{d}{\longrightarrow} \mathcal{N}\left(0, \Sigma(u)\right)$ as $N\rightarrow\infty$, where
\begin{gather}\label{eq:asymptoticvarianceleastsquarestvCAR1}
\Sigma(u)=\frac{1}{2\Delta^2}\begin{cases} e^{2a(u)\Delta}+2e^{2a(u)\Delta}e^{-2a(u)\delta}\frac{1-e^{-2a(u)\delta(\left\lceil\Delta/\delta \right\rceil-1)}}{1-e^{-2a(u)\delta}}-2\left\lceil\Delta/\delta \right\rceil+1, &\text{if \hyperref[observations:O1]{(O1)} holds},\\
e^{2a(u)\Delta}-1,&\text{if \hyperref[observations:O2]{(O2)} holds}.
\end{cases}
\end{gather}
Remind that $\Delta$ denotes the step size in the contrast function (\ref{equation:contrastfunctioninfinitememory}) and $\delta=N\delta_N$ for the observation scheme \hyperref[observations:O1]{(O1)}. In particular, if $\delta=\Delta$, the asymptotic variance is given by $\Sigma(u)= \frac{1}{2\Delta^2}(e^{2a(u)\Delta}-1)$, independent of whether \hyperref[observations:O1]{(O1)} or \hyperref[observations:O2]{(O2)} holds.
\end{Theorem}
\begin{proof}
See Section \ref{sec7-4}.
\end{proof}

\begin{Remark}\label{remark:consistentestimator}
A consistent plug-in estimator of $\Sigma(u)$ can be readily obtained by replacing $a(u)$ in (\ref{eq:asymptoticvarianceleastsquarestvCAR1}) with $\hat\vartheta_N$.
\end{Remark}

\section{Quasi-maximum likelihood and Whittle estimation for time-varying L\'evy-driven state space models}
\label{sec5}
We consider observations of a time-varying L\'evy-driven state space model that follow the sampling scheme from Assumption \ref{assumption:observations} such that either \hyperref[observations:O1]{(O1)} or \hyperref[observations:O2]{(O2)} hold. In this section, we derive consistency results for an $M$-estimator that is based on a log-likelihood contrast. In addition, we establish consistency for a localized Whittle estimator. 

\subsection{Time-varying L\'evy-driven state space models}
\label{sec5-1}
We give a brief review of the definition and basic properties of time-varying L\'evy-driven state space models. For further details we refer to \cite{BSS2021,SS2021}.\\
Let $L=\{L(t),t\in\R\}$ be a two-sided L\'evy process with values in $\R$ satisfying (\ref{eq:condlevyproc}). For $p\in\N$ and arbitrary continuous coefficient functions $A(t)\in M_{p\times p}(\R)$ and $B(t),C(t)\in M_{p\times1}(\R)$, $t\in\R$ we consider the 
observation and state equation
\begin{align}\label{eq:tvLDstatespace}
\begin{aligned}
Y(t)=B(t)'X(t) \text{ and}\qquad
dX(t)=A(t)X(t)dt+C(t)L(dt).
\end{aligned}
\end{align}
The solution of (\ref{eq:tvLDstatespace}) is unique and given by (see \cite[Section 4]{BSS2021})
\begin{align*}
X(t)=\int_{-\infty}^t\Psi(t,s)C(s)L(ds)\text{ and }Y(t)=B(t)\int_{-\infty}^t\Psi(t,s)C(s)L(ds),
\end{align*} 
provided that the integrals exist in $L^2$. The matrix $\Psi(t,t_0)$ for $t>t_0$ is the unique solution to the homogeneous initial value problem (IVP)
\begin{align}\label{eq:ivpA}
\begin{aligned}
\frac{d}{dt}\Psi(t,t_0)&=A(t)\Psi(t,t_0),\text{ with initial condition}\quad\Psi(t_0,t_0)&=\mathbf{1}_p. 
\end{aligned}
\end{align}
A comprehensive discussion on the IVP (\ref{eq:ivpA}) can be found in \cite[Section 3 and 4]{B1970} and in the context of L\'evy-driven state space models in \cite[Section 5.2 and 5.3]{SS2021}. 

\begin{Definition}
Let $X=\{X(t),t\in\R\}$ be a solution to the state space representation (\ref{eq:tvLDstatespace}). Then, we call $X$ a time-varying L\'evy-driven state space process. If the coefficient functions $A,B$ and $C$ are time-invariant, the solution is called a L\'evy-driven state space process.
\end{Definition}

\begin{Remark}
Noticeable examples from the class of time-varying L\'evy-driven state space models are time-varying L\'evy-driven CARMA processes, for which the matrix function $C(t)$ is time-invariant, i.e. $C(t)=C$ for all $t\in\R$ and the matrix function $A(t)$ is of the form
\begin{align*}
A(t)&=\left( \begin{array}{cccc}
0 & 1 & \ldots & 0 \\
\vdots & ~ & \ddots & \vdots \\
0 & ~ & ~ & 1 \\
-a_p(t) & -a_{p-1}(t) & \ldots & -a_1(t) \\
\end{array}\right),
\end{align*} 
where $a_i(t)$, $i=1,\ldots,p$ are continuous real functions. 
\end{Remark}

Now, for continuous coefficient functions $A(t)\in M_{p\times p}(\R)$ and $B(t),C(t)\in M_{p\times1}(\R)$, $t\in\R$ and a two-sided L\'evy process $L$, we consider a sequence $Y_N$ of time-varying L\'evy-driven state space models, where
\begin{align}\label{eq:seqtvLDstatespacesolution}
\begin{aligned}
Y_N(t)&=\int_\R g_N(Nt,Nt-s)L(ds), \text{ with kernel function}\\
g_N(Nt,Nt-s)&= \mathbb{1}_{\{Nt-s\geq0\}} B(t)' \Psi_{N,t}(0,-(Nt-s))C\left(\frac{-(Nt-s)}{N}+t \right),
\end{aligned}
\end{align}
where $t\in\R$. Then, $\Psi_{N,t}(0,-(Nt-u))$ is the solution to the IVP
\begin{gather*}
\frac{d}{ds}\Psi_{N,t}(s,s_0)=A\left(\frac{s}{N}+t\right)\Psi_{N,t}(s,s_0),\text{ with initial condition}\quad\Psi_{N,t}(s_0,s_0)= \mathbf{1}_p
\end{gather*}
for $s>s_0$. 
We assume that $A(u)$, $u\in\R^+$ has eigenvalues with strictly negative real part and 
consider as corresponding locally stationary approximation $\tilde{Y}_{u}$ the process
\begin{align}\label{eq:seqlimLDstatespacesolution}
\begin{aligned}
\tilde{Y}_{u}(t)&= \int_\R g(u,t-s)L(ds), \text{ with kernel function}\\
g(u,t-s)&= \mathbb{1}_{\{t-s\geq0\}} B(u)' e^{A(u)(t-s)}C\left(u \right).
\end{aligned}
\end{align}

\begin{Proposition}[{\cite[Corollary 5.13]{SS2021}}]\label{proposition:tvLDSPhavealocstatapprox}
Let $Y_N$ be a sequence of time-varying L\'evy-driven state space models as given in (\ref{eq:seqtvLDstatespacesolution}). 
Then, $\tilde{Y}_{u}$ as given in (\ref{eq:seqlimLDstatespacesolution}) is a locally stationary approximation of $Y_N$ for $p=2$, if 
\begin{enumerate}[label={(\alph*)}]
\item the coefficient functions $A,B$ and $C$ are Lipschitz with constants $L_A,L_B$ and $L_C$,
\item $\sup_{s\in\R}\norm{B(s)}<\infty$ and $\sup_{s\in\R}\norm{C(s)}<\infty$,
\end{enumerate}
and either (c1)-(e1) or (c1), (d2) and (e2) hold, where
\begin{enumerate}[label={(\alph*)}]
\item[(c1)] $\{A(t)\}_{t\in\R}$ commutes, i.e. $[A(t),A(s)]=0$ for all $s,t\in\R$,
\item[(d1)] $\norm{e^{\nu\int_s^0A(\frac{\tau}{N}+t)d\tau}}\leq \gamma e^{\nu s \lambda}$, with $\gamma,\lambda>0$ for all $\nu\in[0,1]$, $s<0$, $t\in\R$ and $N\in\N$ and
\item[(e1)] $\norm{e^{-\nu A(t)s}}\leq \tilde\gamma e^{\nu s \tilde\lambda}$, with $\tilde\gamma,\tilde\lambda>0$ for all $\nu\in[0,1]$, $s<0$ and $t\in\R$,
\item[(d2)] the eigenvalues $\lambda_j(t)$ of $A(t)$ for $j=1,\ldots,p$ satisfy $\sup_{t\in\R}\max_{j=1,\ldots,p}\mathfrak{Re}(\lambda_j(t))<0$.
\item[(e2)] $A(t)$ is diagonalizable for all $t\in\R$.
\end{enumerate}
If we additionally assume that $\int_{\R}x^4\nu(dx)<\infty$
, then $\tilde{Y}_{u}$ is also a locally stationary approximation of $Y_N$ for $p=4$.
\end{Proposition}

\subsection{Quasi-maximum likelihood estimation}
\label{sec5-2}
Let $\Theta\subset\R^d$ be a compact parameter space and $\vartheta^*=\{\vartheta^*(t), t\in\R\}$ a parameter curve in $\Theta$. Moreover, let $(A_{\vartheta})_{\vartheta\in\Theta}\subset M_{p\times p}(\R)$ and $(B_\vartheta)_{\vartheta\in\Theta},(C_\vartheta)_{\vartheta\in\Theta}\subset M_{p\times 1}(\R)$ be families of matrices, $(\Sigma_\vartheta)_{\vartheta\in\Theta}$ a family of positive numbers, and $(L_\vartheta)_{\vartheta\in\Theta}$ a family of L\'evy processes such that $Var(L_\vartheta(1))=\Sigma_\vartheta$.\\
Consider a sequence of time-varying L\'evy-driven state space models $(Y_N^{\vartheta^*}(t))_{N\in\N}$ such that $Y_N^{\vartheta^*}(t)$ is defined as in (\ref{eq:seqtvLDstatespacesolution}) with coefficient functions $A(t)=A_{\vartheta^*(t)},B(t)=B_{\vartheta^*(t)}$, $C(t)=C_{\vartheta^*(t)}$ and driving noise $L=L_{\vartheta^*(0)}$, such that $Var(L(1))=\Sigma_{\vartheta^*(0)}$.\\ 
Based on the families $A_\vartheta,B_\vartheta$, $C_\vartheta$ and $L_\vartheta$ from above we define a family of processes $(\tilde{Y}^\vartheta(t))_{\vartheta\in\Theta}$, where
\begin{gather}\label{equation:statparameterized}
\tilde{Y}^\vartheta(t)=\int_{-\infty}^t B_\vartheta' e^{A_\vartheta(t-s)}C_\vartheta  L_\vartheta(ds).
\end{gather}

The following assumption is crucial for all results that we derive in the sequel.

\begin{manualassumption}{(C0)}\label{assumption:C0}
We assume that $\tilde{Y}_u(t)=\tilde{Y}^{\vartheta^*(u)}(t)$ is a locally stationary approximation of $Y_N^{\vartheta^*}(t)$ for some $p\geq1$.
\end{manualassumption}

To obtain a consistent estimator for $\vartheta^*(u)$, $u\in\R^+$, we consider an $M$-estimator of a log-likelihood contrast $\Phi^{LL}$ and use results from Section \ref{sec3}. We derive the contrast $\Phi^{LL}$ using results from \cite{SS2012b}, where the authors investigated a related estimator in a stationary setting.

\begin{manualassumption}{(C1)}\label{assumption:C1}
For each $\vartheta\in\Theta$, it holds $E[L_\vartheta(1)]=0$ and $E[L_\vartheta(1)^2]=\Sigma_\vartheta<\infty$. Additionally, we rule out the degenerate case, where $E[L_\vartheta(1)^2]=0$.
\end{manualassumption}

\begin{manualassumption}{(C2)}\label{assumption:C2}
For each $\vartheta\in\Theta$, the eigenvalues of $A_\vartheta$ have strictly negative real parts.
\end{manualassumption}

Under the previous assumptions $\tilde{Y}^\vartheta$ is for all $\vartheta\in\Theta$ the unique stationary centered solution to the observation and state equation 
\begin{align}
\begin{aligned}\label{equation:sampledprocessY_u}
\tilde{Y}^\vartheta(t)&=B_\vartheta' X(t)\text{ and}\qquad
dX(t)=A_\vartheta X(t)dt+C_\vartheta L_\vartheta(dt),
\end{aligned}
\end{align}
i.e. a L\'evy-driven state space process.
%
Moreover, for $\Delta>0$ it follows from \cite[Proposition 3.6]{SS2012b} that the sampled process $(\tilde{Y}^{\vartheta,\Delta}(k))_{k\in\Z}$, where $\tilde{Y}^{\vartheta,\Delta}(k)=\tilde{Y}^\vartheta(\Delta k)$ satisfies the state space representation
%
\begin{align}
\begin{aligned}\label{equation:statespacerepresentationsampled}
\tilde{Y}^{\vartheta,\Delta}(k)=B_\vartheta'X(k) \text{ and}\qquad 
X(k)&=e^{\Delta A_\vartheta}X(k-1)+N^{(\Delta)}_\vartheta(k),
\end{aligned}
\end{align}
where $N^{(\Delta)}_\vartheta(k)=\int_{(k-1)\Delta}^{k\Delta}e^{A_\vartheta(k\Delta-s)}C_\vartheta L_\vartheta(ds)$, $k\in\Z$. The sequence $(N^{(\Delta)}_\vartheta(k))_{k\in\Z}$ is i.i.d with mean zero and covariance matrix
\begin{gather}\label{eq:variancesampledprocess}
\cancel{\Sigma}^{(\Delta)}_\vartheta=\Sigma_\vartheta \int_0^\Delta e^{A_\vartheta s}C_\vartheta C_\vartheta' e^{A_\vartheta's} ds.
\end{gather}
Moreover, the spectral density of $\tilde{Y}^{\vartheta,\Delta}$, denoted by $f_{\tilde{Y}}^{(\Delta)}(\omega,\vartheta)$, is given by
\begin{gather}\label{eq:spectraldensitysampledprocess}
f_{\tilde{Y}}^{(\Delta)}(\omega,\vartheta)=\frac{1}{2\pi}B_\vartheta'(e^{i\omega}\mathbf{1}_p-e^{\Delta A_\vartheta})^{-1} \cancel{\Sigma}^{\ (\Delta)}_\vartheta(e^{-i\omega}\mathbf{1}_p-e^{\Delta A_\vartheta'})B_\vartheta, \quad \omega\in[-\pi,\pi].
\end{gather}

Next, we review some aspects of Kalman filtering that are necessary to define the log-likelihood contrast $\Phi^{LL}$.

\begin{Proposition}[{\cite[Proposition 2.1]{SS2012b}}]\label{proposition:kalmanfilter}
Let $Y=\{Y_n,n\in\Z\}$ be the output process of the state space model $Y_n=B'X_n$ and $X_n=A X_{n-1}+Z_{n-1}$, $n\in\Z$, where $A\in M_{p\times p}(\R)$, $B\in M_{p\times 1}(\R)$ and $Z_n$ is an $\R^p$-valued centered i.i.d. sequence with covariance matrix $Q$. The linear innovations $\varepsilon=(\varepsilon_n)_{n\in\Z}$ of $Y$ are defined as $\varepsilon_n=Y_n-P_{n-1}Y_n$, where $P_n$ is the orthogonal projection onto $\overline{\text{span}}\{Y_k:k\leq n\}$ and the closure is taken in $L^2$. If the eigenvalues of $A$ are less than $1$ in absolute value it holds:
\begin{enumerate}[label={(\alph*)}]	
\item The Ricatti equation $\Omega=A\Omega A'+Q-(A\Omega B)(B'\Omega B)^{-1}(A\Omega B)'$ has a unique positive semidefinite solution $\Omega\in M_{p\times p}(\R)$.
\item The eigenvalues of $A-KB'\in M_{p\times p}(\R)$ have absolute value less than one, where $K=(A\Omega B)(B'\Omega B)^{-1}\in M_{p\times 1}(\R)$ is the steady-state Kalman gain matrix.
\item The linear innovations $\varepsilon$ of $Y$ are the unique stationary solution to $\tilde{X}_n=(A-KB')\tilde{X}_{n-1}+KY_{n-1}$ and $\varepsilon_n=Y_n-B'\tilde{X}_n$, $n\in\Z$. Moreover, $\varepsilon_n$ can equivalently be written as
\begin{gather}\label{equation:varepsilon}
\varepsilon_n=Y_n-B'\sum_{\nu=1}^\infty (A-KB')^{\nu-1}KY_{n-\nu}.
\end{gather}
For the covariance matrix $V=E[\varepsilon_n^2]$ we have $V=B'\Omega B$.
\end{enumerate}
\end{Proposition}
 
\noindent We apply the above proposition to the state space model (\ref{equation:statespacerepresentationsampled}) and obtain the parametrized matrices $\Omega_\vartheta,K_\vartheta$ and $V_\vartheta$. In addition to the assumptions \hyperref[assumption:C1]{(C1)} and \hyperref[assumption:C2]{(C2)}, we impose the following conditions.

\begin{manualassumption}{(C3)}\label{assumption:C3}
The parameter space $\Theta\subset\R^d$ is compact.
\end{manualassumption}

\begin{manualassumption}{(C4)}\label{assumption:C4}
The functions $\vartheta\mapsto A_\vartheta$, $\vartheta\mapsto B_\vartheta$, $\vartheta\mapsto C_\vartheta$ and $\vartheta\mapsto \Sigma_\vartheta$ are continuously differentiable and $B_\vartheta\neq0$ for all $\vartheta\in\Theta$.
\end{manualassumption}

\begin{Lemma}
Assume that \hyperref[assumption:C1]{(C1)} - \hyperref[assumption:C4]{(C4)} hold. Then, $ e^{\Delta A_\vartheta}$ has eigenvalues with absolute value less than $1$ and $V_\vartheta=B_\vartheta'\Omega_\vartheta B_\vartheta>C_V$ for a constant $C_V>0$ and all $\vartheta\in\Theta$.
\end{Lemma}
\begin{proof}
Follows from the proof of \cite[Lemma 2.2 and Lemma 3.14]{SS2012b}.
\end{proof}
Following the sensitivity analysis of the Ricatti equation in \cite{S1998}, the degree of smoothness of $A_\vartheta,B_\vartheta,C_\vartheta$ and $ \Sigma_\vartheta$, namely $C^1$ carries over to the mapping $\vartheta\mapsto\Omega_\vartheta$. Therefore, from the previous assumptions it follows that the functions $\vartheta\mapsto A_\vartheta$, $\vartheta\mapsto B_\vartheta$, $\vartheta\mapsto C_\vartheta$, $\vartheta\mapsto \Sigma_\vartheta$, $\vartheta\mapsto K_\vartheta$ and $\vartheta\mapsto V_\vartheta$ are Lipschitz.\\

\noindent Under the conditions \hyperref[assumption:C1]{(C1)} - \hyperref[assumption:C4]{(C4)} we define for $u\in\R^+$, $\tilde{Y}^{\vartheta^*(u),\Delta}=(\tilde{Y}^{\vartheta^*(u),\Delta}(1-k))_{k\in\N_0}$ and $\tilde{Y}^{\vartheta^*(u),\Delta}(k)=\tilde{Y}^{\vartheta^*(u)}(\Delta k)$ as in (\ref{equation:statespacerepresentationsampled}) the log-likelihood contrast
\begin{gather}\label{equation:quasilogcontrast}
\Phi^{LL}\left( \tilde{Y}^{\vartheta^*(u),\Delta},\vartheta \right) =\log(2\pi)+\log(V_\vartheta)+\frac{\varepsilon_{\vartheta}^2\left(\tilde{Y}^{\vartheta^*(u),\Delta} \right)}{V_\vartheta},
\end{gather}
where $\varepsilon_{\vartheta}(\tilde{Y}^{\vartheta^*(u),\Delta})$ it given by the analogue of (\ref{equation:varepsilon}), i.e.
\begin{align}\label{equation:varepsilontheta}
\begin{aligned}
\varepsilon_{\vartheta}\left(\tilde{Y}^{\vartheta^*(u),\Delta}\right)&=\tilde{Y}^{\vartheta^*(u),\Delta}(1)-B_\vartheta'\sum_{n=1}^\infty \left( e^{\Delta A_\vartheta}-K_\vartheta B_\vartheta' \right)^{n-1}K_\vartheta \tilde{Y}^{\vartheta^*(u),\Delta}(1-n).
\end{aligned}
\end{align}
The localized $M$-estimator resulting from the contrast (\ref{equation:quasilogcontrast}) can be considered as a localized quasi-maximum likelihood estimator. 
It is given by
\begin{align}
\begin{aligned}\label{equation:thetahatqmle}
\hat{\vartheta}_N&=\argmin_{\vartheta\in\Theta} M_N(\vartheta), \text{ where}\\
M_N(\vartheta)&=\frac{\delta_N}{b_N}\sum_{i=-m_N}^{m_N}K\left(\frac{\tau_i^N-u}{b_N} \right)\Phi^{LL}\left(\left(Y_N^{\vartheta^*}\left(\tau_i^N+\frac{\Delta (1-k)}{N}\right)\right)_{k\in\N_0},\vartheta\right),
\end{aligned}
\end{align}
where $K$ is a localizing kernel. The following proposition shows that if \hyperref[assumption:C1]{(C1)} - \hyperref[assumption:C4]{(C4)} hold, $\Phi^{LL}$ satisfies all conditions of Theorem \ref{theorem:consistency} besides the identifiability condition \hyperref[assumption:M1]{(M1)}.

\begin{Proposition}\label{proposition:qmlecontrast}
If \hyperref[assumption:C1]{(C1)} - \hyperref[assumption:C4]{(C4)} hold, then
\begin{enumerate}[label={(\alph*)}]
\item $\Phi^{LL}(\cdot,\vartheta)\in \mathcal{L}_\infty^{1,2}(\alpha)$ for all $\vartheta\in\Theta$, where $\alpha=(\alpha_k)_{k\in\N_0}\subset\R^+$ satisfies $\sum_{k=0}^\infty k\alpha_k<\infty$,
\item $\Phi^{LL}(x,\cdot)\in\mathcal{L}_d(0,D_1(1+\sum_{k=0}^\infty \beta_kx_{1-k}^2))$ for all real sequences $x=(x_{1-k})_{k\in\N_0}$ and some $D_1\geq0$, where $(\beta_k)_{k\in\N_0}\subset\R$ satisfies $\sum_{k=0}^\infty k\beta_k<\infty$ and
\item if the observations follow the sampling scheme \hyperref[observations:O2]{(O2)}, condition (c) of Theorem \ref{theorem:consistency} is satisfied.
\end{enumerate}
\end{Proposition}
\begin{proof}
See Section \ref{sec7-5}.
\end{proof}

The following conditions \hyperref[assumption:C5]{(C5)}-\hyperref[assumption:C7]{(C7)} will help to verify the identifiability condition \hyperref[assumption:M1]{(M1)}.\\
First, it is necessary to ensure that the sampled process $\tilde{Y}^{\vartheta,\Delta}$ is not the output process of any state space representation of lower dimension than $p$ for all $\vartheta\in\Theta$. The concept of minimality is suitable for this purpose.

\begin{Definition}
Let $H:\R\rightarrow\R$ be a rational function and $A\in M_{p\times p}(\R)$ and $B,C\in M_{p\times 1}(\R)$ such that $H(x)=B'(x \mathbf{1}_p-A)^{-1}C$. We then call the triplet $(A,B,C)$ an algebraic realization of $H$ of dimension $p$. If $(A,B,C)$ is an algebraic realization whose dimension is smaller than or equal to the dimension of any other algebraic realization of $H$ we call it minimal. The corresponding dimension of such a minimal algebraic realization is called McMillan degree.
\end{Definition}

\begin{manualassumption}{(C5)}\label{assumption:C5}
For each $\vartheta\in\Theta$ the triplet $(A_\vartheta,B_\vartheta,C_\vartheta)$ is minimal with McMillan degree $p$.
\end{manualassumption}

\begin{Definition}
An algebraic realization $(A,B,C)$ of dimension $p$ is called controllable if the matrix $[C~~ AC ~\hdots ~A^{p-1}C]\in\R_{p\times p}$ has full rank. Moreover, it is called observable if the matrix $[B~~ BA' ~\hdots~ B(A^{p-1})']\in\R_{p\times p}$ has full rank.
\end{Definition}

\begin{Proposition}\label{proposition:controllableobservableminimal}
An algebraic realization $(A,B,C)$ of dimension $p$ is controllable and observable if and only if it is minimal.
\end{Proposition}
\begin{proof}
See \cite[Theorem 2.3.3]{HD1988}.
\end{proof}

\begin{manualassumption}{(C6)}\label{assumption:C6}
Let $(\tilde{Y}^\vartheta)_{\vartheta\in\Theta}$ be the family of output processes of the observation and state equation (\ref{equation:sampledprocessY_u}). For all $\vartheta_1\neq\vartheta_2$ the spectral densities of the two processes $\tilde{Y}^{\vartheta_1}$ and $\tilde{Y}^{\vartheta_2}$ are different.
\end{manualassumption}

\begin{Proposition}
For $\vartheta\in\Theta$ let $\tilde{Y}^\vartheta$ be the output process of (\ref{equation:sampledprocessY_u}). Then, its spectral density $f_{\tilde{Y}^\vartheta}$ is given by
\begin{gather*}
f_{\tilde{Y}^\vartheta}(\omega)=\frac{1}{2\pi}H_{\vartheta}(i\omega)\Sigma_\vartheta H_\vartheta(-i\omega), ~ \omega\in\R,
\end{gather*}
where $H_\vartheta:\R\rightarrow\R$ is the transfer function of $\tilde{Y}^\vartheta$ and defined as $H_\vartheta(x)=B_\vartheta'(x \mathbf{1}_p-A_\vartheta)^{-1}C_\vartheta$.
\end{Proposition}
\begin{proof}
See \cite[Proposition 3.4]{SS2012b}.
\end{proof}

Under the following assumption one can show that \hyperref[assumption:C6]{(C6)} also holds for the sampled process $\tilde{Y}^{\vartheta,\Delta}$.

\begin{manualassumption}{(C7)}\label{assumption:C7}
For each $\vartheta\in\Theta$ the spectrum of $A_\vartheta$ is a subset of $\{z\in\C, \frac{-\pi}{\Delta}< \mathfrak{Im}(z) < \frac{\pi}{\Delta}\}$.
\end{manualassumption}

\begin{Proposition}\label{proposition:qmlem1}
Let $\Phi^{LL}$ be as defined in (\ref{equation:quasilogcontrast}) and assume that \hyperref[assumption:C1]{(C1)} - \hyperref[assumption:C7]{(C7)} hold. Then, \hyperref[assumption:M1]{(M1)} holds.
\end{Proposition}
\begin{proof}
Follows from \cite[Lemma 2.10 and the proof of Theorem 3.16]{SS2012b}.
\end{proof}

\begin{Theorem}\label{theorem:qmleconsistent}
Assume that \hyperref[assumption:C0]{(C0)} is satisfied for p=2 and that either \hyperref[observations:O1]{(O1)} or \hyperref[observations:O2]{(O2)} holds. If \hyperref[assumption:C1]{(C1)} - \hyperref[assumption:C7]{(C7)} hold, then $\hat{\vartheta}_N\!\overset{P}{\longrightarrow}\vartheta^*(u)$ as $N\rightarrow\infty$ for all $u\in\R^+$ with $\hat{\vartheta}_N$ as defined in (\ref{equation:thetahatqmle}).
\end{Theorem}
\begin{proof}
According to Proposition \ref{proposition:qmlecontrast} and \ref{proposition:qmlem1} the contrast $\Phi^{LL}$ as defined in (\ref{equation:quasilogcontrast}) satisfies the conditions of Theorem \ref{theorem:consistency}. 
\end{proof}

\begin{Remark}
Theorem \ref{theorem:qmleconsistent} provides an important first result for the statistical inference of time-varying L\'evy-driven state space models, specifically including time-varying CARMA processes.\\
In addition, it is desirable to have results on the estimator's (asymptotic) distribution at hand to construct for instance confidence intervals. In fact, Theorem \ref{theorem:asymptoticnormality} should pave the way to prove asymptotic normality of $\hat{\vartheta}_N$ as defined in (\ref{equation:thetahatqmle}). However, this requires an in-depth analysis of the regularity properties of the contrast function's first and second order partial derivatives, which is beyond the scope of this work.
\end{Remark}

\begin{Remark}\label{remark:qmleconsistent}
If \hyperref[observations:O1]{(O1)} holds for $N\delta_N=\Delta$, the estimator $\hat{\vartheta}_N$ is given by the simpler expression
\begin{gather*}
\hat{\vartheta}_N=\argmin_{\vartheta\in\Theta}\frac{\delta_N}{b_N}\sum_{i=-m_N}^{m_N}K\left(\frac{i\delta_N}{b_N} \right)\Phi^{LL}\left(\left(Y_N^{\vartheta^*}(\tau_{i+1-k})\right)_{k\in\N_0},\vartheta\right).
\end{gather*}
\end{Remark}

In the next example we briefly present an estimation setting that satisfies the conditions \hyperref[assumption:C0]{(C0)}-\hyperref[assumption:C7]{(C7)}.

\begin{Example}\label{example:timevaryingstatespacemodel}
Assume that $\Theta\subset\R^3$ is a properly restricted compact parameter space. For $\vartheta=(\vartheta_1,\vartheta_2,\vartheta_3)\in\Theta$ we consider the following families of matrices and real numbers
\begin{align*}
A_\vartheta=\left(\begin{array}{cc}\vartheta_1 & 0 \\0 & \vartheta_2 \\\end{array}\right),~
B_\vartheta=\left(\begin{array}{c}\frac{1}{\vartheta_2-\vartheta_1} \\ \frac{-1}{\vartheta_2-\vartheta_1}\\ \end{array}\right),~
C_\vartheta=\left(\begin{array}{c}-\vartheta_1(1+\vartheta_2) \\ -\vartheta_2(1+\vartheta_1)\\ \end{array}\right) \text{ and }\Sigma_\vartheta=\vartheta_3>0.
\end{align*}
Moreover, $L_\vartheta$ denotes a family of two sided L\'evy process that satisfies \hyperref[assumption:C1]{(C1)} for all $\vartheta\in\Theta$. 
For a properly restricted compact parameter space $\mathcal{T}\subset\R^5$, we define the family $R_\Theta=\{\tilde{\vartheta}(t)=\left(\tau_1+\tau_2|\sin(t)|,\tau_1+\tau_3|\cos(x)|,\tau_4\right) \}_{\tau=(\tau_1,\tau_2,\tau_3,\tau_4)\in\mathcal{T}}$ of curves in $\Theta$, such that $(\tilde{\vartheta}(t))_{t\in\R}\subset\Theta$ for all $\tilde{\vartheta}\in R_\Theta$.\\
Following Remark \ref{remark:qmleconsistent}, we define the family $(Y_N^{\tilde{\vartheta}}(t))_{\tilde{\vartheta}\in R_\Theta}$ of sequences of time-varying L\'evy-driven state space models such that $Y_N^{\tilde{\vartheta}}(t)$ is defined as in (\ref{eq:seqtvLDstatespacesolution}) with coefficient functions $A(t)=A_{\tilde{\vartheta}(t)},~B(t)=B_{\tilde{\vartheta}(t)},~C(t)=C_{\tilde{\vartheta}(t)}\text{ and driving noise }L_{\tilde{\vartheta}(t)},~\tilde{\vartheta}\in R_\Theta$. All of the above functions are uniformly bounded and Lipschitz in $t$ for all $\tau\in\mathcal{T}$. In addition, it holds $[A(t),A(s)]=0$ for all $s,t\in\R$ and $\tau\in\mathcal{T}$ (see e.g. \cite{WS1976}). All eigenvalues of $A(t)$ are real and strictly negative for a properly restricted $\mathcal{T}$. Overall, (a), (b), (c1), (d2) and (e2) of Proposition \ref{proposition:tvLDSPhavealocstatapprox} hold and $\tilde{Y}^{\tilde{\vartheta}(u)}(t)$ as given in (\ref{equation:statparameterized}) is a locally stationary approximation of $Y^{\tilde{\vartheta}}_N(t)$ for $p=2$ and all $\tau\in\mathcal{T}$. Now, assume that we observed $Y^{\vartheta^*}_N(t)$ for some true parameter curve $\vartheta^*(t)\in R_\Theta$. Then, \hyperref[assumption:C0]{(C0)} holds.\\
By construction, the functions $\vartheta\mapsto A_\vartheta$, $\vartheta\mapsto B_\vartheta$, $\vartheta\mapsto C_\vartheta$ and $\vartheta\mapsto \Sigma_\vartheta$ are continuously differentiable. It is easy to see that the triplet $(A_\vartheta,B_\vartheta,C_\vartheta)$ is controllable and observable and therefore, according to Proposition \ref{proposition:controllableobservableminimal} also minimal for all $\vartheta\in\Theta$. Overall, \hyperref[assumption:C0]{(C0)}-\hyperref[assumption:C5]{(C5)} and \hyperref[assumption:C7]{(C7)} hold. In view of \hyperref[assumption:C6]{(C6)} it is enough to observe that $f_{\tilde{Y}^\vartheta}(\omega)=\frac{\vartheta_3}{2\pi}\frac{(\omega-i \vartheta_1\vartheta_2)(\omega+i \vartheta_1\vartheta_2)}{(\omega+i \vartheta_1)(\omega+i\vartheta_2)(\omega-i \vartheta_1)(\omega-i\vartheta_2)}$, which satisfies \hyperref[assumption:C6]{(C6)} whenever $\vartheta_1\neq\vartheta_2$, $\vartheta_i\neq\overline\vartheta_i$ and $\mathfrak{Re}(\vartheta_i)<0$ for $i=1,2$.
Finally, for all $u\in\R^+$ Theorem \ref{theorem:qmleconsistent} implies that $\hat{\vartheta}_N$ is a consistent estimator of $\vartheta^*(u)$. If the true parameter curve $\vartheta^*$ can be estimated at $0<u_1<u_2$, one can solve a system of equations to obtain estimators for $\tau$, such that the whole curve $\vartheta^*$ can be reconstructed.\\
It is interesting to note that \cite[Theorem 3.3]{SS2012a} immediately shows that $\tilde{Y}^{\vartheta}$ is a CARMA(2,1) process with AR polynomial $p(z)=(z-\vartheta_1)(z-\vartheta_2)$ and 
MA polynomial $q(z)=z-\vartheta_1\vartheta_2$.\\
We note that the form of $B$ and $C$ ensures the transfer function of $\tilde{Y}^{\vartheta}$ to be properly normed and the AR and MA polynomials to be monic, which helps to obtain the identifiability condition from the spectral density.

\end{Example}

\subsection{Truncated quasi-maximum likelihood estimation}
\label{sec5-3}
We consider observations as described in Assumption \ref{assumption:observations} that follow the sampling scheme \hyperref[observations:O1]{(O1)} for $N\delta_N=\Delta$.\\
It is clear that in practice one does not observe the full history of $Y_N^{\vartheta^*}$ as assumed in (\ref{equation:thetahatqmle}). As unobserved sampling points must not contribute to the estimator, we set $Y_N^{\vartheta^*}(\tau)$ to $0$ if $\tau$ is not included in the observation window $[u-b_N,u+b_N]$.
Thus, for $Y_N^{\vartheta^*}=(Y_N^{\vartheta^*}(\tau_{i+1-k}))_{k\in\N_0}$ we define
\begin{align*}
\tilde\Phi^{LL}_{i,m_N}\left(Y_N^{\vartheta^*},\vartheta\right)&=\bigg(\log(2\pi)+\log(V_\vartheta)+\frac{\tilde{\varepsilon}_{\vartheta,i,m_N}^2\big(Y_N^{\vartheta^*}\big)}{V_\vartheta}\bigg)
\text{ and}\\
\tilde{\varepsilon}_{\vartheta,i,m_N}\left(Y_N^{\vartheta^*}\right)&=Y_N^{\vartheta^*}\left(\tau_{i+1}^N\right)-B_\vartheta'\sum_{n=1}^{m_N+i+1} \left(e^{\Delta A_\vartheta}-K_\vartheta B_\vartheta' \right)^{n-1}K_\vartheta Y_N^{\vartheta^*}\left(\tau_{i+1-n}^N\right),i\in\Z.
\end{align*}
This leads to the truncated estimator
\begin{align}\label{equation:modifiedestimator}
\begin{aligned}
\hat{\vartheta}^{mod}_N&=\argmin_{\vartheta\in\Theta}\tilde{M}_N(\vartheta),\\
\end{aligned}
\end{align}
where $\tilde{M}_N(\vartheta)=\frac{\delta_N}{b_N}\sum_{i=-m_N}^{m_N-1}K\left(\frac{\tau_i^N-u}{b_N} \right)\tilde\Phi^{LL}_{i,m_N}\left(Y_N^{\vartheta^*},\vartheta\right)$. The following theorem extends the consistency result from Theorem \ref{theorem:qmleconsistent} to the truncated estimator $\hat{\vartheta}^{mod}_N$.

\begin{Theorem}\label{equation:consistencymodifiedQMLE}
Assume that \hyperref[assumption:C0]{(C0)} is satisfied for p=2. If \hyperref[observations:O1]{(O1)} as well as \hyperref[assumption:C1]{(C1)} - \hyperref[assumption:C7]{(C7)} hold, then $\hat{\vartheta}^{mod}_N\overset{P}{\longrightarrow}\vartheta^*(u)$ as $N\rightarrow\infty$ for all $u\in\R^+$ with $\hat{\vartheta}^{mod}_N$ as defined in (\ref{equation:modifiedestimator}).
\end{Theorem}
\begin{proof}
See Section \ref{sec7-6}.
\end{proof}

\subsection{Whittle estimation}
\label{sec5-4}

In this section, we investigate under the same setting as in Section \ref{sec5-2} a localized Whittle estimator for $\vartheta^*$. Before we prove consistency of this estimator, we briefly review the Whittle estimator in a stationary setting \cite{FHM2020}.\\
Let $X=\{X(t),t\in\R\}$ be a real-valued centered square integrable L\'evy-driven state space model given by $X(t)=\int_{-\infty}^t B' e^{A(t-s)}CL(ds)$, where $A\in M_{p\times p}(\R)$ and $B,C\in M_{p\times 1}(\R)$. 
For some $\Delta>0$ we consider the sampled process $X^{\Delta}=\{ X^{\Delta}(k), k\in \N_0\}$, where $X^{\Delta}(k)=X(\Delta k)$. 
The spectral density $f_X^{(\Delta)}$ is defined as the Fourier transform of the autocovariance function $\Gamma^{(\Delta)}_X(h)=E[X^\Delta(h)X^\Delta(0)]$, $h\in\Z$, i.e.
\begin{gather*}
f_X^{(\Delta)}(\omega)=\frac{1}{2\pi}\sum_{h\in\Z}\Gamma^{(\Delta)}_X(h)e^{-ih\omega}, \qquad \omega\in[-\pi,\pi],
\end{gather*}
and conversely, using the inverse Fourier transform, $\Gamma_X^{(\Delta)}(h)=\int_{-\pi}^\pi f_X^{(\Delta)}(\omega) e^{ih\omega}$, $h\in\Z$, where we make the additional convention that $\Gamma^{(\Delta)}_X(-h)=\Gamma^{(\Delta)}_X(h)$. Based on the sample autocovariance $\overline{\Gamma}_n(h)=\frac{1}{n}\sum_{k=1}^{n-h}X^\Delta(k+h)X^\Delta(k)$, $h\in\Z$, where $\overline{\Gamma}_n(-h):=\overline{\Gamma}_n(h)$, we define the periodogram $I_n:[-\pi,\pi]\rightarrow [0,\infty)$ as
\begin{gather}\label{eq:periodogram}
I_n(\omega)=\frac{1}{2\pi n} \left( \sum_{j=1}^n X^{\Delta}(j)e^{-ij\omega} \right)\left( \sum_{k=1}^n X^{\Delta}(k)e^{ik\omega} \right)=\frac{1}{2\pi} \sum_{h=-n+1}^{n-1} \overline{\Gamma}_n(h)e^{-ih\omega},\ \omega\in [-\pi,\pi].
\end{gather}
The periodogram $I_n(\omega)$ can be considered as the empirical version of the spectral density and is the main part of the Whittle estimator. \\
Now, let $\Theta\subset\R^d$ be a compact parameter space. For $\vartheta\in\Theta$ let $X^\vartheta$ be a real valued centered square integrable state space model of the form (\ref{equation:statparameterized}) and $f_X^{(\Delta)}(\omega,\vartheta)$ the corresponding spectral density. In this stationary setting the Whittle function is defined as
\begin{gather}\label{eq:whittlefunction}
W_n^{stat}(\vartheta)=\frac{1}{2n}\sum_{j=-n+1}^n\left(\frac{I_n(\omega_j)}{f_X^{(\Delta)}(\omega_j,\vartheta)}+\log\left(f_X^{(\Delta)}(\omega_j,\vartheta)\right)\right), \qquad \vartheta\in\Theta,
\end{gather}
where $\omega_j=\frac{\pi j}{n}$ for $j=1,\ldots,n$. Based on this Whittle function, the Whittle estimator is defined as $\hat{\vartheta}_n^{stat}=\argmin_{\vartheta\in\Theta}W_n^{stat}(\vartheta)$. For more information on the Whittle estimator including conditions that ensure consistency, we refer to \cite{FHM2020}.\\

Let $(Y_N^{\vartheta^*}(t))_{N\in\N}$ denote a sequence of time-varying L\'evy-driven state space models as considered in Section \ref{sec5-2} and $(\tilde{Y}^\vartheta(t))_{\vartheta\in\Theta}$ a family of L\'evy-driven state space models in the form of (\ref{equation:statparameterized}). Based on observations of $Y_N^{\vartheta^*}$, we now give a localized version of the Whittle estimator.\\
We fix $\Delta>0$ and assume that the available observations follow the sampling scheme introduced in Assumption \ref{assumption:observations} such that \hyperref[observations:O1]{(O1)} holds for $N\delta_N=\Delta$. 
For a positive localizing kernel $K$ we consider a localized version $I_N^{loc}:[-\pi,\pi]\rightarrow [0,\infty)$ of the periodogram (\ref{eq:periodogram}) which is given by
\begin{align}\label{eq:localperiodogram}
\begin{aligned}
\!\! I_N^{loc}(\omega)\!&= \!\frac{\delta_N}{2\pi b_N} \!\! \left(\sum_{j=-m_N}^{m_N} \!\!\! \sqrt{K\left(\frac{\tau_j^N-u}{b_N}\right)}Y_N^{\vartheta^*}(\tau_j^N)e^{-ij\omega} \! \right) \!\! \left(\sum_{j=-m_N}^{m_N} \!\!\! \sqrt{K\left(\frac{\tau_j^N-u}{b_N}\right)}Y_N^{\vartheta^*}(\tau_j^N)e^{ij\omega} \! \right)\\
&=\frac{1}{2\pi} \sum_{h=-2m_N}^{2m_N}\hat{\Gamma}_N^{loc}(h)e^{-ih\omega},
\end{aligned}
\end{align}
where $\omega\in[-\pi,\pi]$ and 
\begin{gather*}
\hat{\Gamma}_N^{loc}(h)=\frac{\delta_N}{b_N}\sum_{j=-m_N}^{m_N-h}\sqrt{K\left(\frac{\tau_{j+h}^N-u}{b_N}\right)K\left(\frac{\tau_j^N-u}{b_N}\right)}Y_N^{\vartheta^*}(\tau_j^N)Y_N^{\vartheta^*}(\tau_{j+h}^N),\qquad h\in\Z,
\end{gather*}
with the convention $\hat{\Gamma}_N^{loc}(-h)=\hat{\Gamma}_N^{loc}(h)$. Based on the localized periodogram we define the localized Whittle function $W_N(\vartheta)$ as
\begin{gather*}
W_N(\vartheta)=\frac{1}{4m_N+2}\sum_{j=-2m_N}^{2m_N+1} \left( \frac{I_N^{loc}(\omega_j)}{f_{\tilde{Y}}^{(\Delta)}(\omega_j,\vartheta)}+\log\left(f_{\tilde{Y}}^{(\Delta)}(\omega_j,\vartheta)\right)\right), \qquad \vartheta\in\Theta,
\end{gather*}
where $\omega_j=\frac{\pi j}{2m_N+1}$ for $j=-2m_N,\ldots,2m_N+1$ and $f_{\tilde{Y}}^{(\Delta)}(\cdot,\vartheta)$ denotes the spectral density of $\tilde{Y}^{\vartheta,\Delta}$ given by (\ref{eq:spectraldensitysampledprocess}). Then, the localized Whittle estimator is defined as
\begin{gather}\label{eq:localizedwhittleestimator}
\hat{\vartheta}_N=\argmin_{\vartheta\in\Theta} W_N(\vartheta).
\end{gather}

%
Under the same conditions as in Section \ref{sec5-2} we obtain consistency.
\begin{Theorem}\label{theorem:whittleconsistent}
Assume that \hyperref[assumption:C0]{(C0)} is satisfied for p=2, \hyperref[observations:O1]{(O1)} holds for $N\delta_N=\Delta$ and that the localizing kernel $K$ is continuous and positive or the non-continuous rectangular kernel (\ref{equation:rectangularkernel}). If additionally \hyperref[assumption:C1]{(C1)} - \hyperref[assumption:C7]{(C7)} hold, then $\hat{\vartheta}_N\overset{P}{\longrightarrow}\vartheta^*(u)$ as $N\rightarrow\infty$ for all $u\in\R^+$ with $\hat{\vartheta}_N$ as defined in (\ref{eq:localizedwhittleestimator}).
\end{Theorem}
\begin{proof}
See Section \ref{sec7-7}.
\end{proof}

\begin{Remark}
In contrast to Theorem \ref{theorem:qmleconsistent}, we cannot readily adapt the above theorem for observations following the sampling scheme \hyperref[observations:O2]{(O2)} as $\hat{\Gamma}_N^{loc}(h)$ does not necessarily convergence to $E[\tilde{Y}_u(0)\tilde{Y}_u(h)]$ in this setting (see also Lemma \ref{lemma:convergencegamma}).
\end{Remark}

\begin{Remark}
For time-invariant L\'evy-driven state space models, asymptotic normality of the Whittle estimator has been shown in \cite[Theorem 2]{FHM2020} following an approach that is similar to the proof of Theorem \ref{theorem:asymptoticnormality}. More detailed, the authors approximated the periodogram (\ref{eq:periodogram}) in the score function (\ref{eq:whittlefunction}) by the corresponding periodogram with respect to the process $N^{(\Delta)}_\vartheta$ as defined in (\ref{equation:statespacerepresentationsampled}) (see \cite[Lemma 3]{FHM2020}). It is worth noting that this technique does not immediately carry over to our non-stationary setting.
\end{Remark}

\section{Simulation study}
\label{sec6}
In this section, we study the finite sample behavior and the convergence of the estimators introduced in Section \ref{sec4-3}, Section \ref{sec5-3} and Section \ref{sec5-4} in a simulation study.\\
More precisely, we perform a Monte Carlo study for each estimator and different data generating processes. Using an Euler-Maruyama scheme on the interval $[0,2000]$ we simulate $400$ independent paths of different time-varying L\'evy-driven state space models that start in $0$. For each simulated path we estimate the coefficient function at equidistant points $(u_i)_{i=1,\ldots,101}$, where $u_1=400$ and $u_{101}=1600$.\\
The driving L\'evy process is either a centered Gaussian L\'evy process or a centered normal-inverse Gaussian (NIG) L\'evy process (see e.g. \cite{BN1997} and \cite{R1997} for more details). The distribution of the increments $L(t)-L(t-1)$ of an NIG L\'evy process $L$ is characterized by the density
\begin{align*}
f_{NIG}(x,\mu,\alpha,\beta,\delta_{NIG})=\frac{\delta_{NIG}}{2\pi}\frac{(1+\alpha g(x))}{g(x)^3}e^{\kappa+\beta x -\alpha g(x)},\qquad x\in\R
\end{align*}
with $g(x)=\sqrt{\delta_{NIG}^2+(x-\mu)^2}$ and $\kappa=\sqrt{\alpha^2-\beta^2}$,
where $\mu\in\R$ is a location parameter, $\alpha\geq0$ is a shape parameter, $\beta\in\R$ is a symmetry parameter and $\delta_{NIG}\geq0$ is a scale parameter.\\
It is clear that the step size used for the Euler-Maruyama scheme has a strong impact on the accuracy of the simulated solution of the L\'evy-driven state space model and hence also on a sample taken from this. 
Even for a constant step size, a sample becomes inaccurate if the distance between two observations shrinks, as we observe it for $\delta_N$ when $N$ increases. To overcome possible distortions of our simulation study caused by this issue, we adapt the step size to the sampling scheme by considering a ratio of $1$:$1000$ between the sampled and simulated points, i.e. every $1000$th simulated point is sampled.\\
The observations are sampled according to \hyperref[observations:O1]{(O1)}, where
\begin{gather*}
\delta_N=\frac{1}{N},\qquad b_N=\frac{400}{\sqrt{N}}, \text{ such that}\qquad\Delta=1 \text{ and }\qquad m_N=\lfloor 400\sqrt{N}\rfloor.
\end{gather*}
The bandwidth parameter $b_N$ has to be chosen. Our choice of $b_N$ satisfies the conditions imposed in Theorem \ref{theorem:consistency} and \ref{theorem:asymptoticnormality}. Moreover, the constant $400$ aims to establish a sample size known to provide good results in a stationary setting (see \cite{FHM2020, SS2012b}). An investigation on the choice of the bandwidth is an interesting topic of further research.\\ 
In our simulation study we investigate the values $N=1,4,16,64,256$. If the driving L\'evy process is Gaussian (i.e. $L(1)\sim\mathcal{N}(\mu,\sigma^2)$), we assume that $\mu=0$ and $\sigma^2=0.2$. In the case of an NIG L\'evy process as driving noise we consider the parameters $\alpha=3,\beta=1, \delta_{NIG}=2$ and $\mu=-2/\sqrt{8}$, which implies that $E[L(1)]=0$ and $\Sigma_L=9\sqrt{2}/16\approx 0.7955$. As localizing kernel we consider the rectangular kernel (\ref{equation:rectangularkernel}) or the Epanechnikov kernel 
\begin{align}\label{equation:epankernel}
K_{epan}(x)=\frac{3}{4}(1-x^2) \mathbb{1}_{\{x\in[-1,1]\}}.
\end{align}
To measure the coefficient function estimate's quality, we use the mean integrated square error (MISE), where the integral in the MISE over the interval $[400,1600]$ is replaced by a Riemann sum over the equidistant partition that is based on the estimation points $(u_i)_{i=1,\ldots,101}$.\\
All simulations have been conducted in MATLAB on the BwUniCluster. For numerical optimization, a differential evolution optimization routine has been used.

\subsection{Simulation Study: least squares estimation}
\label{sec6-1}
We simulate a sequence of time-varying Ornstein-Uhlenbeck processes as defined in (\ref{eq:tvcar}) for three different coefficient functions
\begin{align*}
a^{(1)}(t)=\frac{1}{10}+\frac{1}{2} \abs{\cos\left(\frac{t}{500}\right)} ,\quad a^{(2)}(t)=1+\frac{1}{10}\sin\left(\frac{t}{150}\right)\text{ and} \quad a^{(3)}(t)=\frac{1}{2}-\frac{t}{5000},
\end{align*}
$t\in[0,2000]$. The characteristic triplet of the driving L\'evy process is assumed to be known. Using the rectangular kernel as localizing kernel, we compute for the above coefficient functions the localized least squares estimators $\hat{a}^{(1)}(u_i),\hat{a}^{(2)}(u_i)$ and $\hat{a}^{(3)}(u_i)$, $i=1,\ldots,101$. Since the conditions of Theorem \ref{theorem:asymptoticpropertiesleastsquares} are satisfied, all estimators are consistent and asymptotically normal.\\
Indeed, Figure \ref{figure:CAR_NIG} reflects the consistency of $\hat{a}^{(1)}$ and $\hat{a}^{(2)}$. As $N$ grows the mean over $400$ estimations recovers the respective coefficient function with increasing accuracy. Moreover, the MISE of the estimated coefficient functions decreases across all coefficient functions and driving noises as $N$ increases (see Table \ref{table:1}).\\
Qualitatively, we observe a higher bias for estimates conducted near to extreme points of the coefficient functions (see Figure \ref{figure:CAR_NIG}). This arises from the fact that the localizing kernel smoothes the estimation at each fixed estimation point $u_i$ over the window $[u_i-b_N,u_i+b_N]$. If the average of the respective coefficient function on the estimation window deviates from the value of the coefficient function at $u_i$, we observe a comparably high bias (see $N=1,16$ in Figure \ref{figure:CAR_NIG}). For our coefficient functions, the peak effect occurs at extreme points. Since $b_N \downarrow0$, the smoothing window $[u_i-b_N,u_i+b_N]$ shrinks which eventually ensures a low bias also at extreme points (see $N=256$ in Figure \ref{figure:CAR_NIG}).\\
Exemplary, we investigate the performance of the estimates $\hat{a}^{(1)},\hat{a}^{(2)}$ and $\hat{a}^{(3)}$ at fixed estimation points $u_i$, $i=25,50,75$ in Table \ref{table:2}. The MSE presented in Table \ref{table:2} decrease across all estimation points, coefficient functions and noises as $N$ increases.\\ 
It is not surprising that we find high differences in the MSE across the estimation points for each of the coefficient curves $a^{(1)}$ and $a^{(2)}$ at low values of $N$. Again, the main driver for this effect is an increased bias at estimation points close to extreme points of the coefficient functions.
As $a^{(3)}$ is a linear function, we do not find the same effect for this coefficient function.\\
Moreover, in Figure \ref{figure:CAR_QQ}, we compare the empirical distribution of the standardized estimation error of the estimates $\hat{a}^{(1)}$, $\hat{a}^{(2)}$ and $\hat{a}^{(3)}$ at $u_{25}$ for $N=256$ with a standard normal distribution through a Q-Q plot. For standardization we use the consistent estimator from Remark \ref{remark:consistentestimator}. All Q-Q plots show that the sample quantiles of the standardized estimation error are close to those of a standard normal law.\\
Overall, the investigated least squares estimators perform well across all coefficient functions and noises and the finite sample behavior is very well described by the asymptotic results established in Theorem \ref{theorem:asymptoticpropertiesleastsquares}.

\begin{table}
\center
\begin{tabular}{ |c||cc|cc|cc|}
\hline
&\multicolumn{2}{ c| }{$\hat{a}^{(1)}$}& \multicolumn{2}{c|}{$\hat{a}^{(2)}$}& \multicolumn{2}{c|}{$\hat{a}^{(3)}$} \\
\cline{2-3}\cline{4-5}\cline{6-7}
$N$ &Gaussian &
\multicolumn{1}{c|}{NIG} &
Gaussian&
\multicolumn{1}{c|}{NIG} &
\multicolumn{1}{c}{Gaussian}&
\multicolumn{1}{c|}{NIG} \\
\hline
1 & 6.2142 & 6.4279 & 7.5200 & 8.9841 & 1.0293 & 1.1089 \\
4 & 1.3816 & 1.3120 & 2.2083 & 2.8325 & 0.4889 & 0.5416 \\
16  & 0.3550 & 0.3868 & 0.9882 & 1.2945 & 0.2407  & 0.2562  \\
64  & 0.1400   & 0.1543 & 0.4837 & 0.6278 & 0.1203  & 0.1290  \\
256  & 0.0651   & 0.0722 & 0.2402 & 0.3085 & 0.0588  & 0.0650  \\
\hline
\end{tabular}
\vspace*{-2mm}
\caption{\small{MISE of $\hat{a}^{(1)},\hat{a}^{(2)},\hat{a}^{(3)}$ for $N=1,4,16,64,256$ using the rectangular kernel (\ref{equation:rectangularkernel}). As driving noise we use either a Gaussian or NIG L\'evy process.}\label{table:1}}
\end{table}

\begin{figure}
\includegraphics[width=16cm]{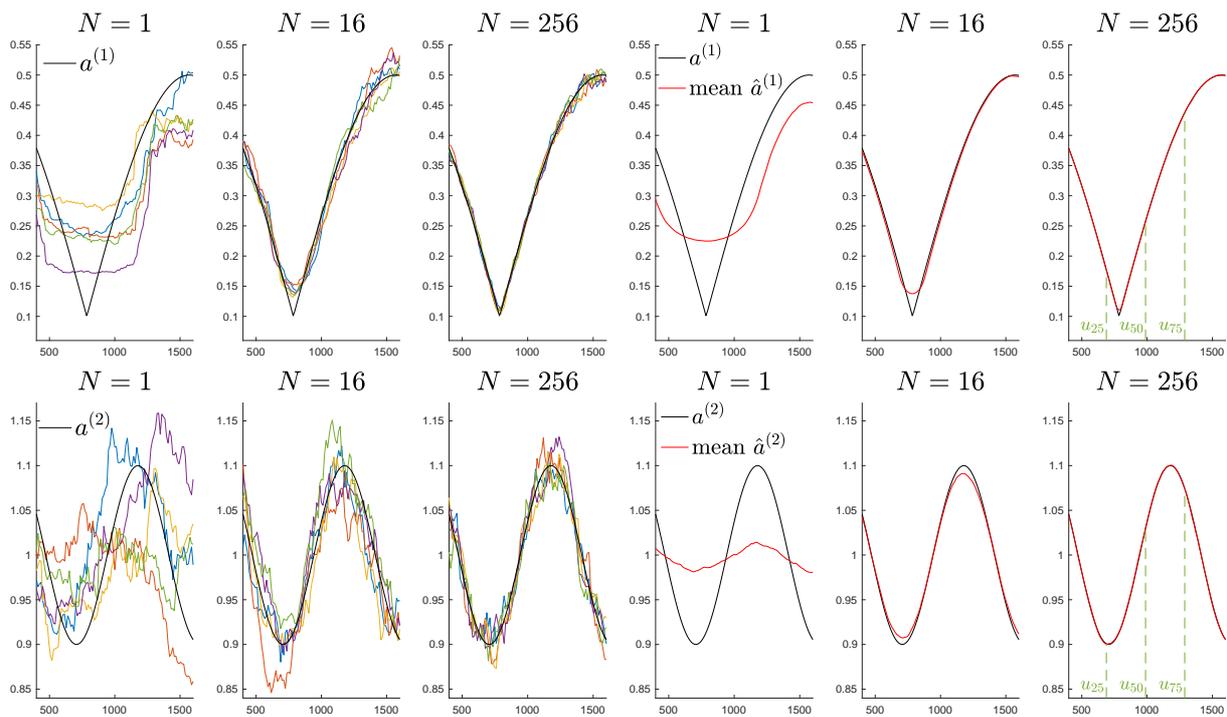}
\vspace*{-7mm}
\caption{\small{First row: coefficient function $a^{(1)}$ with five realizations of $\hat{a}^{(1)}$ and the mean over $400$ realizations of $\hat{a}^{(1)}$ respectively for $N=1,16,256$. Second row: coefficient function $a^{(2)}$ with five realizations of $\hat{a}^{(2)}$ and the mean over $400$ realizations of $\hat{a}^{(2)}$ respectively for $N=1,16,256$. For simulation we considered an NIG L\'evy process. In addition, we indicate $u_{25},u_{50}$ and $u_{75}$ for $N=256$.}\label{figure:CAR_NIG}}
\end{figure}
\vspace{-11pt}
\begin{figure}[H]
\includegraphics[width=16cm]{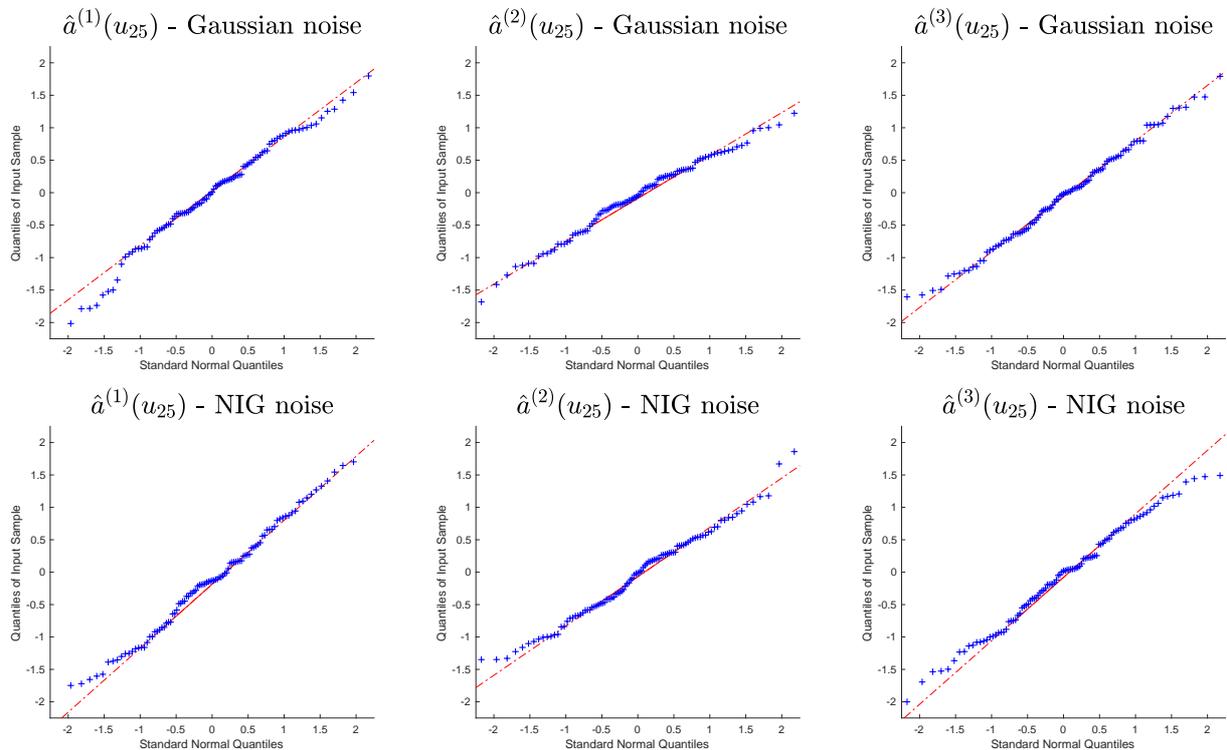}
\vspace*{-7mm}
\caption{\small{Normal Q-Q plots of the standardized estimation error of $\hat{a}^{(1)}(u_{25})$, $\hat{a}^{(2)}(u_{25})$ and $\hat{a}^{(3)}(u_{25})$ for $N=256$, where we considered either a Gaussian or an NIG L\'evy process as driving noise.}\label{figure:CAR_QQ}}
\end{figure}

\begin{table}
\center
\begin{tabular}{ |c||ccc|ccc|ccc|}
\hline
&\multicolumn{3}{ c| }{$\hat{a}^{(1)}$}& \multicolumn{3}{c|}{$\hat{a}^{(2)}$}& \multicolumn{3}{c|}{$\hat{a}^{(3)}$} \\
\cline{2-4}\cline{5-7}\cline{8-10}
$N$ & $u_{25}$ & $u_{50}$ & $u_{75}$ & $u_{25}$ & $u_{50}$ & $u_{75}$ & $u_{25}$ & $u_{50}$ & $u_{75}$  \\
\hline
1 & 3.5077 & 1.5266 & 5.6332 & 10.6488 & 5.1043 & 8.4159 & 1.0586 & 0.8761 & 0.7687 \\
4 & 0.3127 & 1.5355 & 0.7543 & 2.4044 & 2.1300 & 2.6313 & 0.4783 & 0.4464 & 0.3483 \\
16 & 0.2698 & 0.2551 & 0.3058 & 0.9397 & 1.0692 & 1.1790 & 0.2917 & 0.2272 & 0.1739 \\
64  & 0.0597 & 0.0868 & 0.1576 & 0.4116 & 0.5218 & 0.5996 & 0.1205 & 0.1134 & 0.0767 \\
256  & 0.0338 & 0.0470 & 0.0776 & 0.2163 & 0.2569 & 0.2948 & 0.0662 & 0.0591 & 0.0417 \\
\hline
\end{tabular}
\vspace*{-2mm}
\caption{\small{MSE$\times10^{3}$  of the estimators $\hat{a}^{(1)}(u_{i}),\hat{a}^{(2)}(u_{i})$ and $\hat{a}^{(3)}(u_{i})$ for $N=1,4,16,64,256$ and $i=25,50,75$ using the rectangular kernel (\ref{equation:rectangularkernel}). As driving noise we consider an NIG L\'evy process.}\label{table:2}}
\end{table}

\subsection{Simulation study: quasi-maximum likelihood and Whittle estimation}
\label{sec6-2}
We simulate a sequence of time-varying L\'evy-driven state space models as defined in (\ref{eq:seqtvLDstatespacesolution}) for the matrix functions
\begin{align*}
A(t)=\left(\begin{array}{cc}\vartheta_1(t) & 0 \\0 & \vartheta_2(t) \\\end{array}\right),~
B(t)=\left(\begin{array}{c}\frac{1}{\vartheta_2(t)-\vartheta_1(t)} \\ \frac{-1}{\vartheta_2(t)-\vartheta_1(t)}\\ \end{array}\right),
C(t)=\left(\begin{array}{c}-\vartheta_1(t)(1+\vartheta_2(t)) \\ -\vartheta_2(t)(1+\vartheta_1(t))\\ \end{array}\right),
\end{align*}
and $\Sigma_L=\vartheta_3(t)$, where
\begin{align*}
\vartheta_1(t)=-\tfrac{1}{2}+0.1\abs{\sin(\tfrac{t}{500})}\text{, and}\quad
\vartheta_2(t)=-3-0.2\abs{\cos(\tfrac{t}{500})},\quad t\in[0,2000].
\end{align*}
In the Gaussian case, we consider $\vartheta_3(t)=0.2$ and $\vartheta_3(t)=\frac{9\sqrt{2}}{16}\approx 0.7955$ in the NIG case. Using either the rectangular kernel (\ref{equation:rectangularkernel}) or the Epanechnikov kernel (\ref{equation:epankernel}) as localizing kernel, we compute for the aforementioned coefficient functions the
\begin{alignat*}{5}
&\text{quasi-maximum likelihood estimators }\qquad &&(\hat{\vartheta}_1^{QML}(u_i),&&\hat{\vartheta}_2^{QML}(u_i),&&\hat{\vartheta}_3^{QML}(u_i))\qquad &&\text{and the}\\
&\text{Whittle estimators }&&(\hat{\vartheta}_1^{W}(u_i),&&\hat{\vartheta}_2^{W}(u_i),&&\hat{\vartheta}_3^{W}(u_i)),\qquad &&i=1,\ldots,101,
\end{alignat*}
from Section \ref{sec5-3} and Section \ref{sec5-4}, respectively. All conditions of Theorem \ref{equation:consistencymodifiedQMLE} and Theorem \ref{theorem:whittleconsistent} are satisfied (see also Example \ref{example:timevaryingstatespacemodel}) such that both estimators are consistent.\\
Figure \ref{figure:CARMA_NIG_QML} and \ref{figure:CARMA_NIG_W} are in line with our theoretical findings and reflect the estimators' consistency. For both estimators, we observe that the mean over $400$ estimates recovers the respective true coefficient function more precisely as $N$ increases, independently of the driving noise and the localizing kernel. We note that the Epanechnikov kernel has a stronger smoothing effect compared to the rectangular kernel (see Figure \ref{figure:CARMA_NIG_QML} and \ref{figure:CARMA_NIG_W}).\\
Table \ref{table:3} and Table \ref{table:4} show that also the MISE of all estimated coefficient functions decreases as $N$ increases for both estimators, all localizing kernels and driving noises. However, for the rectangular kernel, we observe lower MISE.\\
Exemplary, we compare in Figure \ref{figure:CARMA_QQ_QML} and Figure \ref{figure:CARMA_QQ_W} for fixed estimation points $u_i$, $i=25,75$ the empirical distribution of the estimation error with a standard normal distribution, where we consider different localizing kernels and driving noises for the quasi-maximum likelihood and Whittle estimator. The results show that the estimation error's distribution of both estimators can be well approximated by a normal distribution and strengthen the hypothesis that the estimators are asymptotically normal.\\
Overall, the performances of the quasi-maximum likelihood and the Whittle estimator are very similar, and neither of the estimators is preferable.
%

\begin{figure}[H]
\center
\includegraphics[width=16cm]{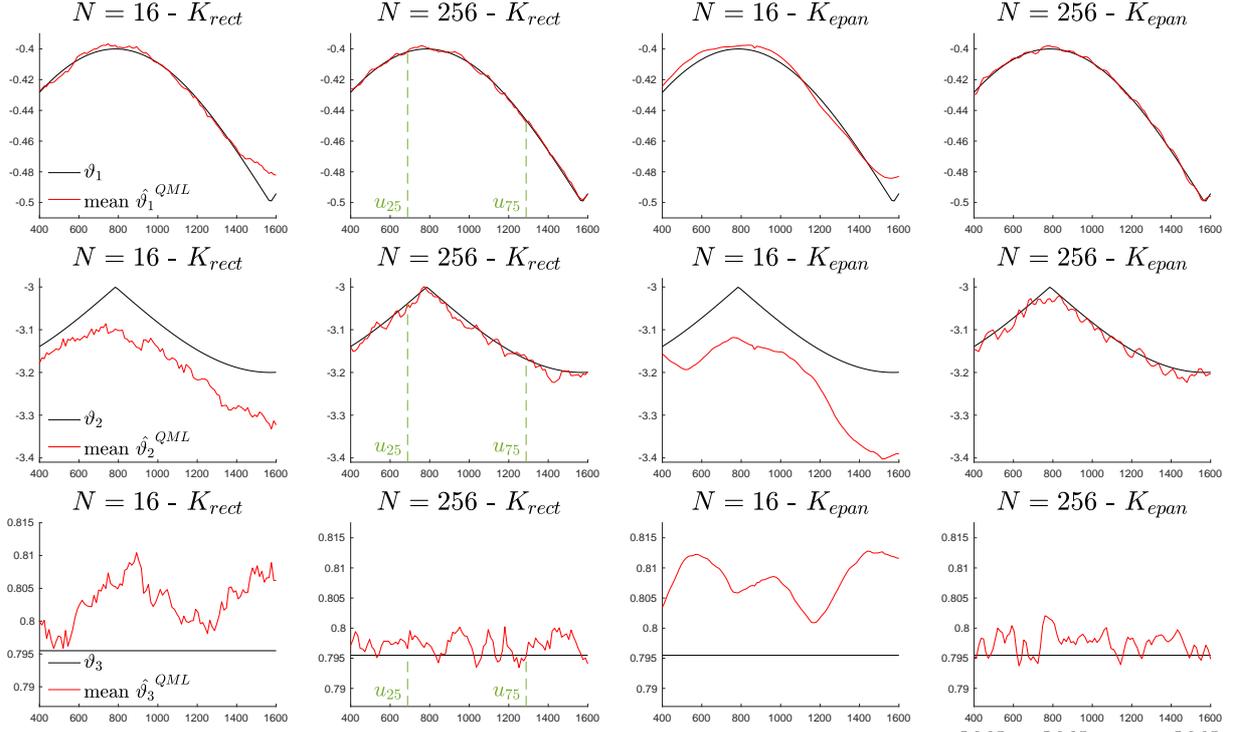}
\vspace*{-7mm}
\caption{\small{Coefficient functions $\vartheta_1, \vartheta_2, \vartheta_3$ and the mean over $400$ realizations of $\hat{\vartheta}_1^{QML}, \hat{\vartheta}_2^{QML}$ and $\hat{\vartheta}_3^{QML}$ for $N=16,256$ using either the rectangular or the Epanechnikov kernel.
For the simulation we considered an NIG L\'evy process.}\label{figure:CARMA_NIG_QML}}
\end{figure}
\vspace{-6pt}
\begin{figure}[H]
\center
\includegraphics[width=16cm]{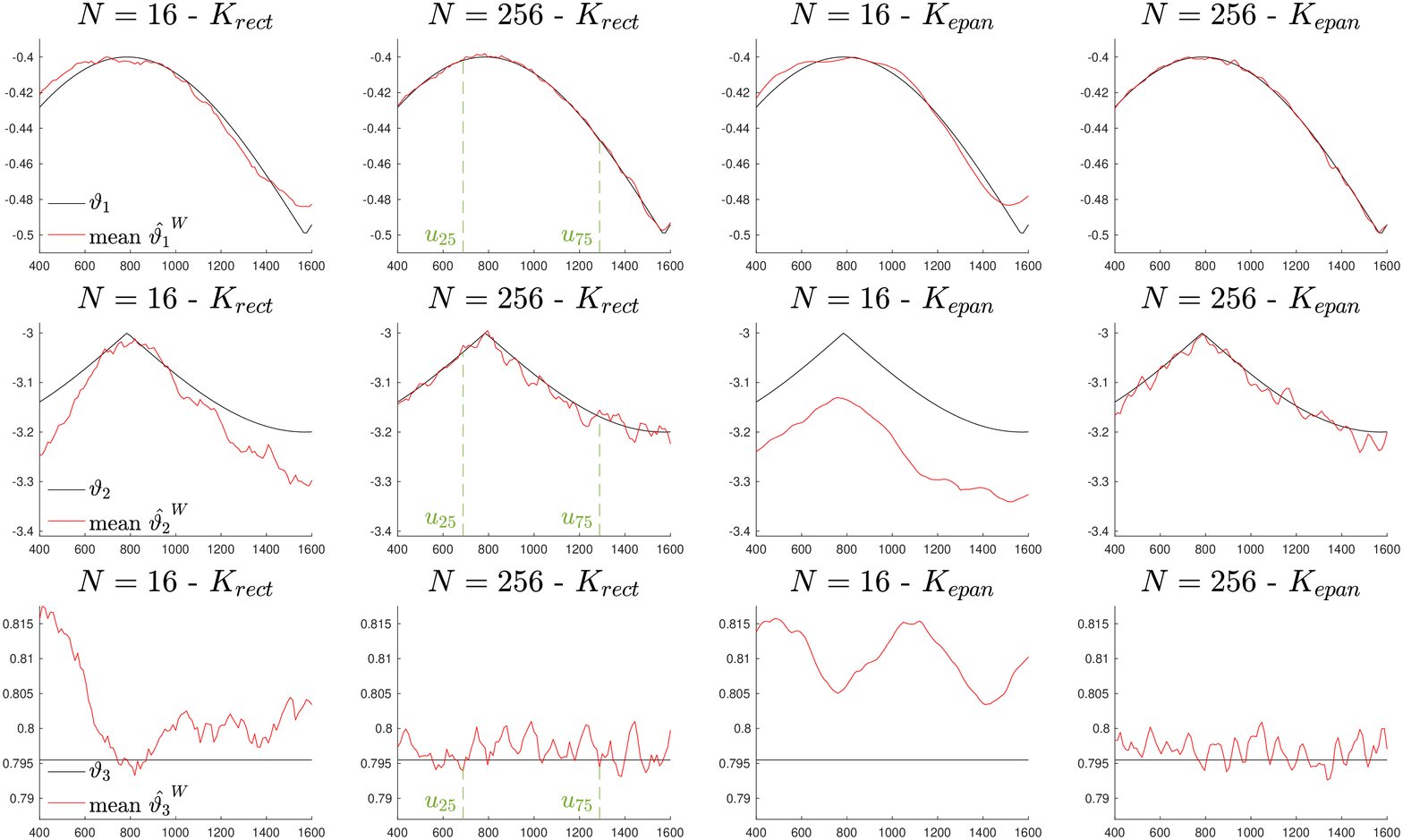}
\vspace*{-7mm}
\caption{\small{Coefficient functions $\vartheta_1, \vartheta_2, \vartheta_3$ and the mean over $400$ realizations of $\hat{\vartheta}_1^{W}, \hat{\vartheta}_2^{W}$ and $\hat{\vartheta}_3^{W}$ for $N=16,256$ using either the rectangular or the Epanechnikov kernel.
For the simulation we considered an NIG L\'evy process.}\label{figure:CARMA_NIG_W}}
\end{figure}

\begin{table}[H]
\center
\begin{tabular}{ |cc||ccc|ccc|}
\hline
& &\multicolumn{3}{ c| }{$K_{rect}$}& \multicolumn{3}{c|}{$K_{epan}$}\\
\cline{1-3}\cline{3-5}\cline{6-8}
 & \multicolumn{1}{|c||}{} &&& &&& \\[-11pt]
 Noise& \multicolumn{1}{|c||}{$N$} & $\hat{\vartheta}_1^{QML}$ & $\hat{\vartheta}_2^{QML}$ & $\hat{\vartheta}_3^{QML}$ &$\hat{\vartheta}_1^{QML}$ & $\hat{\vartheta}_2^{QML}$ & $\hat{\vartheta}_3^{QML}$   \\
\hline
\multirow{5}{*}{Gaussian}& \multicolumn{1}{|c||}{1} & 12.2178 & 954.9114 & 1.5844 & 15.9248 & 1074.9837 & 1.6684 \\
&\multicolumn{1}{|c||}{4} & 5.6206 & 527.5252 & 0.8054 & 6.7165 &  623.1291 & 0.9417 \\
&\multicolumn{1}{|c||}{16} & 2.7979 & 257.3131 & 0.3921 & 3.4222 & 314.5225 & 0.4845 \\
&\multicolumn{1}{|c||}{64}  & 1.3190 & 125.3640 & 0.1887 & 1.6603 & 150.1066 & 0.2357 \\
&\multicolumn{1}{|c||}{256}  & 0.6793 & 62.2253 & 0.0965 & 0.7965 & 74.7023 & 0.1159 \\
\hline
\multirow{5}{*}{NIG}&  \multicolumn{1}{|c||}{1} & 13.3915 &962.2520  & 25.2487 & 15.1466 & 1081.8759 & 28.8395  \\
& \multicolumn{1}{|c||}{4} & 5.8316 & 528.2882 & 13.0114 & 6.8370 & 632.4238 & 15.7627  \\
& \multicolumn{1}{|c||}{16} & 2.8093 & 267.8274 & 6.5422 & 3.2770 & 310.6784 & 7.5447  \\
& \multicolumn{1}{|c||}{64}  & 1.3502 & 125.6270 & 3.1771 & 1.6305 & 155.1777 & 3.8812 \\
& \multicolumn{1}{|c||}{256}  & 0.6743 & 63.4687 & 1.6086 & 0.8241 & 76.5732 & 1.9886  \\
\hline
\end{tabular}
\vspace*{-2mm}
\caption{\small{MISE of the estimators $\hat{\vartheta}_1^{QML}, \hat{\vartheta}_2^{QML}$ and $\hat{\vartheta}_3^{QML}$ for $N=1,4,16,64,256$ using either the rectangular kernel (\ref{equation:rectangularkernel}) or the Epanechnikov kernel (\ref{equation:epankernel}). As driving noise we consider either a Gaussian or an NIG L\'evy process.}\label{table:3}}
\end{table}

\begin{figure}[H]
\includegraphics[width=16cm]{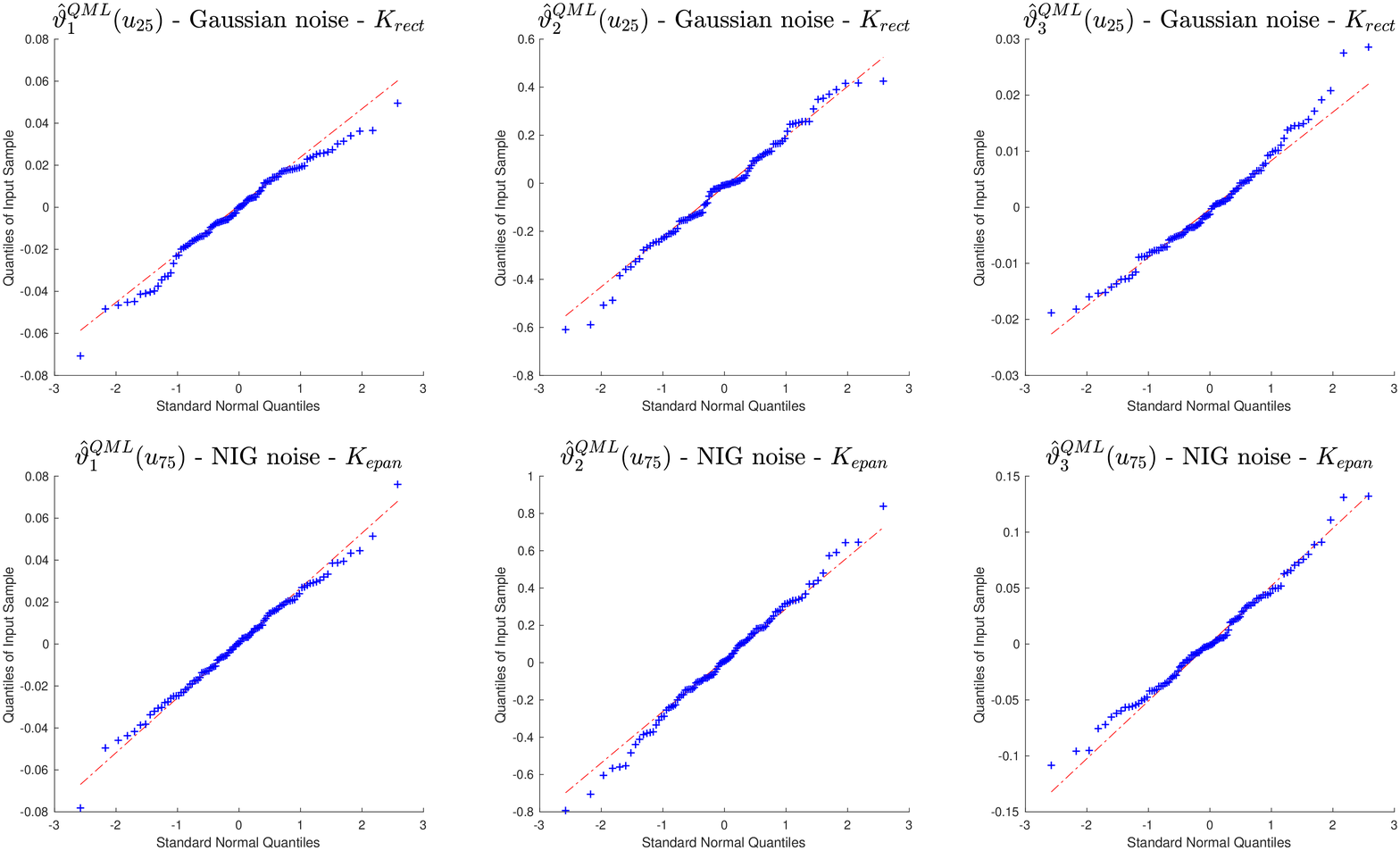}
\vspace*{-6mm}
\caption{\small{First row: normal Q-Q plots of the estimation error of the estimates $\hat{\vartheta}_1^{QML}(u_{25}),\hat{\vartheta}_2^{QML}(u_{25})$ and $\hat{\vartheta}_3^{QML}(u_{25})$ for $N=256$ using the rectangular kernel (\ref{equation:rectangularkernel}) and a Gaussian L\'evy process as driving noise.
Second row: normal Q-Q plots of the estimation error of the estimates $\hat{\vartheta}_1^{QML}(u_{75}),\hat{\vartheta}_2^{QML}(u_{75})$ and $\hat{\vartheta}_3^{QML}(u_{75})$ for $N=256$ using the Epanechnikov kernel (\ref{equation:epankernel}) and an NIG L\'evy process as driving noise.}\label{figure:CARMA_QQ_QML}}
\end{figure}

\begin{table}[H]
\center
\begin{tabular}{ |cc||ccc|ccc|}
\hline
& &\multicolumn{3}{ c| }{$K_{rect}$}& \multicolumn{3}{c|}{$K_{epan}$}\\
\cline{1-3}\cline{3-5}\cline{6-8}
 & \multicolumn{1}{|c||}{} &&& &&& \\[-11pt]
 Noise& \multicolumn{1}{|c||}{$N$} & $\hat{\vartheta}_1^{W}$ & $\hat{\vartheta}_2^{W}$ & $\hat{\vartheta}_3^{W}$ &$\hat{\vartheta}_1^{W}$ & $\hat{\vartheta}_2^{W}$ & $\hat{\vartheta}_3^{W}$   \\
\hline
\multirow{5}{*}{Gaussian}& \multicolumn{1}{|c||}{1} & 12.8997 & 1082.6087 & 1.5627 & 14.3332 & 1143.9525 & 1.8873 \\
&\multicolumn{1}{|c||}{4} & 5.2176 & 526.1243 & 0.7166 & 6.9049 & 626.2685 & 0.9188 \\
&\multicolumn{1}{|c||}{16} & 2.8774 & 256.1971 & 0.4000 & 3.3788 & 295.5396 & 0.4658 \\
&\multicolumn{1}{|c||}{64}  & 1.3734 & 125.1810 & 0.1950 & 1.6296 & 153.6250 & 0.2351 \\
&\multicolumn{1}{|c||}{256}  & 0.6737 & 61.8317 & 0.0960 & 0.8060 & 75.4662 & 0.1164 \\
\hline
\multirow{5}{*}{NIG}&  \multicolumn{1}{|c||}{1} & 12.1261 & 960.6616 & 22.9088 & 14.0046 & 1127.7369 & 27.4745 \\
& \multicolumn{1}{|c||}{4} & 5.6328 & 485.4932 & 12.8253 & 6.8520 & 639.1451 & 15.1066 \\
& \multicolumn{1}{|c||}{16} & 2.7545 & 255.6482 & 6.4252 & 3.4720 & 313.6670 & 7.8922 \\
& \multicolumn{1}{|c||}{64}  & 1.3238 & 124.9270 & 3.1781 & 1.5885 & 155.3078 & 3.8583 \\
& \multicolumn{1}{|c||}{256}  & 0.6855 & 63.3847 & 1.6103 & 0.8083 & 77.1361 & 1.9350 \\
\hline
\end{tabular}
\vspace*{-2mm}
\caption{\small{MISE of the estimators $\hat{\vartheta}_1^{W}, \hat{\vartheta}_2^{W}$ and $\hat{\vartheta}_3^{W}$ for $N=1,4,16,64,256$ using either the rectangular kernel (\ref{equation:rectangularkernel}) or the Epanechnikov kernel (\ref{equation:epankernel}). As driving noise we consider a Gaussian or an NIG L\'evy process.}\label{table:4}}
\end{table}

\begin{figure}[H]
\includegraphics[width=16cm]{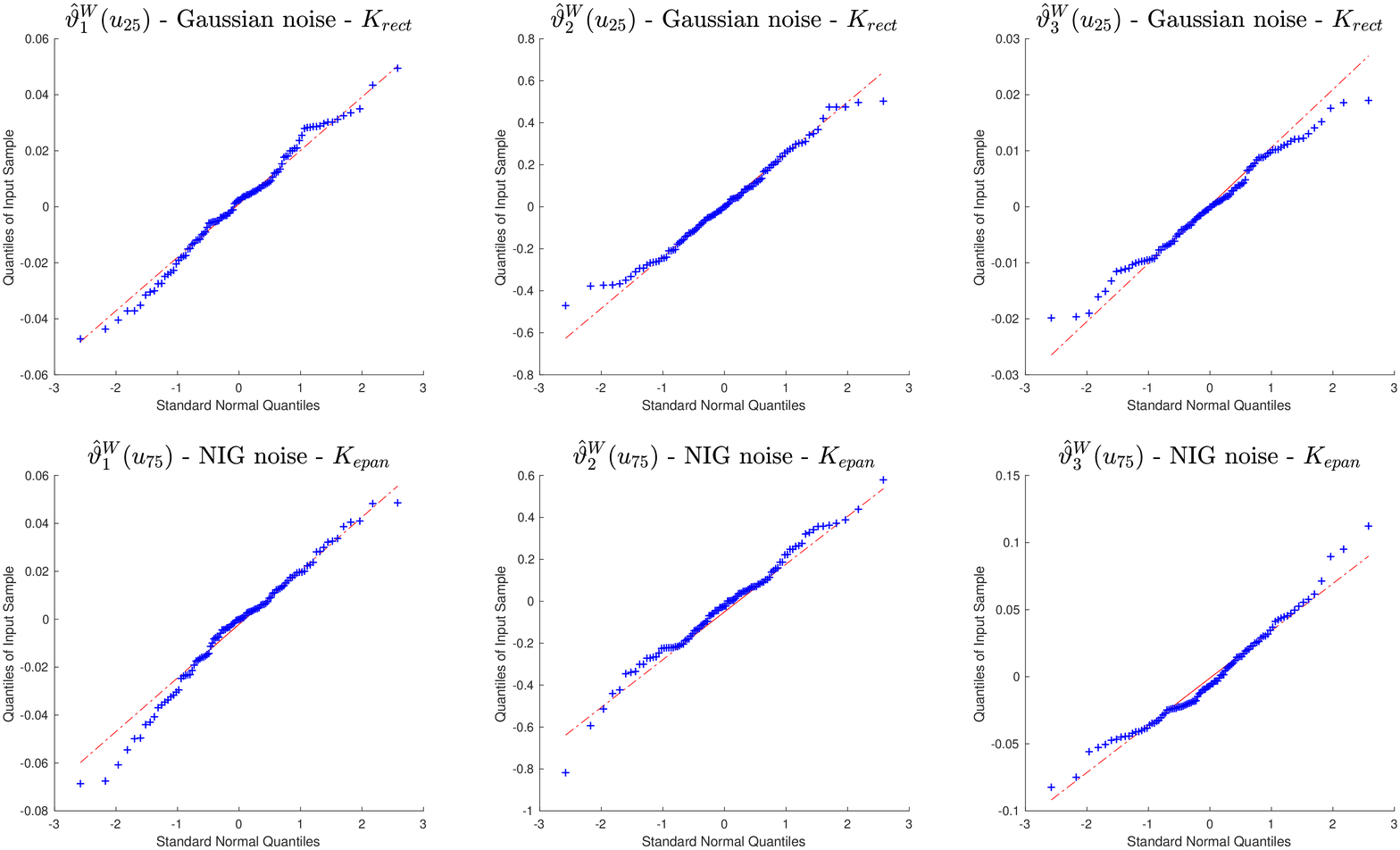}
\vspace*{-6mm}
\caption{\small{First row: normal Q-Q plots of the estimation error of the estimates $\hat{\vartheta}_1^{W}(u_{25}),\hat{\vartheta}_2^{W}(u_{25})$ and $\hat{\vartheta}_3^{W}(u_{25})$ for $N=256$ using the rectangular kernel (\ref{equation:rectangularkernel}) and a Gaussian L\'evy process as driving noise.
Second row: normal Q-Q plots of the estimation error of the estimates $\hat{\vartheta}_1^{W}(u_{75}),\hat{\vartheta}_2^{W}(u_{75})$ and $\hat{\vartheta}_3^{W}(u_{75})$ for $N=256$ using the Epanechnikov kernel (\ref{equation:epankernel}) and an NIG L\'evy process as driving noise.}\label{figure:CARMA_QQ_W}}
\end{figure}


\section{Proofs}
\label{sec7}
\subsection{Proof for Section \ref{sec3-1}}
\label{sec7-1}

\begin{proof}[Proof of Theorem \ref{theorem:consistency}]
Clearly, Lemma \ref{lemma:integrabilityphi} implies that $\Phi$ is integrable. In the following we show the sufficient conditions of \cite[Theorem 5.7]{V1998}.\\
First, we note that $M(\vartheta)$ is continuous, since for $\vartheta_1,\vartheta_2\in\Theta$ we have
\begin{align}\label{eq:continuityMforlater}
\begin{aligned}
\abs{M(\vartheta_1)-M(\vartheta_2)}&\leq\norm{\Phi(\tilde{Y},\vartheta_1)-\Phi(\tilde{Y},\vartheta_2)}_{L^1}\\
&\leq \norm{\vartheta_1-\vartheta_2} 3D_1\left(1+\sum_{k=0}^\infty \beta_k E[|\tilde{Y}_u(\Delta (1-k))|^q]\right).
\end{aligned}
\end{align} 
To show uniform convergence in probability of $M_N(\vartheta)$ we use \cite{N1991}. For all $\vartheta \in\Theta$ Proposition \ref{proposition:inheritanceproperties} and Lemma \ref{lemma:integrabilityphi} imply that $\Phi((\tilde{Y}_u(t+\Delta (1-k)))_{k\in\N_0},\vartheta)$ is a locally stationary approximation of $\Phi\big(\big(Y_N\big(t+\Delta \frac{(1-k)}{N}\big)\big)_{k\in\N_0},\vartheta\big)$ for $p\geq1$. An application of \cite[Theorem 3.5]{SS2021} in the case of \hyperref[observations:O1]{(O1)} and \cite[Theorem 3.6]{SS2021} if \hyperref[observations:O2]{(O2)} holds, gives
\begin{gather*}
\norm{M_N(\vartheta)-M(\vartheta)}_{L^1}\underset{N\rightarrow\infty}{\longrightarrow}0, \text{ for all } \vartheta\in\Theta.
\end{gather*} 
It is left to show stochastic equicontinuity of the family $(M_N(\vartheta))_{N\in\N}$. Define $g:\R^\infty\rightarrow\R$ where $g(x)=\sum_{k=0}^\infty \beta_k|x_k|^q$. Using the mean value theorem, we obtain $||y|^q-|z|^q|\leq q|y-z|(1+|y|^{q-1}+|z|^{q-1})$ for all $y,z\in\R$. An application of Hoelder's inequality ensures $g\in\mathcal{L}_\infty^{1,q}(\beta)$. Since $\sum_{k=0}^\infty k\beta_k<\infty$, Proposition \ref{proposition:inheritanceproperties} implies that $g((\tilde{Y}_u(t+\Delta (1-k)))_{k\in\N_0})$ is a locally stationary approximation of $g\big(\big(Y_N\big(t+\frac{\Delta (1-k)}{N}\big)\big)_{k\in\N_0}\big)$ for $p=1$. Noting that $\frac{|K(x)|}{\int |K(x)|dx}$ is again a localizing kernel, we obtain from either \cite[Theorem 3.5]{SS2021} or \cite[Theorem 3.6]{SS2021} that \begin{align}
\begin{aligned}\label{equation:stochequiproof1}
&\frac{\delta_N}{b_N}\sum_{i=-m_N}^{m_N} \left|K\left(\frac{\tau_i^N-u}{b_N}\right)\right| \left(1+g\left(\left(Y_N\left(\tau_i^N+\frac{\Delta (1-k)}{N}\right)\right)_{k\in\N_0}\right)\right)\\
&\quad\underset{N\rightarrow\infty}{\overset{P}{\longrightarrow}}E\left[1+g\left(\left(\tilde{Y}_u\left(\Delta (1-k)\right)\right)_{k\in\N_0}\right)\right]\int_\R |K(x)|dx
=:E.
\end{aligned}
\end{align}
Then, for $\lambda=\frac{\eta}{6D_1 E}$ it holds
\begin{align}
\begin{aligned}\label{equation:stochequiproof2}
&P\left( \sup_{\norm{\vartheta_1-\vartheta_2}<\lambda}\left|M_N(\vartheta_1)-M_N(\vartheta_2) \right|>\eta\right)\\
&\leq P\left( \left| \frac{\delta_N}{b_N}\!\!\sum_{i=-m_N}^{m_N} \left|K\left(\frac{\tau_i^N-u}{b_N}\right)\right| \left(\!1\!+\!g\left(\!\left(Y_N\left(\tau_i^N\!+\!\frac{\Delta(1- k)}{N}\right)\right)_{k\in\N_0}\right)\!\right)\!-\!E \right| \!>\!E\!\right)\!\!\underset{N\rightarrow\infty}{\longrightarrow}\!0.
\end{aligned}
\end{align}
It follows that $\sup_{\vartheta\in\Theta}\norm{M_N(\vartheta)-M(\vartheta)}\underset{N\rightarrow\infty}{\overset{P}{\longrightarrow}}0$. Finally, we conclude with \cite[Theorem 5.7]{V1998}.
\end{proof}

\subsection{Proof for Section \ref{sec3-2}}
\label{sec7-2}

\begin{proof}[Proof of Theorem \ref{theorem:asymptoticnormality}]
For $M_N$ as defined in (\ref{eq:M_N}) we investigate the Taylor expansion of $\nabla_\vartheta M_N$ at $\vartheta^*$, which is given by
\begin{gather*}
\sqrt{\frac{b_N}{\delta_N}}\Big( \nabla_\vartheta M_N\left(\vartheta^*\right)\Big)=\sqrt{\frac{b_N}{\delta_N}} \left(\nabla_\vartheta M_N(\hat{\vartheta}_N) \right) - \sqrt{\frac{b_N}{\delta_N}} \left(\hat{\vartheta}_N-\vartheta^* \right) \left(\nabla_\vartheta^2 M_N(\tilde{\vartheta}) \right)
\end{gather*}
for some $\tilde{\vartheta}\in\Theta$ satisfying $\lVert\tilde{\vartheta}-\vartheta^*\rVert\leq\lVert\hat{\vartheta}_N-\vartheta^*\rVert$. From the definition of $\hat{\vartheta}_N$ it follows for sufficiently large $N$ that $\nabla_\vartheta M_N(\hat{\vartheta}_N)=0$. Hence, for $\bar{Y}_N=\big(Y_N\big(\tau_i^N+\Delta\frac{(1-k)}{N}\big)\big)_{k\in\N_0}$ and $\bar{Y}_u=(\tilde{Y}_u(N\tau_i^N+\Delta(1-k)))_{k\in\N_0}$ we obtain
\begin{align*}
\sqrt{\frac{b_N}{\delta_N}}\left(\hat{\vartheta}_N-\vartheta^* \right)&= -\sqrt{\frac{b_N}{\delta_N}}\left( \nabla_\vartheta M_N\left(\vartheta^*\right)\right)\left(\nabla_\vartheta^2 M_N(\tilde{\vartheta}) \right)^{-1}\\
&=-\Bigg(\sqrt{\frac{b_N}{\delta_N}} \sum_{i=-m_N}^{m_N} K_{rect}\left( \frac{\tau_i^N-u}{b_N}\right) \left( \nabla_\vartheta \Phi\left(\bar{Y}_N,\vartheta^*\right)-E\left[\nabla_\vartheta \Phi\left(\bar{Y}_N,\vartheta^*\right)\right]\right)\\
&\qquad \qquad +\sqrt{\frac{b_N}{\delta_N}} \sum_{i=-m_N}^{m_N} K_{rect}\left( \frac{\tau_i^N-u}{b_N}\right) E\left[\nabla_\vartheta \Phi\left(\bar{Y}_N,\vartheta^*\right)\right]\Bigg)\left(\nabla_\vartheta^2 M_N(\tilde{\vartheta}) \right)^{-1}\\
&=:-(P_1+P_2)\left(\nabla_\vartheta^2 M_N(\tilde{\vartheta}) \right)^{-1}.
\end{align*}
As first step, we show asymptotic normality of $P_1$, which mainly follows from \cite[Theorem 3.7]{SS2021} and the Cramer-Wold device. \\
Let $a\in\R^d$. Then, since $\frac{\partial}{\partial\vartheta_i}\Phi(\cdot,\vartheta^*)\in \mathcal{L}_{\infty}^{\tilde{p},\tilde{q}}(\alpha)$ for all $i=1,\ldots,d$ and some $\tilde{p}\geq2$, we have $a'\nabla_\vartheta\Phi(\cdot,\vartheta^*)\in \mathcal{L}_{\infty}^{\tilde{p},\tilde{q}}(\alpha)$. Thus, due to Proposition \ref{proposition:inheritanceproperties}, $a' \nabla_\vartheta\Phi(\bar{Y}_u,\vartheta^*)$ is a locally stationary approximation of $a' \nabla_\vartheta\Phi(\bar{Y}_N,\vartheta^*)$ for $\tilde{p}$. Furthermore, since linear functions are Lipschitz, the locally stationary approximation $a' \nabla_\vartheta\Phi(\bar{Y}_u,\vartheta^*)$ is $\theta$-weakly dependent with $\theta$-coefficients, that decay with the same rate as the $\theta$-coefficients of $\nabla_\vartheta\Phi(\bar{Y}_u,\vartheta^*)$. Overall, all assumptions of \cite[Theorem 3.7]{SS2021} are satisfied and
\begin{gather*}
\sqrt{\frac{b_N}{\delta_N}} \sum_{i=-m_N}^{m_N} K_{rect}\left( \frac{\tau_i^N-u}{b_N}\right) a'\left(\nabla_\vartheta\Phi(\bar{Y}_N,\vartheta^*)- E\left[\nabla_\vartheta\Phi(\bar{Y}_N,\vartheta^*)\right]\right)
\overset{d}{\underset{N\rightarrow\infty}{\longrightarrow}} \mathcal{N}\left(0,a'I(u)a\right).
\end{gather*}
From the Cramer-Wold device we obtain immediately
\begin{gather}\label{eq:cramerwoldp1}
P_1\overset{d}{\underset{N\rightarrow\infty}{\longrightarrow}} \mathcal{N}\left(0, I(u)\right).
\end{gather}
Since $E\left[\nabla_\vartheta \Phi\left(\bar{Y}_u,\vartheta^*\right)\right]=0$, we obtain
\begin{align}\label{eq:convergencep2}
P_2&=\sqrt{\frac{b_N}{\delta_N}} \sum_{i=-m_N}^{m_N} K_{rect}\left( \frac{\tau_i^N-u}{b_N}\right) \left(E\left[\nabla_\vartheta \Phi\left(\bar{Y}_N,\vartheta^*\right)\right]- E\left[\nabla_\vartheta \Phi\left(\bar{Y}_u,\vartheta^*\right)\right]\right)\notag\\
&\leq d\sqrt{\frac{b_N}{\delta_N}}(2m_N+1)\abs{K}_\infty  C\left( \frac{1}{N}+b_N \right) \underset{N\rightarrow\infty}{\longrightarrow}0
\end{align}
for some constant $C>0$. Regarding $\nabla_\vartheta^2 M_N(\tilde{\vartheta})$, we first note that
\begin{align*}
\norm{\nabla_\vartheta^2 M_N(\tilde{\vartheta})-V(u)}\leq& \norm{\sqrt{\frac{b_N}{\delta_N}} \sum_{i=-m_N}^{m_N} K\left( \frac{\tau_i^N-u}{b_N}\right)\nabla_\vartheta^2 \Phi\left(\bar{Y}_N, \tilde{\vartheta}\right)-E\left[\nabla_\vartheta^2 \Phi\left(\bar{Y}_u, \tilde{\vartheta}\right) \right]}\\
&+ \norm{E\left[\nabla_\vartheta^2 \Phi\left(\bar{Y}_u, \tilde{\vartheta}\right)\right]- E\left[\nabla_\vartheta^2\Phi\left(\bar{Y}_u,\vartheta^*\right)\right] }=:\lVert R_1(\tilde{\vartheta})\rVert +\norm{R_2}.
\end{align*}
From the conditions on $\frac{\partial^2}{\partial\vartheta_i\partial\vartheta_j}\Phi$ it follows that the map $\vartheta\mapsto E\left[\nabla_\vartheta^2\Phi\left(\bar{Y}_u,\vartheta\right)\right]$ is continuous. Additionally noting that $\tilde{\vartheta}\overset{P}{\longrightarrow}\vartheta^*$ as $\hat{\vartheta}_N\overset{P}{\longrightarrow}\vartheta^*$, we obtain $\norm{R_2}\longrightarrow0$ as $N\rightarrow\infty$.\\ 
To show that $\lVert R_1(\tilde{\vartheta})\rVert$ tends to $0$ in probability, we show that $\sup_{\vartheta\in\Theta}\norm{R_1(\vartheta)}$ converges to $0$ in probability as $N\rightarrow\infty$. To this end, it is sufficient to investigate the asymptotic behavior of $R_1^{(i,j)}(\vartheta)$ for all $i,j=1,\ldots,d$, $\vartheta\in\Theta$, where $R_1^{(i,j)}(\vartheta)$ denotes the $i,j$-th entry of $R_1(\vartheta)$. More precisely,
\begin{gather*}
R_1^{(i,j)}(\vartheta)=\sqrt{\frac{b_N}{\delta_N}} \sum_{i=-m_N}^{m_N} K\left( \frac{\tau_i^N-u}{b_N}\right) \frac{\partial^2 }{\partial\vartheta_i\partial\vartheta_j}\Phi\left(\bar{Y}_N, \vartheta\right) - E\left[  \frac{\partial^2}{\partial\vartheta_i\partial\vartheta_j}\Phi\left(\bar{Y}_u, \vartheta\right) \right].
\end{gather*}
Noting that $\frac{\partial^2\Phi}{\partial\vartheta_i\partial\vartheta_j}$ satisfies all assumptions on the contrast function from Theorem \ref{theorem:consistency}, we can follow the same steps as in the proof of Theorem \ref{theorem:consistency} to obtain
\begin{gather}\label{eq:ucpR_1}
\sup_{\vartheta\in\Theta}\norm{R_1^{(i,j)}(\vartheta)}\underset{N\rightarrow\infty}{\overset{P}{\longrightarrow}}0.
\end{gather}
Finally, combining (\ref{eq:cramerwoldp1}), (\ref{eq:convergencep2}), (\ref{eq:ucpR_1}) and Slutsky's theorem we obtain the convergence as stated in (\ref{eq:asymptoticnormalitythetahat}).
\end{proof}

\subsection{Proof for Section \ref{sec4-1}}
\label{sec7-3}

\begin{proof}[Proof of Proposition \ref{proposition:infinitememorymovingaverage}]
In the following we suppress the index $u\in\R^+$ and assume $q=2$. We fix $h\geq0$ and define, for all $t\in\R$ and $m\in\R^+$ the truncated process $X^{(m)}(t)=\int_{t-m}^{t}g(u,t-s)L(ds)$ and for $l\in\N$
\begin{alignat*}{3}
Z^{(l)}(t)&=\left(X\left(t+\Delta\right),\ldots,X\left(t-\Delta(l-2)\right)\right),\quad  &&\varphi^{(l)}(t) & &=\varphi(Z^{(l)}(t),0,\ldots)\text{ as well as }\\
Z^{(l,m)}(t)&=\left(X^{(m)}\left(t+\Delta\right),\ldots,X^{(m)}\left(t-\Delta(l-2)\right)\right),\quad&&\varphi^{(l,m)}(t) & &=\varphi(Z^{(l,m)}(t),0,\ldots).
\end{alignat*}
Then, for $\Sigma_L=\Sigma + \int_{\R}x^2\nu(dx)$
\begin{align}\label{eq:truncationinequality}
\norm{X(t)-X^{(m)}(t)}_{L^2}=\norm{\int_{-\infty}^{-m}g(u,-s)L(ds)}_{L^2}=\left(\Sigma_L\int_{-\infty}^{-m}g(u,-s)^2ds\right)^{\frac{1}{2}}.
\end{align}
Let $F\in\mathcal{G}^*_u$, $G\in\mathcal{G}_1$ and $i_1\leq\ldots\leq i_u\leq i_u+h\leq j$, $u\in\N$. In view of Definition \ref{thetaweaklydependent} we define $\varphi^*=\left(\varphi(Z(i_1),\ldots,\varphi(Z(i_u)))\right)$ and obtain
\begin{align*}
\abs{Cov(F(\varphi^*),G(\varphi(Z(j))))}&\leq \abs{Cov\left(F(\varphi^*),G\left(\varphi(Z(j))\right)-G\left(\varphi^{(l)}(j)\right)\right)}\\
&\quad+\abs{Cov\left(F(\varphi^*),G\left(\varphi^{(l)}(j)\right)-G\left(\varphi^{(l,m)}(j)\right)\right)}\\
&\quad+\abs{Cov\left(F(\varphi^*),G\left(\varphi^{(l,m)}(j)\right)\right)}=:P_1+P_2+P_3.
\end{align*}
We show that $\varphi^*$ and $\varphi^{(l,m)}(j)$ are independent which implies $P_3=0$. Due to to the independence of $L$ on non-overlapping intervals it is sufficient to truncate the stochastic integrals entering in $\varphi^{(l,m)}(j)$ with suitable choices of $l$ and $m$ such that the involved integration sets are disjoint from any integration set entering in $g^*$.\\
Since $i_u\geq i_a$ for all $a=1,\ldots,u$, it is enough to investigate $Z(i_u)$ and $Z^{(l,m)}(j)$. In turn, this requires to analyze the stochastic integrals in $Z(i_u)$ and $Z^{(l,m)}(j)$ whose integration sets have the smallest distance, i.e.
\begin{align*}
X^{(m)}(j-\Delta(l-2))&=\int_{j-\Delta(l-2)-m}^{j-\Delta(l-2)} g(u,j-\Delta(l-2)-s)L(ds)\text{ and}\\
X(i_u+\Delta)&=\int_{-\infty}^{i_u+\Delta} g(u,i_u+\Delta-s)L(ds).
\end{align*}  
We set $m=\frac{h}{2}$ and $l=\left\lfloor \frac{h}{2\Delta}\right\rfloor$. Then, $i_u+\Delta< j-\Delta(l-2)-m$ and $P_3=0$.\\
For $P_1$ we have for some function $f:\R_0^+\rightarrow\R_0^+$ due to the stationarity of $X$
\begin{align*}
P_1&\leq 2 \norm{F}_\infty \norm{G\left(\varphi(Z(j))\right)-G\left(\varphi^{(l)}(j)\right)}_{L^1}\\
&\leq 2 \norm{F}_\infty Lip(G)f\left(\norm{X(0)}_{L^2} \right) \norm{X(0)}_{L^2}\sum_{k=l}^\infty \alpha_k\underset{h\ra\infty}{\longrightarrow} 0,
\end{align*}
by the dominated convergence theorem, since we set $l=\left\lfloor \frac{h}{2\Delta}\right\rfloor$. Moreover, using (\ref{eq:truncationinequality})
\begin{align*}
P_2&\leq 2 \norm{F}_\infty \norm{G\left(\varphi^{(l)}(j)\right)-G\left(\varphi^{(l,m)}(j)\right)}_{L^1}\\
&\leq2 \norm{F}_\infty Lip(G) f\left(\norm{X(0)}_{L^2}\vee \norm{X^{(m)}(0)}_{L^2} \right)\\
&\qquad\times\sum_{k=0}^{l-1} \alpha_k\norm{X(j-\Delta(k-1))-X^{(m)}(j-\Delta(k-1))}_{L^2}\\
&\leq 2\norm{F}_\infty Lip(G) f\left(\norm{X(0)}_{L^2}\vee \norm{X^{(m)}(0)}_{L^2} \right)\left(\sum_{k=0}^{\infty} \alpha_k\right) \left(\Sigma_L \int_{-\infty}^{-m}g(u,-s)^2ds\right)^{\frac{1}{2}} \underset{m\ra\infty}{\longrightarrow} 0,
\end{align*}
again by the dominated convergence theorem for $m=\frac{h}{2}$. For $q=4$ we follow similar steps and use that $\lVert X(t)-X^{(m)}(t)\rVert_{L^4}=((\int_{-\infty}^{-m}g(u,-s)^4ds)(\int_\R x^4\nu(dx))+3\Sigma_L^2 (\int_{-\infty}^{-m}g(u,-s)^2ds)^2)^{\frac{1}{4}}$ (see \cite[Lemma 5.2]{SS2021}).
\end{proof}

\subsection{Proof for Section \ref{sec4-3}}
\label{sec7-4}

To calculate the asymptotic variance in Theorem \ref{theorem:asymptoticpropertiesleastsquares}, it is necessary to evaluate $4$-th order mixed moments of $\tilde{Y}_u$.

\begin{Lemma}\label{lemma:4thmixedmoments}
Let $\tilde{Y}_u$ be a L\'evy-driven Ornstein-Uhlenbeck process as given in (\ref{eq:locapproxcar}) such that $\gamma+\int_{\abs{x}>1}x\nu(dx)=0$ and $\int_{\abs{x}>1}x^{4}\nu(dx)<\infty$. Then, for any real numbers $t_1\leq t_2, t_3, t_4$
\begin{align*}
E[\tilde{Y}_u(t_1)\tilde{Y}_u(t_2)\tilde{Y}_u(t_3)\tilde{Y}_u(t_4)]=&\frac{\Sigma_L^2}{4a(u)^2} \Big(e^{a(u)(t_1-t_2-\abs{t_3-t_4})}\!+\!e^{a(u)(t_1-t_3-\abs{t_2-t_4})}\!+\!e^{a(u)(t_1-t_4-\abs{t_2-t_3})}\Big)\\
&+\frac{e^{a(u)(3t_1-t_2-t_3-t_4)}}{4a(u)}\int_\R x^4\nu(dx),
\end{align*}
where $\Sigma_L=\Sigma+\int_\R x^2\nu(dx)$.
\end{Lemma}
\begin{proof}
Define the random vector
\begin{align*}
X=\int_\R (g(u,t_1-s),g(u,t_2-s),g(u,t_3-s),g(u,t_4-s))'L(ds).
\end{align*}
From Section \ref{sec4-1} it is known that $X$ is infinitely divisible. Since $X$ is centered, we obtain for the joint cumulant $\kappa(\tilde{Y}_u(t_1),\tilde{Y}_u(t_2),\tilde{Y}_u(t_3),\tilde{Y}_u(t_4))$, as defined in \cite[Definition 4.2.1.]{GKS2012}, 
\begin{align*}
\kappa(\tilde{Y}_u(t_1),\tilde{Y}_u(t_2),\tilde{Y}_u(t_3),\tilde{Y}_u(t_4))&=\frac{\partial^4}{\partial z_1\partial z_2\partial z_3\partial z_4}\log\left(E\left[e^{i (z_1,z_2,z_3,z_4)X}\right]\right)\!\!\bigg|_{z_1,\ldots,z_4=0}\\
&=\int_\R x^4\nu(dx) \int_\R g(u,t_1-s)g(u,t_2-s)g(u,t_3-s)g(u,t_4-s)ds\\
&=\frac{e^{a(u)(3t_1-t_2-t_3-t_4)}}{4a(u)} \int_\R x^4\nu(dx) .
\end{align*}
On the other hand, due to \cite[Proposition 4.2.2.]{GKS2012},
\begin{align*}
&E[\tilde{Y}_u(t_1)\tilde{Y}_u(t_2)\tilde{Y}_u(t_3)\tilde{Y}_u(t_4)]\\
&=\kappa(\tilde{Y}_u(t_1),\tilde{Y}_u(t_2),\tilde{Y}_u(t_3),\tilde{Y}_u(t_4))+E\left[\tilde{Y}_u(t_1)\tilde{Y}_u(t_2)\right]E\left[\tilde{Y}_u(t_3)\tilde{Y}_u(t_4)\right]\\
&\quad +E\left[\tilde{Y}_u(t_1)\tilde{Y}_u(t_3)\right]E\left[\tilde{Y}_u(t_2)\tilde{Y}_u(t_4)\right]+E\left[\tilde{Y}_u(t_1)\tilde{Y}_u(t_4)\right]E\left[\tilde{Y}_u(t_2)\tilde{Y}_u(t_3)\right].
\end{align*}
We conclude by noting that $E[\tilde{Y}_u(x)\tilde{Y}_u(y)]=\frac{\Sigma_L}{2a(u)}e^{-a(u)\abs{x-y}}$ for any $x,y\in\R$.
\end{proof}

\begin{proof}[Proof of Theorem \ref{theorem:asymptoticpropertiesleastsquares}]
First, we note that $\Phi^{LS}(0,0,\vartheta)=0$. We use Theorem \ref{theorem:consistency} and Remark \ref{remark:consistencyfinitememory}. From Proposition \ref{proposition:car1locstat} it follows that $\tilde{Y}_u$ is a locally stationary approximation of $Y_N$ for $q=2$.
Basic calculations show that for $\Sigma_L=\Sigma+\int_\R x^2\nu(dx)$ we have
\begin{align*}
E\left[\Phi^{LS}\left(\left(\tilde{Y}_u(\Delta (1-k))\right)_{k\in\N_0},\vartheta\right)\right]
&=(1+e^{-2\Delta\vartheta})E[\tilde{Y}_u(0)^2]-2e^{-\Delta\vartheta}E[\tilde{Y}_u(\Delta)\tilde{Y}_u(0)]\\
&= \frac{\Sigma_L}{2a(u)}(1+e^{-2\Delta\vartheta}-2e^{-\Delta(a(u)+\vartheta)}),
\end{align*} 
which has a unique minimum in $\vartheta=a(u)$ such that \hyperref[assumption:M1]{(M1)} holds for $\vartheta^*=a(u)$. We show that the conditions (a) and (b) of Theorem \ref{theorem:consistency} hold. In view of (a) we have for $x_0,x_1,y_0,y_1\in\R$ and $\vartheta\in\Theta$ 
\begin{align}\label{equation:phiinequalityinx}
\begin{aligned}
\abs{\Phi^{LS}(x_0,x_1,\vartheta)-\Phi^{LS}(y_0,y_1,\vartheta)}&\leq \abs{x_0^2-y_0^2}+2e^{-\Delta\vartheta}\abs{y_0y_1-x_0x_1}+ e^{-2\Delta\vartheta}\abs{x_1^2-y_1^2}\\
&\leq\left(\abs{x_0-y_0}+\abs{x_1-y_1}\right)\\
&\quad\times (1+2e^{-\Delta\vartheta}+e^{-2\Delta\vartheta})(\abs{x_0}+\abs{x_1}+\abs{y_0}+\abs{y_1}). 
\end{aligned}
\end{align}
Using Hoelder's inequality and (\ref{equation:phiinequalityinx}), we obtain for $X_0,X_1,Y_0,Y_1\in L^2$ and $\vartheta\in\Theta$
\begin{align*}
&\norm{\Phi^{LS}(X_0,X_1,\vartheta)-\Phi^{LS}(Y_0,Y_1,\vartheta)}_{L^1}\\
&\leq 16\max_{k=0,1}\{\norm{X_k}_{L^2}\vee\norm{Y_k}_{L^2}\}\left(\norm{X_0-Y_0}_{L^2}+\norm{X_1-Y_1}_{L^2}\right),
\end{align*}
which shows $\Phi^{LS}(\cdot,\vartheta)\in\mathcal{L}_{\infty}^{1,2}(\alpha)$, where $\alpha_k=1$ for $k=0,1$ and $\alpha_k=0$ for $k\geq2$. To show (b) let $x_0,x_1\in\R$ and $\vartheta_1,\vartheta_2\in\Theta$. Then,
\begin{align*}
\abs{\Phi^{LS}(x_0,x_1,\vartheta_1)-\Phi^{LS}(x_0,x_1,\vartheta_2)}&\leq 2\abs{x_0}\abs{x_1}\abs{e^{-\Delta\vartheta_2}-e^{-\Delta\vartheta_1}}+x_1^2 \abs{e^{-2\Delta\vartheta_1}-e^{-2\Delta\vartheta_2}}\\
&\leq \Delta\left(2x_0^2+4x_1^2\right) \abs{\vartheta_1-\vartheta_2},
\end{align*}
such that $\Phi^{LS}(x,\cdot)\in\mathcal{L}_{2}(0,D_1(1+\sum_{k=0}^\infty\beta_k\abs{x_k}^2))$ for $D_1=\Delta$ and $\beta_0=2$, $\beta_1=4$ and $\beta_k=0$ for $k\geq2$. If \hyperref[observations:O1]{(O1)} holds, the stated convergence follows immediately from Theorem \ref{theorem:consistency}.\\ 
Using Remark \ref{remark:consistencyfinitememory} we also obtain consistency if \hyperref[observations:O2]{(O2)} holds. Indeed (c$^*$) holds, since the contrast $\Phi^{LS}$ is of finite memory (for $n=1$), $|\Phi^{LS}(x_0,x_1,\vartheta)|\leq(\abs{x_0}+\abs{x_1})^2$ and (\ref{equation:phiinequalityinx}) gives that $\Phi^{LS}(\cdot,\vartheta)\in\mathcal{L}_{2}(1,4)$ for all $\vartheta\in\Theta$ and $x_0,x_1\in\R$. Since $\int_{\abs{x}>1}x^{2+\gamma_1}\nu(dx)<\infty$, we have $\tilde{Y}_u\in L^{2+\gamma_1}$. Analogous to \cite[Theorem 3.36]{CSS2021} one can show that $\tilde{Y}_u$ is $\theta$-weakly dependent with exponentially decaying $\theta$-coefficients $\theta(h)$. Then, the stated result follows from Remark \ref{remark:consistencyfinitememory}.\\
To show asymptotic normality we apply Theorem \ref{theorem:asymptoticnormality} and Remark \ref{remark:asymptoticnormalityfinitememory}. Proposition \ref{proposition:car1locstat} implies that $\tilde{Y}_u$ is a locally stationary approximation of $Y_N$ for $q=4$. The conditions (a)-(c) of Theorem \ref{theorem:asymptoticnormality} are immediately satisfied. For all $x_0,x_1\in\R$ and $\vartheta\in\Theta$ it holds
\begin{align*}
\frac{d}{d\vartheta}\Phi^{LS}(x_0,x_1,\vartheta)&=2\Delta e^{-\Delta\vartheta} \left(x_0x_1-e^{-\Delta \vartheta}x_1^2\right)\text{ and}\\
\frac{d^2}{d\vartheta^2}\Phi^{LS}(x_0,x_1,\vartheta)&=2\Delta^2e^{-\Delta\vartheta}\left(2e^{-\Delta\vartheta}x_1^2-x_0x_1\right).
\end{align*}
Note that $\frac{d^2}{d\vartheta^2}\Phi^{LS}(0,\vartheta)=0$ for all $\vartheta\in\Theta$. Then, for $x_0,x_1,y_0,y_1\in\R$ and $\vartheta\in\Theta$ we have
\begin{align}
\abs{\frac{d}{d\vartheta}\Phi^{LS}(x_0,x_1,\vartheta)-\frac{d}{d\vartheta}\Phi^{LS}(y_0,y_1,\vartheta)}&\leq \left(\abs{x_0-y_0}+\abs{x_1-y_1}\right)2\Delta e^{-\Delta\vartheta}\left(1+e^{-\Delta\vartheta}\right)\notag\\ 
&\quad \times \left(\abs{x_0}+\abs{x_1}+\abs{y_0}+\abs{y_1}\right)\text{ and}\label{equation:firstphiinequalityinx}\\
\abs{\frac{d^2}{d\vartheta^2}\Phi^{LS}(x_0,x_1,\vartheta)-\frac{d^2}{d\vartheta^2}\Phi^{LS}(y_0,y_1,\vartheta)}&\leq\left(\abs{x_0-y_0}+\abs{x_1-y_1}\right)2\Delta^2 e^{-\Delta\vartheta}\left(1+2e^{-\Delta\vartheta}\right)\notag\\ 
&\quad \times \left(\abs{x_0}+\abs{x_1}+\abs{y_0}+\abs{y_1}\right).\label{equation:secondphiinequalityinx}
\end{align}
In view of (d) we apply Hoelder's inequality to (\ref{equation:firstphiinequalityinx}) and obtain for $X_0,X_1,Y_0,Y_1\in L^2$ and $\vartheta\in\Theta$
\begin{align*}
\norm{\frac{d}{d\vartheta}\Phi^{LS}(X_0,X_1,\vartheta)-\frac{d}{d\vartheta}\Phi^{LS}(Y_0,Y_1,\vartheta)}_{L^2}&\leq 16\Delta \max_{k=0,1}\{\norm{X_k}_{L^4}\vee\norm{Y_k}_{L^4}\}\\
&\qquad \left(\norm{X_0-Y_0}_{L^4}+\norm{X_1-Y_1}_{L^4}\right),
\end{align*}
which shows $\frac{d}{d\vartheta}\Phi^{LS}(\cdot,\vartheta^*)\in\mathcal{L}_{\infty}^{2,4}(\tilde{\alpha})$, where $\tilde{\alpha}_k=1$ for $k=0,1$ and $\tilde{\alpha}_k=0$ for $k\geq2$.\\ 
Note that $\frac{d}{d\vartheta}\Phi^{LS}(0,0,\vartheta)=0$ and, due to (\ref{equation:firstphiinequalityinx}), also $\frac{d}{d\vartheta}\Phi^{LS}(\cdot,\vartheta)\in \mathcal{L}_2(1,4\Delta)$. Moreover, since $\int_{\abs{x}>1}x^{4+\gamma_2}\nu(dx)<\infty$, it holds $\tilde{Y}_u\in L^{4+\gamma_2}$. As explained above, $\tilde{Y}_u$ is $\theta$-weakly dependent with exponentially decaying $\theta$-coefficients. Hence (e$^*$) from Remark \ref{remark:asymptoticnormalityfinitememory} holds.\\
In view of (f) we first apply Hoelder's inequality to (\ref{equation:secondphiinequalityinx}) and obtain for $X_0,X_1,Y_0,Y_1\in L^2$ and $\vartheta\in\Theta$  
\begin{align*}
\norm{\frac{d^2}{d\vartheta^2}\Phi^{LS}(X_0,X_1,\vartheta)-\frac{d^2}{d\vartheta^2}\Phi^{LS}(Y_0,Y_1,\vartheta)}_{L^1}&\leq 24\Delta^2 \max_{k=0,1}\{\norm{X_k}_{L^2}\vee\norm{Y_k}_{L^2}\}\\
&\qquad \left(\norm{X_0-Y_0}_{L^2}+\norm{X_1-Y_1}_{L^2}\right),
\end{align*}
such that $\frac{d^2}{d\vartheta^2}\Phi^{LS}(\cdot,\vartheta)\in\mathcal{L}_{\infty}^{1,2}(\bar{\alpha})$, where $\bar{\alpha}_k=1$ for $k=0,1$ and $\bar{\alpha}_k=0$ for $k\geq2$. Now, let $x_0,x_1\in\R$ and $\vartheta_1,\vartheta_2\in\Theta$. Then,
\begin{align*}
&\abs{\frac{d^2}{d\vartheta^2}\Phi^{LS}(x_0,x_1,\vartheta_1)-\frac{d^2}{d\vartheta^2}\Phi^{LS}(x_0,x_1,\vartheta_2)}\\
&\leq 2\Delta^2\abs{x_0x_1}\abs{e^{-\Delta\vartheta_2}-e^{-\Delta\vartheta_1}}+4\Delta^2x_1^2 \abs{e^{-2\Delta\vartheta_1}-e^{-2\Delta\vartheta_2}}\leq \Delta^3\left(4x_0^2+12x_1^2\right) \abs{\vartheta_1-\vartheta_2},
\end{align*}
such that $\frac{d^2}{d\vartheta^2}\Phi^{LS}(x,\cdot)\in\mathcal{L}_{1}(0,D_2(1+\sum_{k=0}^\infty\bar{\beta}_k\abs{x_k}^2))$ for $D_2=\Delta^3$ and $\bar{\beta}_0=4$, $\bar{\beta}_1=12$ and $\bar{\beta}_k=0$ for $k\geq2$. Overall, (f) holds.\\
If \hyperref[observations:O2]{(O2)} holds we show (g$1^*$) and (g$2^*$) from Remark \ref{remark:asymptoticnormalityfinitememory}. Indeed, we have already shown (g$1^*$). Moreover $|\frac{d^2}{d\vartheta^2}\Phi^{LS}(x_0,x_1,\vartheta)|\leq6\Delta^2(\abs{x_0}+\abs{x_1})^2$ and (\ref{equation:secondphiinequalityinx}) implies that $\frac{d^2}{d\vartheta^2}\Phi^{LS}(\cdot,\vartheta)\in\mathcal{L}_{2}(1,6\Delta^2)$ for all $\vartheta\in\Theta$ and $x_0,x_1\in\R$. From $\int_{\abs{x}>1}x^{2+\gamma_1}\nu(dx)<\infty$, it follows $\tilde{Y}_u\in L^{2+\gamma_1}$. Since $\tilde{Y}_u$ is $\theta$-weakly dependent, the assertion of Remark \ref{remark:asymptoticnormalityfinitememory} holds. It is left to investigate the asymptotic variance $\Sigma(u)$. We obtain
\begin{align*}
I(u,k)=4\Delta^2e^{-2 a(u)\Delta}\Big(&E\left[\tilde{Y}_u(0)\tilde{Y}_u(\Delta)\tilde{Y}_u(k\delta+\Delta)\tilde{Y}_u(k\delta)\right]-e^{-a(u)\Delta}E\left[\tilde{Y}_u(0)\tilde{Y}_u(\Delta)\tilde{Y}_u(k\delta)^2\right]\\
&-e^{-a(u)\Delta}E\left[\tilde{Y}_u(0)^2\tilde{Y}_u(k\delta+\Delta)\tilde{Y}_u(k\delta)\right]+e^{-2a(u)\Delta}E\left[\tilde{Y}_u(0)^2\tilde{Y}_u(k\delta)^2\right]\Big).
\end{align*}
An application of Lemma \ref{lemma:4thmixedmoments} gives
\begin{align*}
I(u,k)=&\frac{\int_\R x^4\nu(dx)\Delta^2e^{-2 a(u)\Delta}}{a(u)}\left(e^{-2a(u)(k\delta+\Delta)}-2e^{-a(u)\Delta}e^{-a(u)(2k\delta+\Delta)}+e^{-2a(u)\Delta}e^{-2a(u)k\delta)} \right)\\
&+\frac{\Delta^2\Sigma_L^2e^{-2 a(u)\Delta}}{a(u)^2}\Big(e^{-2a(u)\Delta}+e^{-a(u)(k\delta+\Delta+\abs{k\delta-\Delta})}+e^{-2a(u)k\delta}-e^{-2a(u)\Delta}\\
&\qquad-2e^{-a(u)(k\delta+\Delta+\abs{k\delta-\Delta})}-e^{-2a(u)\Delta}-2e^{-2a(u)(k\delta+\Delta)}+ e^{-2a(u)\Delta}+2e^{-2a(u)(k\delta+\Delta)}\Big)\\
=&\frac{\Delta^2\Sigma_L^2e^{-2 a(u)\Delta}}{a(u)^2}\left(e^{-2a(u)k\delta}-e^{-a(u)(k\delta+\Delta+\abs{k\delta-\Delta})}\right).
\end{align*}
First, we note that 
\begin{align*}
I(u,0)=\frac{\Delta^2\Sigma_L^2e^{-2 a(u)\Delta}}{a(u)^2}\left(1-e^{-2a(u)\Delta}\right).
\end{align*}
Moreover, straightforward calculations give
\begin{align}
\sum_{k=1}^\infty I(u,k)&=\frac{\Delta^2\Sigma_L^2e^{-2 a(u)\Delta}}{a(u)^2}\left(\sum_{k=1}^\infty e^{-2a(u)k\delta}- \sum_{k=1}^{\left\lceil\Delta/\delta \right\rceil-1}e^{-2a(u)\Delta}-\sum_{k=\left\lceil\Delta/\delta \right\rceil}^\infty e^{-2a(u)k\delta} \right)\notag\\
&=\frac{\Delta^2\Sigma_L^2e^{-2 a(u)\Delta}}{a(u)^2} \sum_{k=1}^{\left\lceil\Delta/\delta \right\rceil-1}\left(e^{-2a(u)k\delta}-e^{-2a(u)\Delta}\right)\label{eq:startpointforIpositive}\\
&=\frac{\Delta^2\Sigma_L^2e^{-2 a(u)\Delta}}{a(u)^2}\left(e^{-2a(u)\delta}\frac{1-e^{-2a(u)\delta(\left\lceil\Delta/\delta \right\rceil-1)}}{1-e^{-2a(u)\delta}}-e^{-2a(u)\Delta}(\left\lceil\Delta/\delta \right\rceil-1) \right).\notag
\end{align}
Hence,
\begin{align*}
I(u)
=\frac{\Delta^2\Sigma_L^2e^{-2 a(u)\Delta}}{2a(u)^2}\begin{cases}1+2e^{-2a(u)\delta}\frac{1-e^{-2a(u)\delta(\left\lceil\Delta/\delta \right\rceil-1)}}{1-e^{-2a(u)\delta}}-e^{-2a(u)\Delta}(2\left\lceil\Delta/\delta \right\rceil-1),& \text{if \hyperref[observations:O1]{(O1)} holds,}\\
1-e^{-2a(u)\Delta},& \text{if \hyperref[observations:O2]{(O2)} holds.}
\end{cases}
\end{align*}
If \hyperref[observations:O2]{(O2)} holds, it is easy to see that $I(u)>0$. If \hyperref[observations:O1]{(O1)} holds, it is enough to additionally observe that all summands in (\ref{eq:startpointforIpositive}) are non-negative since $k\delta\leq \Delta$ for all $k=1,\ldots,\left\lceil\Delta/\delta \right\rceil-1$.
Finally,
\begin{align*}
V(u)=4\Delta^2e^{-2a(u)\Delta}E[\tilde{Y}_u(0)^2]-2\Delta^2e^{-a(u)\Delta}E[\tilde{Y}_u(\Delta)\tilde{Y}_u(0)]=\frac{\Delta^2e^{-2a(u)\Delta}\Sigma_L}{a(u)}>0.
\end{align*}
\end{proof}

\subsection{Proof for Section \ref{sec5-2}}
\label{sec7-5}

\begin{proof}[Proof of Proposition \ref{proposition:qmlecontrast}]
First, we note that $\sup_{\vartheta\in\Theta}|\Phi^{LL}(0,\vartheta)|<\infty$. For all $\vartheta\in\Theta$ and sequences $X=(X_{1-k})_{k\in\N_0}\in\ell^\infty(L^2)$ and $Y=(Y_{1-k})_{k\in\N_0}\in\ell^\infty(L^2)$, 
we have
\begin{gather*}
\norm{\Phi^{LL}(X,\vartheta)-\Phi^{LL}(Y,\vartheta)}_{L^1}=\frac{1}{V_\vartheta}\norm{\varepsilon_{\vartheta}(X)^2-\varepsilon_{\vartheta}(Y)^2}_{L^1},
\end{gather*}
where $\varepsilon_{\vartheta}(X)$ and $\varepsilon_{\vartheta}(Y)$ are the analogues of (\ref{equation:varepsilontheta}), defined in terms of $X$ and $Y$. From Hoelder's inequality we obtain
\begin{align*}
\leq& \frac{1}{V_\vartheta} \Bigg(\norm{X_1}_{L^2}+ \Bigg\lVert  B_\vartheta'\sum_{n=1}^\infty \left( e^{\Delta A_\vartheta}-K_\vartheta B_\vartheta' \right)^{n-1}K_\vartheta X_{1-n}\Bigg\rVert_{L^2}\\
&+\norm{Y_1}_{L^2}+ \Bigg\lVert  B_\vartheta'\sum_{n=1}^\infty \left(e^{\Delta A_\vartheta}-K_\vartheta B_\vartheta'\right)^{n-1}K_\vartheta Y_{1-n}\Bigg\rVert_{L^2}\Bigg)\\
&\times\left( \Lnorm{X_1-Y_1} +\Lnorm{B_\vartheta'\sum_{n=1}^\infty \left( e^{\Delta A_\vartheta}-K_\vartheta B_\vartheta' \right)^{n-1}K_\vartheta\left(Y_{1-n}-X_{1-n}\right) }\right)\\
\leq&  \frac{1}{V_\vartheta} \Bigg(4 \sup_{k\in\N_0}\Big\{\Lnorm{X_{1-k}}\vee\Lnorm{Y_{1-k}}\Big\}C_1 \sum_{n=1}^\infty \left\lVert\left( e^{\Delta A_\vartheta}-K_\vartheta B_\vartheta' \right)^{n-1}\right\rVert\Bigg)\\
&\times \left( \Lnorm{X_1-Y_1}+C_1 \sum_{n=1}^\infty \left\lVert\left( e^{\Delta A_\vartheta}-K_\vartheta B_\vartheta' \right)^{n-1}\right\rVert \Lnorm{X_{1-n}-Y_{1-n}}\right)
\end{align*} 
for some constant $C_1>0$. In the following we bound the expression $( e^{\Delta A_\vartheta}-K_\vartheta B_\vartheta' )^{n-1}$, $n\in\N$. From Proposition \ref{proposition:kalmanfilter} (b), \hyperref[assumption:C3]{(C3)} and since eigenvalues are continuous functions of the entries of a matrix (see \cite[Fact 10.11.2]{B2009}), we obtain
\begin{align*}
\max\left\{\abs{\lambda},\lambda\in\sigma\left(e^{\Delta A_\vartheta}-K_\vartheta B_\vartheta'\right)\right\}\leq \rho
\end{align*}
for some positive number $\rho<1$ and all $\vartheta\in\Theta$, where $\sigma(\cdot)$ denotes the spectrum of a matrix. Let $\varepsilon>0$ be small enough, such that $\rho+\varepsilon<1$. Then, using Gelfand's formula there exists $N_0\in\N$ such that 
\begin{align}\label{equation:lemmafromgelfand}
\norm{\left( e^{\Delta A_\vartheta}-K_\vartheta B_\vartheta' \right)^{n-1}}\leq\begin{cases}(\rho+\varepsilon)^{n-1},&\text{ for all }n\geq N_0\\ 
S_\rho^{n-1},&\text{ for all } n<N_0.
\end{cases}
\end{align}
for some constant $S_\rho\geq1$. We set
\begin{align}\label{equation:alphan}
\alpha_n=\max(1,C_1)\begin{cases}1&, n=0,\\
S_\rho^{n-1}&,1\leq n< N_0,\\
(\rho+\varepsilon)^{n-1}&, N_0\leq n.
\end{cases}
\end{align}
Then, we readily obtain $\sum_{n=0}^\infty n \alpha_n<\infty$ and $\Phi^{LL}(\cdot,\vartheta)\in \mathcal{L}_\infty^{1,2}(\alpha)$ for all $\vartheta\in\Theta$.\\
It remains to show $\Phi^{LL}(x,\cdot)\in\mathcal{L}_d(0,D_1(1+\sum_{k=0}^\infty \beta_kx_{1-k}^2))$ for any real sequence $x=(x_{1-k})_{k\in\N_0}$. Define $S_B=\sup_{\vartheta\in\Theta}\norm{B_\vartheta}$, $S_V=\sup_{\vartheta\in\Theta}\norm{V_\vartheta}$ and $S_K=\sup_{\vartheta\in\Theta}\norm{K_\vartheta}$. First,
\begin{gather*}
\abs{\Phi^{LL}(x,\vartheta_1)-\Phi^{LL}(x,\vartheta_2)}\leq\abs{\log(V_{\vartheta_1})-\log(V_{\vartheta_2})}+\abs{\frac{\varepsilon_{\vartheta_1}^2(x)}{V_{\vartheta_1}}-\frac{\varepsilon_{\vartheta_2}^2(x)}{V_{\vartheta_2}}}=:P_1+P_2.
\end{gather*}
Since, $\vartheta\mapsto V_{\vartheta}$ is Lipschitz with constant $L_V$ and bounded from below by $C_V>0$ we obtain
\begin{gather*}
P_1\leq \frac{L_V}{C_V} \norm{\vartheta_1-\vartheta_2}.
\end{gather*}
For $P_2$ it holds
\begin{align*}
P_2\leq& \frac{1}{C_V^2} \left(\abs{V_{\vartheta_2}\varepsilon_{\vartheta_1}^2(x)-V_{\vartheta_1}\varepsilon_{\vartheta_1}^2(x)}+\abs{V_{\vartheta_1}\varepsilon_{\vartheta_1}^2(x)-V_{\vartheta_1}\varepsilon_{\vartheta_2}^2(x)}\right)\\
\leq&  \frac{1}{C_V^2} \left(\abs{V_{\vartheta_1}-V_{\vartheta_2}}\varepsilon_{\vartheta_1}^2(x)+2S_V\left(\sup_{\vartheta\in\Theta}\abs{\varepsilon_{\vartheta}(x)}\right)\abs{\varepsilon_{\vartheta_1}(x)-\varepsilon_{\vartheta_2}(x)}\right)=:\frac{1}{C_V^2}(R_1+R_2).
\end{align*}
Again, due to (\ref{equation:lemmafromgelfand}) we obtain
\begin{align}
\abs{\varepsilon_{\vartheta}(x)}&\leq \abs{x_1}+S_BS_K\left(\sum_{n=1}^\infty \norm{ \left(e^{\Delta A_\vartheta}-K_\vartheta B_\vartheta \right)^{n-1}} \abs{x_{1-n}}\right)\notag\\
&\leq\abs{x_1}+S_BS_K\left(\sum_{n=1}^\infty \alpha_n \abs{x_{1-n}}\right) \text{ and}\label{equation:varepsilonineq}\\
\varepsilon_{\vartheta}^2(x)&
\leq2 x_1^2+2S_B^2S_K^2\left( \sum_{n=1}^\infty \alpha_n \abs{x_{1-n}}\right)^2\notag.
\end{align}
From Hoelder's inequality for sequences we obtain
\begin{gather}
\leq2 x_1^2+2S_B^2S_K^2\left( \sum_{n=1}^\infty\alpha_n x_{1-n}^2\right)\left( \sum_{n=1}^\infty\alpha_n \right).\label{equation:varepsilonsquaredineq}
\end{gather}
Because $\vartheta\mapsto V_\vartheta$ is Lipschitz with constant $L_V$ we get from (\ref{equation:varepsilonsquaredineq}) 
\begin{gather*}
R_1\leq \norm{\vartheta_1-\vartheta_2}\sum_{n=0}^\infty \beta^{(1)}_n x_{1-n}^2,
\end{gather*}
where $\beta^{(1)}_0=2L_V$ and $\beta^{(1)}_n= 2L_VS_B^2S_K^2\left(\sum_{k=1}^\infty \alpha_k \right)\alpha_n$ for $n\geq1$. Moreover,
\begin{align*}
\abs{\varepsilon_{\vartheta_1}(x)-\varepsilon_{\vartheta_2}(x)}\leq&S_K\norm{B_{\vartheta_1}'-B_{\vartheta_2}'} \left(\sum_{n=1}^\infty \norm{\left(e^{\Delta A_{\vartheta_1}}-K_{\vartheta_1} B_{\vartheta_1}\right)^{n-1} }\abs{x_{1-n}}\right)\\
&+S_BS_K\left( \sum_{n=2}^\infty \norm{\left(e^{\Delta A_{\vartheta_1}}-K_{\vartheta_1} B_{\vartheta_1} \right)^{n-1}- \left(e^{\Delta A_{\vartheta_2}}-K_{\vartheta_2} B_{\vartheta_2} \right)^{n-1}}\abs{x_{1-n}}\right)\\
&+S_B\norm{K_{\vartheta_1}-K_{\vartheta_2}}\left( \sum_{n=1}^\infty\norm{\left(e^{\Delta A_{\vartheta_2}}-K_{\vartheta_2} B_{\vartheta_2}\right)^{n-1}}\abs{x_{1-n}}\right)=:S_1+S_2+S_3.
\end{align*}
Since $\vartheta\mapsto B_\vartheta$ and $\vartheta\mapsto K_\vartheta$ are Lipschitz with constants $L_B$ and $L_K$ we obtain 
\begin{align*}
S_1+S_3 \leq  (L_BS_K+S_B L_K) \norm{\vartheta_1-\vartheta_2} \left( \sum_{n=1}^\infty \alpha_n \abs{x_{1-n}}\right).
\end{align*}
In view of $S_2$, a modification of \cite[Lemma B.4]{H2008} and an application of (\ref{equation:lemmafromgelfand}) give
\begin{align*}
&\norm{\left(e^{\Delta A_{\vartheta_1}}-K_{\vartheta_1} B_{\vartheta_1} \right)^{n-1}- \left(e^{\Delta A_{\vartheta_2}}-K_{\vartheta_2} B_{\vartheta_2} \right)^{n-1}}\\
&\leq \left(\norm{e^{\Delta A_{\vartheta_1}}-e^{\Delta A_{\vartheta_2}}}+\norm{K_{\vartheta_2} B_{\vartheta_2}-K_{\vartheta_1} B_{\vartheta_1}}\right)\\
&\quad\times \sum_{i=0}^{n-2} \norm{\left(e^{\Delta A_{\vartheta_1}}-K_{\vartheta_1} B_{\vartheta_1} \right)^{i}}\norm{\left(e^{\Delta A_{\vartheta_2}}-K_{\vartheta_2} B_{\vartheta_2} \right)^{n-i-2}}\\
&\leq \left(\norm{e^{\Delta A_{\vartheta_1}}-e^{\Delta A_{\vartheta_2}}}+\norm{K_{\vartheta_2} B_{\vartheta_2}-K_{\vartheta_1} B_{\vartheta_1}}\right)\\
&\quad\times \sum_{i=0}^{n-2}\left(S_\rho^i \mathbb{1}_{\{i-1<N_0\}} +\left(\rho+\varepsilon\right)^i  \mathbb{1}_{\{i-1\geq N_0\}}\right)\left(S_\rho^{n-i-2} \mathbb{1}_{\{n-i-1<N_0\}} +\left(\rho+\varepsilon\right)^{n-i-2}  \mathbb{1}_{\{n-i-1\geq N_0\}}\right)\\
&\leq \left(\norm{e^{\Delta A_{\vartheta_1}}\!-\!e^{\Delta A_{\vartheta_2}}}\!+\!\norm{K_{\vartheta_2} B_{\vartheta_2}\!-\!K_{\vartheta_1} B_{\vartheta_1}}\right)\\
&\quad\times\Bigg((N_0-1)S_\rho^{2N_0-2} \mathbb{1}_{\{n<2N_0\}} \!+\!\left(\rho+\varepsilon\right)^{n-1} \left(\frac{S_\rho}{\rho\!+\!\varepsilon}\right)^{\!N_0} \sum_{i=0}^{N_0-2}\left(\rho\!+\!\varepsilon\right)^{N_0-i-1}\\&\qquad\quad\!+\!S_\rho^{N_0-1}\sum_{i=\left\lceil \frac{n}{2}\right\rceil-1}^{n-2}\left(\rho\!+\!\varepsilon\right)^{i}\!+\!(n\!-\!1)\left(\rho\!+\!\varepsilon\right)^{n-2} \Bigg)
\end{align*}
Since  $\vartheta\mapsto A_{\vartheta}$, $\vartheta\mapsto K_{\vartheta} B_{\vartheta}$ and the matrix exponential on a compact subset of $M_{p\times p}(\R)$ are Lipschitz with constants $L_A$, $L_{KB}$ and $L_{exp}$, respectively, we can bound the above expression by $\norm{\vartheta_1-\vartheta_2 }\left(\Delta L_AL_{exp}+L_{KB}\right)\beta^{(2)}_{n}$, where for $n\geq2$
\begin{align*}
\beta^{(2)}_n=&(N_0-1) S_\rho^{2N_0-2} \mathbb{1}_{\{n<2N_0\}}\!+\!\left(\!\!\left(\frac{S_\rho}{\rho+\varepsilon}\right)^{N_0}\!\left(\frac{\left(\rho+\varepsilon\right)-\left(\rho+\varepsilon\right)^{N_0}}{1-\left(\rho+\varepsilon\right)}\right)\!+\!\frac{S_\rho^{N_0}}{1-\left(\rho+\varepsilon\right)}\!+\!(n-1)\!\!\right)\\
&\qquad\qquad\qquad\qquad\qquad\times \left(\rho+\varepsilon\right)^{ \frac{n}{2}-1}.
\end{align*}
It is clear that $\sum_{n=2}^\infty n\beta^{(2)}_n<\infty$. Thus
\begin{align*}
S_2\leq& \norm{\vartheta_1-\vartheta_2}S_BS_K\left(\Delta L_AL_{exp}+L_{KB}\right)\left(\sum_{n=2}^\infty \beta^{(2)}_{n} \abs{x_{1-n}}\right).
\end{align*}
Using (\ref{equation:varepsilonineq})
\begin{align*}
R_2\leq& \norm{\vartheta_1-\vartheta_2} C_2  \left( x_1^2+\alpha_1 x_0^2 + \sum_{n=2}^\infty \left( \alpha_n+\beta^{(2)}_{n}\right) x_{1-n}^2\right)= \norm{\vartheta_1-\vartheta_2} C_2 \sum_{n=0}^\infty \beta^{(3)}_n x_{1-n}^2,
\end{align*}
where 
\begin{gather*}
\beta^{(3)}_0=1,\quad \beta^{(3)}_1= \alpha_1,\quad \beta^{(3)}_n= \alpha_n+\beta^{(2)}_{n },\quad n\geq2
\end{gather*}
and the finite constant $C_2$ is given by
\begin{align*}
C_2=&4S_V\Bigg(1+2\left(S_B^2S_K^2+\left(L_BS_K+S_BL_K\right)^2\right)\left(\sum_{n=1}^\infty \alpha_n\right)\\
&\qquad\quad+S_B^2S_K^2\left(\Delta L_AL_{exp}+L_{KB} \right)^2\left(\sum_{n=2}^\infty\beta^{(2)}_{n}\right) \Bigg).
\end{align*}
Finally,
\begin{gather*}
\abs{\Phi^{LL}(x,\vartheta_1)-\Phi^{LL}(x,\vartheta_2)}\leq \norm{\vartheta_1-\vartheta_2} D_1\left(1+\sum_{n=0}^\infty\beta_n x_{1-n}^2 \right),
\end{gather*}
where $D_1=\frac{L_V}{C_V}+\frac{C_2}{C_V^2}$ and $\beta_n= \beta^{(1)}_n+\beta^{(3)}_n$, for all $n\geq0$, such that $\sum_{n=0}^\infty n\beta_n<\infty$ and $\Phi^{LL}(x,\cdot)\in\mathcal{L}_d(0,D_1(1+\sum_{k=0}^\infty \beta_k x_{1-k}^2))$.
To show part (c) we note that
\begin{align*}
\tilde{Y}_u(t)=\tilde{Y}^{\vartheta^*(u)}(t)=\int_{-\infty}^t B_{\vartheta^*(u)}'e^{A_{\vartheta^*(u)}}C_{\vartheta^*(u)}L(ds)
\end{align*} 
is of the form (\ref{equation:infinitememorylevydrivenmovingaverage}). From \hyperref[assumption:C2]{(C2)} it follows, that the kernel is of exponential decay and therefore also in $L^1\cap L^2$. As shown above $\Phi^{LL}(\cdot,\vartheta)\in\mathcal{L}_{\infty}^{1,2}(\alpha)$ for all $\vartheta\in\Theta$. Moreover, $\gamma+\int_{\abs{x}>1}x\nu(dx)=0$ and (\ref{eq:condlevyproc}) holds. Therefore, condition (c) from Theorem \ref{theorem:consistency} follows from Corollary  \ref{corollary:thetaweakdependenceinfinitememory}.
\end{proof}

\subsection{Proof for Section \ref{sec5-3}}
\label{sec7-6}

\begin{proof}[Proof of Theorem \ref{equation:consistencymodifiedQMLE}]
We use the same notation as in the proof of Proposition \ref{proposition:qmlecontrast}. First, we note that $\sup_{\vartheta\in\Theta}|\tilde\Phi^{LL}(0,\vartheta)|<\infty$. Following the proof of Theorem \ref{theorem:consistency} it is sufficient to show that for any $\eta>0$ there exists $\lambda>0$ such that
\begin{align}
\norm{\tilde{M}_N(\vartheta)-M(\vartheta)}_{L^1}&\underset{N\rightarrow\infty}{\longrightarrow}0 \text{ for all }\vartheta\in\Theta\text{ and}\label{equation:truncatedQMLEproof1}\\
P\left( \sup_{\norm{\vartheta_1-\vartheta_2}<\lambda}\left|\tilde{M}_N(\vartheta_1)-\tilde{M}_N(\vartheta_2) \right|>\eta\right) &\underset{N\rightarrow\infty}{\longrightarrow}0.\label{equation:truncatedQMLEproof2}
\end{align}
We start by proving (\ref{equation:truncatedQMLEproof1}). From Proposition \ref{proposition:qmlecontrast} it follows that $\norm{M_N(\vartheta)-M(\vartheta)}_{L^1}\underset{N\rightarrow\infty}{\longrightarrow}0$, such that it is left to show that $\lVert\tilde{M}_N(\vartheta)-M_N(\vartheta)\rVert_{L^1}\underset{N\rightarrow\infty}{\longrightarrow}0$. Now
\begin{align*}
\norm{\tilde{M}_N(\vartheta)-M_N(\vartheta)}_{L^1}\leq&\norm{\frac{\delta_N}{b_N}\sum_{i=-m_N}^{m_N-1}K\left(\frac{\tau_i^N-u}{b_N} \right)\frac{1}{V_\vartheta}\left(\varepsilon_{\vartheta}^2(Y_N^{\vartheta^*})-\tilde{\varepsilon}_{\vartheta,i,m_N}^2(Y_N^{\vartheta^*})\right)}_{L^1}\\
&+ \Bigg\lVert\frac{\delta_N}{b_N} K\left(\frac{\tau_{m_N}^N-u}{b_N}\right)\Bigg( \log(2\pi)-\log(V_\vartheta)-\frac{1}{V_\vartheta}\Bigg(Y_N^{\vartheta^*}(\tau_{m_N+1}^N)\\
&-B_\vartheta'\sum_{n=1}^\infty \left(e^{\Delta A_\vartheta}-K_\vartheta B_\vartheta' \right)^{n-1}K_\vartheta Y_N^{\vartheta^*}(\tau_{m_N+1-n}) \Bigg)^2\Bigg)\Bigg\rVert_{L^1}=:P_1+P_2.
\end{align*}
From \hyperref[assumption:LS]{(LS)} we obtain $\sup_{t\in\R}\lVert Y^{\vartheta^*}_N(t)\rVert_{L^2}\leq S_Y$ for some constant $S_Y>0$ and all $N\in\N$. Thus, using Hoelder's inequality for sequences
\begin{gather*}
P_2\leq \frac{\delta_N}{b_N}|K|_\infty \left(\log(2\pi)+\abs{\log(V_\vartheta)}+\frac{2S_Y^2}{C_V}+\frac{2S_B^2S_K^2S_Y^2}{C_V}\left(\sum_{n=1}^\infty \alpha_{n}\right)^2\right)\underset{N\rightarrow\infty}{\longrightarrow}0,
\end{gather*}
where $\alpha_n$ is defined in (\ref{equation:alphan}). For $P_1$ we obtain from Hoelder's inequality
\begin{align*}
P_1\leq&\frac{\delta_N}{b_N}\sum_{i=-m_N}^{m_N-1}\frac{|K|_\infty}{C_V}\left(\Lnorm{\varepsilon_{\vartheta}(Y_N^{\vartheta^*})+\tilde{\varepsilon}_{\vartheta,i,m_N}(Y_N^{\vartheta^*})}\Lnorm{\varepsilon_{\vartheta}(Y_N^{\vartheta^*})-\tilde{\varepsilon}_{\vartheta,i,m_N}(Y_N^{\vartheta^*})}\right).
\end{align*}
Similar arguments as in the first part of the proof of Proposition \ref{proposition:qmlecontrast} give 
\begin{align*}
\Lnorm{\varepsilon_{\vartheta}(Y_N^{\vartheta^*})+\tilde{\varepsilon}_{\vartheta,i,m_N}(Y_N^{\vartheta^*})}\leq&2S_Y^2\left(1+S_BS_K\sum_{n=1}^\infty \alpha_n\right)=C_1<\infty.
\end{align*}
Thus
\begin{align*}
P_1\leq& \frac{C_1|K|_\infty}{C_V}\frac{\delta_N}{b_N}\sum_{i=-m_N}^{m_N}\Lnorm{\varepsilon_{\vartheta}(Y_N^{\vartheta^*})-\tilde{\varepsilon}_{\vartheta,i,m_N}(Y_N^{\vartheta^*})}\\
\leq& \frac{C_1S_Y^2|K|_\infty S_BS_K}{C_V}\frac{\delta_N}{b_N}\sum_{i=-m_N}^{m_N}\sum_{n=m_N+i+2}^\infty \norm{(e^{\Delta A_\theta}-K_\theta B_\theta')^{n-1}}\\
\leq& \frac{C_1S_Y^2|K|_\infty S_BS_K}{C_V}\frac{\delta_N}{b_N}\sum_{i=0}^{2m_N}\sum_{n=1}^\infty\norm{(e^{\Delta A_\theta}-K_\theta B_\theta')^{n+i}}\\
\leq& \frac{C_1S_Y^2|K|_\infty S_BS_K}{C_V}\frac{\delta_N}{b_N}\left(\sum_{n=0}^\infty\alpha_n\right)^2\underset{N\rightarrow\infty}{\longrightarrow}0.
\end{align*}
To show (\ref{equation:truncatedQMLEproof2}) it is enough to observe that analogous to the second part of the proof of Proposition \ref{proposition:qmlecontrast} it holds for $\vartheta_1\neq \vartheta_2$ and any real sequence $x=(x_{1-k})_{k\in\N_0}$
\begin{align*}
\abs{\tilde{\Phi}_{i,m_N}^{LL}\left(x,\vartheta_1\right)-\tilde{\Phi}_{i,m_N}^{LL}\left(x,\vartheta_2\right)}\leq \norm{\vartheta_1-\vartheta_2}D_1\left(1+\sum_{n=0}^\infty\beta_n x_{1-n}^2 \right).
\end{align*}
Finally, similar steps as in (\ref{equation:stochequiproof1}) and (\ref{equation:stochequiproof2}) ensure that (\ref{equation:truncatedQMLEproof2}) holds.
\end{proof}

\subsection{Proof for Section \ref{sec5-4}}
\label{sec7-7}

The following auxiliary lemmata will be essential for the proof of Theorem \ref{theorem:whittleconsistent}.

\begin{Lemma}\label{lemma:convergencegamma}
Assume that \hyperref[assumption:C0]{(C0)} is satisfied for p=2, \hyperref[observations:O1]{(O1)} holds for $N\delta_N=\Delta$ and that the localizing kernel $K$ is continuous and positive or the non-continuous rectangular kernel (\ref{equation:rectangularkernel}). Then, for all $h\in\N_0$
\begin{align*}
\hat{\Gamma}_N^{loc}(h)
&\overset{L^1}{\underset{N\rightarrow\infty}{\longrightarrow}}E[\tilde{Y}_u(0)\tilde{Y}_u(h)].
\end{align*}
\end{Lemma}
\begin{proof}
In view of \cite[Theorem 5.14]{SS2021} it is enough to observe
\begin{align*}
&\norm{\hat{\Gamma}_N^{loc}(h)-\frac{\delta_N}{b_N} \sum_{j=-m_N}^{m_N}K\left(\frac{\tau_j^N-u}{b_N} \right)Y_N^{\vartheta^*}(\tau_j^N)Y_N^{\vartheta^*}(\tau_{j+h}^N)}_{L^1}\\
&\leq \frac{\delta_N}{b_N}  \sum_{j=-m_N}^{m_N}\abs{\sqrt{K\left(\frac{\tau_{j+h}^N-u}{b_N}\right)K\left(\frac{\tau_j^N-u}{b_N}\right)}-K\left(\frac{\tau_j^N-u}{b_N}\right)} \norm{Y_N^{\vartheta^*}(\tau_j^N)Y_N^{\vartheta^*}(\tau_{j+h}^N)}_{L^1}\\
&+ \frac{\delta_N}{b_N}  \sum_{j=m_N-h+1}^{m_N}\sqrt{K\left(\frac{\tau_{j+h}^N-u}{b_N}\right)K\left(\frac{\tau_j^N-u}{b_N}\right)}  \norm{Y_N^{\vartheta^*}(\tau_j^N)Y_N^{\vartheta^*}(\tau_{j+h}^N)}_{L^1} =:P_1+P_2.
\end{align*}
Since $\lVert Y_N^{\vartheta^*}(\tau_j^N)Y_N^{\vartheta^*}(\tau_{j+h}^N)\rVert_{L^1} \leq \lVert Y_N^{\vartheta^*}(\tau_j^N)\rVert_{L^2} \lVert Y_N^{\vartheta^*}(\tau_{j+h}^N)\rVert_{L^2}<S_Y^2$ for some constant $S_Y>0$, we obtain
\begin{gather*}
P_2\leq  \frac{hS_Y^2 \abs{K}_\infty\delta_N}{b_N}\underset{N\rightarrow\infty}{\longrightarrow}0.
\end{gather*}
From the continuity of $K$ it follows that, given any $\varepsilon>0$, $\Big|K\Big(\tfrac{\tau_{j+h}^N-u}{b_N}\Big)-K\Big(\tfrac{\tau_j^N-u}{b_N}\Big)\Big|<\varepsilon$ for sufficiently large $N$. Thus it holds
\begin{align*}
P_1&\leq  \frac{S_Y^2\delta_N}{b_N}  \sum_{j=-m_N}^{m_N}\sqrt{\abs{K\left(\frac{\tau_{j+h}^N-u}{b_N}\right)-K\left(\frac{\tau_j^N-u}{b_N}\right)}K\left(\frac{\tau_j^N-u}{b_N}\right)}\\
&\leq  \sqrt{\abs{K}_\infty}S_Y^2 (2m_N+1)\frac{ \delta_N}{b_N}  \sqrt{\varepsilon}
\underset{N\rightarrow\infty}{\longrightarrow}  \sqrt{\abs{K}_\infty}2 S_Y^2 \sqrt{\varepsilon}.
\end{align*}
For the non-continuous rectangular kernel (\ref{equation:rectangularkernel}) we can directly bound $P_1$ and obtain
\begin{align*}
P_1\leq S_Y^2 \frac{\delta_N}{b_N}\frac{1}{2} \sum_{j=-m_N}^{m_N}\abs{\mathbb{1}_{\left\{(j+h)\frac{\delta_N}{b_N}\in[-1,1]\right\}}-\mathbb{1}_{\left\{j\frac{\delta_N}{b_N}\in[-1,1]\right\}}}\mathbb{1}_{\left\{j\frac{\delta_N}{b_N}\in[-1,1]\right\}} \leq S_Y^2 \frac{\delta_N}{b_N} \frac{h}{2}\underset{N\rightarrow\infty}{\longrightarrow}0.
\end{align*}
\end{proof}

\begin{Lemma}\label{lemma:f_Yislipschitz}
Assume that \hyperref[assumption:C1]{(C1)} - \hyperref[assumption:C7]{(C7)} hold. Then, there exists a constant $f_{inf}$ such that $f_{\tilde{Y}}^{(\Delta)}(\omega,\vartheta)>f_{inf}>0$ for all $\omega\in[-\pi,\pi]$ and $\vartheta\in\Theta$. Moreover, there exists a constant $L_f>0$ such that
\begin{gather}\label{eq:f_Yislipschitz}
\sup_{\omega\in[-\pi,\pi]}\abs{f_{\tilde{Y}}^{(\Delta)}(\omega,\vartheta_1)-f_{\tilde{Y}}^{(\Delta)}(\omega,\vartheta_2)}\leq L_f\norm{\vartheta_1-\vartheta_2} \text{ for all } \vartheta_1,\vartheta_2\in\Theta.
\end{gather}
\end{Lemma}
\begin{proof}
Since $(\omega,\vartheta)\mapsto f_{\tilde{Y}}^{(\Delta)}(\omega,\vartheta)^{-1}$ is continuous (see e.g. the proof of \cite[Proposition 2]{FHM2020}) and $\Theta \times [-\pi,\pi]$ is compact, there exists a constant $f_{inf}>0$ such that $f_{\tilde{Y}}^{(\Delta)}(\omega,\vartheta)>f_{inf}>0$.\\
In order to show (\ref{eq:f_Yislipschitz}) we first note that, using the representation of $f_{\tilde{Y}}^{(\Delta)}(\omega,\vartheta)$ from (\ref{eq:spectraldensitysampledprocess}), we obtain
\begin{align}\label{eq:decompositionf_Y}
\abs{f_{\tilde{Y}}^{(\Delta)}(\omega,\vartheta_1)-f_{\tilde{Y}}^{(\Delta)}(\omega,\vartheta_2)}
&=\frac{1}{2\pi} \Bigg| \ \cancel{\Sigma}^{(\Delta)}_{\vartheta_1}B_{\vartheta_1}'\left(\sum_{j=0}^\infty e^{A_{\vartheta_1}\Delta j}e^{-ij\omega}\right)\left(\sum_{j=0}^\infty e^{A_{\vartheta_1}'\Delta j}e^{ij\omega}\right)B_{\vartheta_1}\notag\\
&\quad-\cancel{\Sigma}^{(\Delta)}_{\vartheta_2}B_{\vartheta_2}'\left(\sum_{j=0}^\infty e^{A_{\vartheta_2}\Delta j}e^{-ij\omega}\right)\left(\sum_{j=0}^\infty e^{A_{\vartheta_2}'\Delta j}e^{ij\omega}\right)B_{\vartheta_2}\Bigg|\notag\\
&\leq R_1+R_2+R_3+R_4+R_5,
\end{align}
where 
\begin{align*}
R_1&=\frac{1}{2\pi}\abs{\ \cancel{\Sigma}^{(\Delta)}_{\vartheta_1}- \cancel{\Sigma}^{(\Delta)}_{\vartheta_2}}\norm{B_{\vartheta_1}'}\sum_{j=0}^\infty \norm{e^{A_{\vartheta_1}\Delta j}}\sum_{j=0}^\infty \norm{e^{A_{\vartheta_1}'\Delta j}}\norm{B_{\vartheta_1}}\\
R_2&=\frac{1}{2\pi}\abs{\ \cancel{\Sigma}^{(\Delta)}_{\vartheta_2}} \norm{B_{\vartheta_1}'-B_{\vartheta_2}'}\sum_{j=0}^\infty \norm{e^{A_{\vartheta_1}\Delta j}}\sum_{j=0}^\infty \norm{e^{A_{\vartheta_1}'\Delta j}}\norm{B_{\vartheta_1}}\\
R_3&=\frac{1}{2\pi}\abs{\ \cancel{\Sigma}^{(\Delta)}_{\vartheta_2}} \norm{B_{\vartheta_2}'}\norm{\sum_{j=0}^\infty \left(e^{A_{\vartheta_1}\Delta j}-e^{A_{\vartheta_2}\Delta j}\right)e^{-ij\omega}}\sum_{j=0}^\infty \norm{e^{A_{\vartheta_1}'\Delta j}}\norm{B_{\vartheta_1}}\\
R_4&=\frac{1}{2\pi}\abs{\ \cancel{\Sigma}^{(\Delta)}_{\vartheta_2}} \norm{B_{\vartheta_2}'}\sum_{j=0}^\infty \norm{e^{A_{\vartheta_2}\Delta j}}\norm{\sum_{j=0}^\infty \left(e^{A_{\vartheta_1}'\Delta j}-e^{A_{\vartheta_2}'\Delta j}\right)e^{ij\omega}}\norm{B_{\vartheta_1}}\\
R_5&=\frac{1}{2\pi}\abs{\ \cancel{\Sigma}^{(\Delta)}_{\vartheta_2}} \norm{B_{\vartheta_2}'}\sum_{j=0}^\infty \norm{e^{A_{\vartheta_2}\Delta j}}\sum_{j=0}^\infty \norm{e^{A_{\vartheta_2}'\Delta j}}\norm{B_{\vartheta_1}-B_{\vartheta_2}}.
\end{align*}
It is sufficient to show that each summand $R_1$-$R_5$ can be bounded by $C_3\norm{\vartheta_1-\vartheta_2}$ for some constant $C_3>0$. To avoid repeating calculations we just give the main ideas.\\
Since $\Theta$ is compact and \hyperref[assumption:C4]{(C4)} holds, the functions $\vartheta\mapsto B_\vartheta$, $\vartheta\mapsto A_\vartheta$ are bounded and Lipschitz with constants $L_A$ and $L_B$. Moreover, from \hyperref[assumption:C2]{(C2)} we can follow $\lVert e^{A_\vartheta t}\rVert\leq De^{-\alpha t}$ for some constants $D,\alpha>0$ and all $t\in\R_0^+$ such that
\begin{gather*}
\norm{\sum_{j=0}^\infty e^{A_{\vartheta}\Delta j}}\leq \sum_{j=0}^\infty \norm{ e^{A_{\vartheta}\Delta j}}\leq D\sum_{j=0}^\infty e^{-\alpha \Delta j}<C_1, \text{ as well as }\norm{\sum_{j=0}^\infty e^{A_{\vartheta}'\Delta j}}<C_1
\end{gather*}
for some $C_1>0$. Note that since the eigenvalues of a matrix are continuous functions of its entries the above bounds hold uniformly in $\vartheta$. In addition, due to \cite[page 238]{H2008}, we have for some constants $C_2,\tilde{\alpha}>0$ and all $j\in\Z$
\begin{align*}
\norm{ e^{A_{\vartheta_1} \Delta j}- e^{A_{\vartheta_2} \Delta j}}&
\leq\norm{A_{\vartheta_1}-A_{\vartheta_2}}\Delta j \int_{0}^1 \norm{e^{\nu A_{\vartheta_1} \Delta j}} \norm{e^{(1-\nu)A_{\vartheta_2}\Delta j}}d\nu\\
&\leq C_2 L_A\Delta j \norm{\vartheta_1-\vartheta_2} e^{-\tilde{\alpha} \Delta j}.
\end{align*}
Thus, the functions $\vartheta\mapsto e^{A_\vartheta \Delta j}$ and $\vartheta\mapsto e^{A_\vartheta' \Delta j}$ are Lipschitz. Next, the representation (\ref{eq:variancesampledprocess}) of $\cancel{\Sigma}^{(\Delta)}_{\vartheta}$ implies that $\vartheta\mapsto \cancel{\Sigma}^{(\Delta)}_{\vartheta}$ is bounded. By using a similar decomposition as in (\ref{eq:decompositionf_Y}) for $\cancel{\Sigma}^{(\Delta)}_{\vartheta}$ one can show that there exists a constant $L_{\cancel{\Sigma}}>0$, such that
\begin{gather*}
\abs{\ \cancel{\Sigma}^{\ (\Delta)}_{\vartheta_1}-\cancel{\ \Sigma}^{\ (\Delta)}_{\vartheta_2}}\leq L_{\cancel{\Sigma}}\norm{\vartheta_1-\vartheta_2} \text{ for all } \vartheta_1,\vartheta_2\in\Theta. 
\end{gather*}
Overall, each summand $R_1$-$R_5$ in (\ref{eq:decompositionf_Y}) can be bounded by $C_3 \norm{\vartheta_1-\vartheta_2}$ for some constant $C_3$, which in turn implies that (\ref{eq:f_Yislipschitz}) holds.
\end{proof}

\begin{proof}[Proof of Theorem \ref{theorem:whittleconsistent}]
In order to show consistency of $\hat{\vartheta}_N$ we follow the same steps as in the proof of Theorem \ref{theorem:consistency}. As limiting function we consider
\begin{gather*}
W(\vartheta)=\frac{1}{2\pi} \int_{-\pi}^\pi \frac{f_{\tilde{Y}}^{(\Delta)}(\omega,\vartheta^*(u))}{f_{\tilde{Y}}^{(\Delta)}(\omega,\vartheta)}+\log\left(f_{\tilde{Y}}^{(\Delta)}(\omega,\vartheta)\right)d\omega,\qquad \vartheta\in\Theta.
\end{gather*}
Then, \cite[Lemma 1]{FHM2020} yields that 
\begin{align*}
W(\vartheta)=L(\vartheta)=-\log(2\pi)+\log(V_\vartheta)+\frac{E\left[\varepsilon_{\vartheta}^2\left(\tilde{Y}^{\vartheta^*(u),\Delta} \right)\right]}{V_\vartheta},
\end{align*}
where $V_\vartheta$, $\tilde{Y}^{\vartheta^*(u),\Delta}$ and $\varepsilon_{\vartheta}$ are defined as in Section \ref{sec5-2}. Note that $L(\vartheta)$ is closely related to the limiting function of the quasi maximum likelihood estimator from Section \ref{sec5-2} such that Proposition \ref{proposition:qmlecontrast} along with the inequalities in (\ref{eq:continuityMforlater}) ensure that $W$ is continuous. Moreover, Proposition \ref{proposition:qmlem1} ensures that $W(\vartheta)$ has a unique minimum. Therefore, it is left to show $\norm{W_N(\vartheta)-W(\vartheta))}_{L^1}\underset{N\rightarrow\infty}{\longrightarrow}0$ and that the sequence $(W_N(\vartheta))_{N\in\N}$ is stochastically equicontinuous. Indeed,
\begin{align*}
\norm{W_N(\vartheta)-W(\vartheta))}_{L^1}\leq&\abs{\frac{1}{4m_N+2}\sum_{j=-2m_N}^{2m_N+1} \log\left(f_{\tilde{Y}}^{(\Delta)}(\omega_j,\vartheta)\right) -\frac{1}{2\pi}\int_{-\pi}^{\pi}\log\left(f_{\tilde{Y}}^{(\Delta)}(\omega,\vartheta)\right)d\omega}\\
&+ \norm{\frac{1}{4m_N+2}\sum_{j=-2m_N}^{2m_N+1}\frac{I_N^{loc}(\omega_j)}{f_{\tilde{Y}}^{(\Delta)}(\omega_j,\vartheta)}-\frac{1}{2\pi}\int_{-\pi}^{\pi}\frac{f_{\tilde{Y}}^{(\Delta)}(\omega,\vartheta^*(u))}{f_{\tilde{Y}}^{(\Delta)}(\omega,\vartheta)}d\omega}_{L^1}\\
&=:P_1+P_2.
\end{align*}
From \cite[equation (15)]{FHM2020} we obtain $P_1\longrightarrow0$ as $N\rightarrow\infty$. To show $P_2\longrightarrow0$ as $N\rightarrow\infty$ we follow the approach from \cite[Proposition 2]{FHM2020} and approximate $f_{\tilde{Y}}^{(\Delta)}(\omega,\vartheta)^{-1}$ by the Ces\`aro mean of its Fourier series of size $M$ for $M$ sufficiently large. Define 
\begin{align*}
q_M(\omega,\vartheta)&=\frac{1}{M}\sum_{j=0}^{M-1}\sum_{\abs{k}\leq j}b_k(\vartheta)e^{-ik\omega}=\sum_{\abs{k}< M}\left(1-\frac{\abs{k}}{M} \right)b_k(\vartheta)e^{-ik\omega}, where\\
b_k(\vartheta)&=\frac{1}{2\pi}\int_{-\pi}^{\pi}\frac{1}{f_{\tilde{Y}}^{(\Delta)}(\omega,\vartheta)}e^{ik\omega}d\omega.
\end{align*}
Similar arguments as in the proof of \cite[Proposition 2]{FHM2020} yield that for all $\varepsilon>0$ and $M$ sufficiently large
\begin{gather}\label{eq:q_Mvsf}
\sup_{\omega\in[-\pi,\pi]}\sup_{\vartheta\in\Theta} \abs{q_M(\omega,\vartheta)-\frac{1}{f_{\tilde{Y}}^{(\Delta)}(\omega,\vartheta)}}<\varepsilon.
\end{gather}
From the definition of $I_N^{loc}$ in (\ref{eq:localperiodogram}) it is clear that it is non-negative, such that
\begin{align*}
&\norm{\frac{1}{4m_N+2}\sum_{j=-2m_N}^{2m_N+1}\frac{I_N^{loc}(\omega_j)}{f_{\tilde{Y}}^{(\Delta)}(\omega_j,\vartheta)}-\frac{1}{4m_N+2}\sum_{j=-2m_N}^{2m_N+1}q_M(\omega_j,\vartheta) I_N^{loc}(\omega_j) }_{L^1}\\ 
& \leq \frac{\varepsilon}{4m_N+2} E\left[ \sum_{j=-2m_N}^{2m_N+1}I_N^{loc}(\omega_j)\right]= \frac{\varepsilon}{4m_N+2} E\left[ \sum_{j=-2m_N}^{2m_N+1} \frac{1}{2\pi} \sum_{h=-2m_N}^{2m_N}\hat{\Gamma}_N^{loc}(h)e^{-ih\omega_j}  \right]
\\
&= \frac{\varepsilon}{2\pi} E\left[\hat{\Gamma}_N^{loc}(0) \right],
\end{align*}
where the last equality follows from the identity (see \cite[Lemma 4]{FHM2020})
\begin{gather}\label{eq:fourierseriesisindicator}
\frac{1}{4m_N+2}\sum_{j=-2m_N}^{2m_N+1}e^{-ih\omega_j}=
\begin{cases}1\quad , \text{there exists } z\in\Z: h=z(4m_N+2),\\
0\quad, \text{else}.
\end{cases}
\end{gather}
Now, from Lemma \ref{lemma:convergencegamma} it follows that
\begin{gather*}
E\left[\hat{\Gamma}_N^{loc}(0) \right] \underset{N\rightarrow\infty}{\longrightarrow} E[\tilde{Y}_u(0)^2]<\infty,
\end{gather*}
such that overall
\begin{gather*}
\norm{\frac{1}{4m_N+2}\sum_{j=-2m_N}^{2m_N+1}\frac{I_N^{loc}(\omega_j)}{f_{\tilde{Y}}^{(\Delta)}(\omega_j,\vartheta)}-\frac{1}{4m_N+2}\sum_{j=-2m_N}^{2m_N+1}q_M(\omega_j,\vartheta) I_N^{loc}(\omega_j) }_{L^1}\leq C_1 \varepsilon,
\end{gather*}
for some constant $C_1>0$. In view of $P_2$ it is therefore enough to show
\begin{align*}
\norm{\frac{1}{4m_N+2}\sum_{j=-2m_N}^{2m_N+1}q_M(\omega_j,\vartheta) I_N^{loc}(\omega_j) -\frac{1}{2\pi}\int_{-\pi}^{\pi}\frac{f_{\tilde{Y}}^{(\Delta)}(\omega,\vartheta^*(u))}{f_{\tilde{Y}}^{(\Delta)}(\omega,\vartheta)}d\omega}_{L^1} \underset{N\rightarrow\infty}{\longrightarrow}0.
\end{align*}
On the one hand, using (\ref{eq:fourierseriesisindicator}), we obtain 
\begin{align*}
&\frac{1}{4m_N+2}\sum_{j=-2m_N}^{2m_N+1}q_M(\omega_j,\vartheta) I_N^{loc}(\omega_j)\\
&=\frac{1}{2\pi} \sum_{\abs{k}< M}\sum_{\abs{h}\leq 2m_N} \left(1-\frac{\abs{k}}{M} \right)b_k(\vartheta)\hat{\Gamma}_N^{loc}(h) \left( \frac{1}{4m_N+2}\sum_{j=-2m_N}^{2m_N+1}e^{-i(k+h)\omega_j}\right)\\
&= \frac{1}{2\pi} \sum_{\abs{k}< M} \left(1-\frac{\abs{k}}{M} \right)b_k(\vartheta)\hat{\Gamma}_N^{loc}(-k)
\overset{L^1}{\underset{N\rightarrow\infty}{\longrightarrow}} \frac{1}{2\pi} \sum_{\abs{k}< M} \left(1-\frac{\abs{k}}{M} \right)b_k(\vartheta) E[\tilde{Y}_u(0)\tilde{Y}_u(k)],
\end{align*}
where the convergence follows from Lemma \ref{lemma:convergencegamma} and 
\begin{gather*}
\sup_{k\in\Z}\sup_{\vartheta\in\Theta} \abs{b_k(\vartheta)}\leq \max_{\vartheta\in\Theta}\max_{\omega\in[-\pi,\pi]}\abs{f_Y^{(\Delta)}(\omega,\vartheta)^{-1}}<\infty, 
\end{gather*}
since $(\omega,\vartheta)\mapsto f_{\tilde{Y}}^{(\Delta)}(\omega,\vartheta)^{-1}$ is continuous (see e.g. the proof of \cite[Proposition 2]{FHM2020}). On the other hand we obtain from (\ref{eq:q_Mvsf})
\begin{align*}
&\abs{ \frac{1}{2\pi} \sum_{\abs{k}< M} \left(1-\frac{\abs{k}}{M} \right)b_k(\vartheta) E[\tilde{Y}_u(0)\tilde{Y}_u(k)]-\frac{1}{2\pi}\int_{-\pi}^{\pi}\frac{f_{\tilde{Y}}^{(\Delta)}(\omega,\vartheta^*(u))}{f_{\tilde{Y}}^{(\Delta)}(\omega,\vartheta)}d\omega}\\
&=\abs{ \frac{1}{2\pi} \sum_{\abs{k}< M} \left(1-\frac{\abs{k}}{M} \right)b_k(\vartheta) \int_{-\pi}^\pi f_{\tilde{Y}}^{(\Delta)}(\omega,\vartheta^*(u))e^{ik\omega}d\omega  -\frac{1}{2\pi}\int_{-\pi}^{\pi}\frac{f_{\tilde{Y}}^{(\Delta)}(\omega,\vartheta^*(u))}{f_{\tilde{Y}}^{(\Delta)}(\omega,\vartheta)}d\omega}\\
&=\abs{  \frac{1}{2\pi}\int_{-\pi}^\pi f_{\tilde{Y}}^{(\Delta)}(\omega,\vartheta^*(u)) \left(q_M(\omega,\vartheta)-\frac{1}{f_{\tilde{Y}}^{(\Delta)}(\omega,\vartheta)} \right)d\omega}\leq C_2 \varepsilon,
\end{align*}
for some constant $C_2>0$, since $\omega\mapsto f_{\tilde{Y}}^{(\Delta)}(\omega,\vartheta^*(u))$ is continuous. Overall, it holds $P_2\underset{N\rightarrow\infty}{\longrightarrow}0$.\\
It is left to show stochastic equicontinuity of the sequence $(W_N(\vartheta))_{N\in\N}$. For $\eta,\lambda>0$ we have
\begin{align*}
&P\left( \sup_{\norm{\vartheta_1-\vartheta_2}<\lambda}\left|W_N(\vartheta_1)-W_N(\vartheta_2) \right|>\eta\right)\\
&\leq P\left(\sup_{\norm{\vartheta_1-\vartheta_2}\leq\lambda}\abs{\frac{1}{4m_N+2}\sum_{j=-2m_N}^{2m_N+1}I_N^{loc}(\omega_j) \left(\frac{1}{f_{\tilde{Y}}^{(\Delta)}(\omega_j,\vartheta_1)}-\frac{1}{f_{\tilde{Y}}^{(\Delta)}(\omega_j,\vartheta_2)}\right)} >\frac{\eta}{2}\right)\\
&+P\left(\sup_{\norm{\vartheta_1-\vartheta_2}\leq\lambda}\abs{\frac{1}{4m_N+2}\sum_{j=-2m_N}^{2m_N+1}\left(\log\left(f_{\tilde{Y}}^{(\Delta)}(\omega_j,\vartheta_1)\right)-\log\left(f_{\tilde{Y}}^{(\Delta)}(\omega_j,\vartheta_2)\right)\right) }>\frac{\eta}{2} \right)=:T_1+T_2.
\end{align*}
From Lemma \ref{lemma:f_Yislipschitz} it is clear, that 
\begin{align}
\sup_{\omega\in[-\pi,\pi]}\abs{\frac{1}{f_{\tilde{Y}}^{(\Delta)}(\omega,\vartheta_1)}-\frac{1}{f_{\tilde{Y}}^{(\Delta)}(\omega,\vartheta_2)}}&\leq \frac{L_f}{f_{inf}^2}\norm{\vartheta_1-\vartheta_2}\text{ as well as}\label{eq:1/xlipschitz}\\
\sup_{\omega\in[-\pi,\pi]}\abs{\log\left(f_{\tilde{Y}}^{(\Delta)}(\omega,\vartheta_1)\right)-\log\left(f_{\tilde{Y}}^{(\Delta)}(\omega,\vartheta_2)\right)}&\leq \frac{L_f}{f_{inf}}\norm{\vartheta_1-\vartheta_2}\label{eq:loglipschitz}.
\end{align}
We set  
\begin{gather*}
\lambda=\min\left(\frac{\eta\pi f_{inf}^2}{2L_fE[\tilde{Y}_u(0)^2]},\frac{\eta f_{inf}}{4L_f}\right).
\end{gather*}
Since $I_N^{loc}$ is non-negative we can follow from (\ref{eq:1/xlipschitz}) that 
\begin{gather*}
T_1\leq P\left(\frac{\lambda L_f}{f_{inf}^2}\abs{\frac{1}{4m_N+2}\sum_{j=-2m_N}^{2m_N+1}I_N^{loc}(\omega_j)}>\frac{\eta}{2} \right)=P\left(\frac{\lambda L_f}{2\pi f_{inf}^2} \hat{\Gamma}_N^{loc}(0) >\frac{\eta}{2} \right).
\end{gather*}
From Lemma \ref{lemma:convergencegamma} we have $\hat{\Gamma}_N^{loc}(0)\overset{P}{\underset{N\rightarrow\infty}{\longrightarrow}}E[\tilde{Y}_u(0)^2]$, such that
\begin{gather*}
P\left(\frac{\lambda L_f}{2\pi f_{inf}^2} \hat{\Gamma}_N^{loc}(0) >\frac{\eta}{2} \right)\leq P\left(\abs{ \hat{\Gamma}_N^{loc}(0)-E[\tilde{Y}_u(0)^2]}>E[\tilde{Y}_u(0)^2] \right)\underset{N\rightarrow\infty}{\longrightarrow}0.
\end{gather*}
Using (\ref{eq:loglipschitz}) it holds for $T_2$
\begin{gather*}
T_2\leq P\left(\frac{\lambda L_f}{ f_{inf}}>\frac{\eta}{2} \right)\leq P\left(\frac{1}{2}>1 \right)=0,
\end{gather*}
which concludes the proof.
\end{proof}

\section*{Acknowledgements}
The author would like to thank Robert Stelzer for several inspiring discussions and careful reading. Moreover, the author thanks Vicky Fasen-Hartmann and Celeste Mayer for providing their R-code from \cite{FHM2020} and Imma Valentina Curato for several helpful comments. This research was supported by the scholarship program of the Hanns-Seidel Foundation, funded by the Federal Ministry of Education and Research and by the state of Baden-Württemberg through bwHPC.

\bibliographystyle{acm}

\end{document}